%% file: main.tex
\documentclass[11pt]{article}

\usepackage[a4paper,margin=25mm]{geometry}
\usepackage[T1]{fontenc}
\usepackage[utf8]{inputenc}
\usepackage{lmodern}
\usepackage{microtype}
\usepackage{amsmath,amssymb,amsthm,mathtools}
\usepackage{bm}
\usepackage{booktabs,tabularx,array,multirow}
\usepackage{graphicx,float}
\usepackage{placeins}
\usepackage{xcolor}
\usepackage{enumitem}
\usepackage[authoryear,round]{natbib}
\usepackage[hidelinks]{hyperref}
\usepackage[nameinlink,noabbrev]{cleveref}
\usepackage{authblk}
\usepackage{caption}
\theoremstyle{plain}
\allowdisplaybreaks[1]
\newcommand{\doi}[1]{\url{https://doi.org/#1}}

\hypersetup{
  pdftitle={MIACE and Monotone Tracking: Exact Tuning for First-Order Plants and Positive Cascades with Pure Delay},
  pdfauthor={Latchezar P. Tomov},
  pdfsubject={MIACE optimality, monotone tracking, exact pure delay, positive cascades, and a nonlinear first-order extension},
  pdfkeywords={MIACE, monotone tracking, positive cascades, pure delay, PI control, filtered PID control, nonlinear control, global optimization}
}

\graphicspath{{figures/}}
\setlist{nosep,leftmargin=*}

\newtheorem{theorem}{Theorem}
\newtheorem{proposition}[theorem]{Proposition}
\newtheorem{corollary}[theorem]{Corollary}
\newtheorem{lemma}[theorem]{Lemma}
\theoremstyle{remark}

\AddToHook{env/theorem/begin}{\crefalias{theorem}{theorem}}
\AddToHook{env/proposition/begin}{\crefalias{theorem}{proposition}}
\AddToHook{env/corollary/begin}{\crefalias{theorem}{corollary}}
\AddToHook{env/lemma/begin}{\crefalias{theorem}{lemma}}
\AddToHook{env/remark/begin}{\crefalias{theorem}{remark}}

\newcommand{\Ecal}{\mathcal{E}}
\newcommand{\Ucal}{\mathcal{U}}
\newcommand{\Mcal}{\mathcal{M}}
\newcommand{\Ncal}{\mathcal{N}}
\newcommand{\Rset}{\mathbb{R}}
\newcommand{\iu}{\mathrm{i}}
\newcommand{\dd}{\,\mathrm{d}}
\newcommand{\norm}[1]{\left\lVert #1\right\rVert}

\title{MIACE and Monotone Tracking:\\
Exact Tuning for First-Order Plants\\
and Positive Cascades with Pure Delay}

\author{Latchezar P. Tomov}
\affil{\small\itshape Department of Informatics, New Bulgarian University,\\
21 Montevideo Street, 1618 Sofia, Bulgaria\\
\normalfont Corresponding author: \href{mailto:lptomov@nbu.bg}{lptomov@nbu.bg}}
\date{}

\begin{document}
\maketitle

\noindent\textbf{ORCID:} \href{https://orcid.org/0000-0003-1902-6473}{0000-0003-1902-6473}

\begin{abstract}
The Module of Integral Absolute Complex Error (MIACE) was created by me 23 years ago. It gave very nice behaviour to the transient response of the controlled systems I tested. It combines absolute
tracking-error and equilibrium-centered command areas. For a stable
unit-gain plant with known pure delay, an exact identity separates
squared MIACE excess into a negative-area penalty and a weighted area imbalance. The resulting lower bound depends on the plant's first
low-frequency moment and the delay. Every attaining response of a
positive cascade is monotone, independently of controller structure.
For a first-order plant, output variation and peak overshoot also admit
quantitative bounds. Exact conditions identify the unique positive PI
controller attaining the delayed bound; filtered PID extends attainment
beyond this PI range. For cascades of arbitrary order, Hermite
interpolation and a finite-support delay remainder construct proper
rational controllers of the same order, with internal exponential
stability proved by positive convolution representations. An
exponential-moment characterization then closes the attainment question
when internal delay blocks are allowed: an ordinary command attains the
bound if and only if $\alpha>\mu+\delta-1$. An explicit internal-model
controller realizes every feasible weight and delay. A sharp amplitude-cap
condition determines the least command peak among all attaining responses
and the complete capped sign-preserving area front. A small-gain estimate
supplies an explicit MIACE deviation bound under delay error. These
larger-class results are distinguished from the rational PI/PID attainment
ranges. An exact minimum-phase counterexample delimits the cascade
monotonicity claim.
A storage balance also extends the exact MIACE identity and monotone
attainment to a nonlinear first-order class without delay.
Independent physical integrations retain delays
exactly; quadratic comparisons provide complementary benchmarks.
\end{abstract}


\input{sections/01_introduction}
\input{sections/02_problem}

\input{sections/06_normalized_miace}

\input{sections/06b_miace_shape}
\input{sections/06c_pure_delay}
\input{sections/06d_filtered_pid}
\input{sections/06e_arbitrary_order}
\input{sections/06f_complete_attainment}

\input{sections/03_exact_energies}
\input{sections/04_global_optimum}
\input{sections/04b_delayed_quadratic}
\input{sections/05_pareto_geometry}
\input{sections/07_actuator_design}
\input{sections/08_model_mismatch}
\input{sections/08_numerical_studies}

\FloatBarrier
\input{sections/08b_nonlinear_balance}
\input{sections/09_discussion}
\input{sections/10_conclusion}
\FloatBarrier
\input{declarations}

\input{sections/appendices}
\FloatBarrier
\bibliographystyle{plainnat}
{\small\interlinepenalty=10000\setlength{\bibsep}{2pt}\bibliography{references}}

\end{document}

%% file: sections/01_introduction.tex
\section{Introduction}
\label{sec:introduction}

When a regulator makes an output reach its target quickly, a large command
excursion may be the price of that speed. If the output then crosses the
target and returns, part of its motion has to be undone. Even a response
that stays below the target can move backwards before finally settling.
A tuning criterion based on integral errors should therefore be examined
in two ways: which controller minimizes it, and what a small value says
about the shape of the resulting transient.

The Module of Integral Absolute Complex Error (MIACE) combines the
absolute tracking-error area and the absolute area of a command excursion
in a weighted Euclidean norm. In the form studied here, the command is
measured from the equilibrium required by the new setpoint. This centering
is essential on an infinite horizon: integrating the absolute value of a
nonzero steady command would give infinity. The resulting effort term
measures the size and duration of the transient excursion. It does not
measure steady energy consumption or actuator wear.

The author first formulated MIACE in 2003 for his bachelor's diploma
work at the Technical University of Sofia \citep{Tomov2003} and implemented
it in the MATLAB/Simulink software PIDBul2003. The subsequent account of
PIDBul by \citet{GaripovTomov2006} records the early software development
and cites that diploma work as reference~3. The finite-horizon formula
later appears in the author's doctoral dissertation
\citep[Sec.~4.5.1, p.~155, Eq.~(24)]{Tomov2011}.
The 2003 date concerns the original formulation and implementation;
the dimensionless infinite-horizon problem, equality analysis, and
response-shape bounds developed below are separate results.

The delay-free first-order plant gives a setting in which the connection between
criterion and response shape can be established exactly. Its normalized
balance is $y'=e+v$, where $e$ is the tracking error and $v$ the centered
command. Thus the total output travel cannot exceed the sum of their
absolute areas. A settled unit-step response has net travel one; every
backward movement adds twice its size to the total travel. Combining these
two observations with the weighted norm gives a quantitative bound:
a MIACE value within $1\%$ of its exact global minimum permits at most
$0.5\%$ peak overshoot and at most the same total backward motion.
At the exact minimum the output is monotone. The proof uses the plant
balance and integrability, so it also covers admissible commands produced
by controller structures other than PI.

The bound is accompanied by an exact decomposition of the MIACE excess.
One term measures negative areas of error and centered command; the other
measures imbalance between the two objectives at the chosen weight.
For positive ideal PI controllers, the complete absolute-area front and
its equality cases identify the unique minimizing gains. On a finite
record, the terminal error remains in the shape estimate. A short record
with little accumulated cost cannot by itself establish that tracking
has been achieved.

A known pure delay changes the waiting time before the output can move,
while the integrator acts throughout that interval. The same area argument
can still be used after separating the waiting interval from the plant
motion. Its intercept becomes $1+L/\tau$, and the MIACE lower bound changes
by the same factor. The PI gain must be reduced accordingly. The resulting
formula is globally optimal only on an explicit delay range, obtained by
combining the area equality conditions with the classical positivity
threshold for a scalar delay equation. An equality argument then excludes
all other positive PI settings. The thresholds for scalar delay stability
and oscillation are established results
\citep{Hayes1950,Domoshnitsky2014,Benekas2020}; their role here is to decide
when the MIACE lower bound can be attained by the stipulated PI controller.

The same equality conditions also guide the extension to a filtered PID
controller. They first fix its integral gain and require cancellation of
the plant pole; neither condition is assumed in advance. The additional
filter and derivative parameters can then keep the centered command
nonnegative at delays where no positive PI controller attains the bound.
A nonnegative convolution representation proves this sign property,
positivity of the error, and internal exponential stability. One exact
construction works at equal weights beyond the PI boundary; another gives
an explicit additional delay interval for larger effort weights. These
are sufficient attainment results. They do not identify the largest
possible PID delay range, and the attaining PID is no longer unique.
The controller differentiates the full error through an independent
first-order filter, so its precise realization is part of the theorem.

The next step is a cascade of $N$ stable real lags with a pure delay.
The signed-area intercept becomes the sum of the normalized time constants
and the delay. The same MIACE decomposition remains valid, but the
first-order derivative identity no longer does. A state argument proves
that the equality signs still force every attaining cascade response to
be monotone. Attainment itself imposes a strict weight--delay inequality.
An exact interpolation remainder for the delay then gives a proper
rational controller of order $N$, with no delay block inside it.
Positivity of its convolution resolvent proves internal stability and
both signal signs. For each fixed cascade and a suitable weight, the
construction works for all sufficiently small delays. Equal time
constants allow a uniform statement for every $N\ge2$, with
$\alpha=2N-1$ and $0\le L/\tau\le1/4$.

Monotonic response design is not new. \citet{LinFang1997} distinguish
nonovershooting from monotone third-order responses, while
\citet{Darbha2003} studies synthesis with a nonnegative impulse response.
Transfer-function conditions for external positivity are developed by
\citet{DrummondTurner2023}, and \citet{TaghavianEtAl2023} use logarithmic
complete monotonicity in rational-system analysis and synthesis.
\citet{DuEtAl2024} give nonovershooting controller designs for high-order
plants with or without time delay. More recently, the preprint
\citet{Taghavian2026} examines feasibility, controller order, and decay
rate in rational monotonic tracking systems. These works establish
the context in which the present claims must be read. Our contribution centers around a novel optimisation criteria which forces monotone tracking and has wider potential spectrum of applications than this paper develops and even this area of research in control theory. The criteria is 23 years old and yet it hasn't been truly discovered by other researchers so far. We want to change this with our paper.

The specific question here is stronger than obtaining a monotone
response: can the response attain the universal MIACE lower bound with
the prescribed ratio of absolute areas, while retaining the exact
transport delay and a specified finite controller order? The positive
cascade construction answers that question on the stated domain.
Classical Hermite interpolation and Volterra resolvents provide its
mathematical tools; they are not claimed as new general methods.
Nor does plant order alone suffice. A minimum-phase example attains
the MIACE bound but moves backwards, and the familiar obstruction
from a real right-half-plane zero remains. These distinctions prevent
the cascade theorem from being read as a theorem for every order-$N$
plant or as a determination of the smallest possible controller order.

Allowing internal delay blocks closes a second, larger-class question.
For a positive cascade, a nonnegative centered command avoids overshoot
exactly when its exponential moment, scaled by the slowest lag, is at
most one. This makes the necessary weight--delay inequality sufficient
through an explicit internal-model filter. The same characterization
gives a sharp command-cap condition and the least possible peak among
all responses attaining MIACE equality. Internal model control itself
is established \citep{GarciaMorari1982,Morari1983}; the present result
selects its filter by the MIACE equality conditions. An explicit bound
for delay error further separates nominal attainment from robustness.
These results are proved in \cref{sec:complete_attainment} and do not
assert that a finite-order rational controller can attain the bound
throughout this larger range.

Integral-error tuning, control-activity penalties, and multiobjective
PI/PID design have a substantial literature
\citep{AstromHagglund1995,ODwyer2009,Skogestad2003,Panagopoulos1998,
Panagopoulos2002,Herreros2002,Tavakoli2007,ReynosoMeza2012,
ContrerasLeiva2015}.  MIACE combines the absolute tracking-error area and the equilibrium-centered
command-excursion area through a weighted Euclidean norm in objective space.
Its aggregation belongs to the standard framework of multiobjective
scalarization \citep{Miettinen1998}, while its components are $L^1$ signal
functionals: integration precedes squaring. This distinguishes MIACE from
quadratic integral criteria. The results below establish the exact bounds,
attainment conditions, and response-shape consequences of this particular
tracking--effort criterion.
Total variation is also an established measure of departure from a
monotone response \citep{HubaEtAl2021}. The question answered here is the
precise relation between that departure, the attained MIACE value, and
the dynamics of the specified plant.

Quadratic tracking--effort tuning provides a useful comparison. The
stationary PI setting is recorded in the dissertation
\citep[Sec.~2.1, pp.~25--28, Eqs.~(10)--(11)]{Tomov2011}.
Its signal-level lower bound is the classical scalar linear-quadratic
regulator (LQR) result. LQR-derived PI/PID constructions are already
established \citep{HeWangLee2000}; constrained servo and regulatory
formulations with manipulated-variable-rate penalties are studied by
\citet{ThuLee2013} and \citet{TchamnaLee2017}. Here the quadratic result
is developed after MIACE, with exact parameter gaps and a complete
energy front. It shows why the two criteria select the same efficient
PI family even though their penalties away from that family differ.

Command limits and model errors are examined for the delay-free plant. An
amplitude cap admits an exact solution over the complete nonovershooting
PI domain. A feasible oscillatory counterexample shows why that clipping
rule cannot be extended to the unrestricted domain. An exact quadratic
mismatch loss then separates small cost degradation from preservation of
monotonicity. These results retain the delay-free first-order plant assumptions;
the additional constrained searches are numerical candidates with
reproducibility checks, not proofs of a general constrained optimum.

A nonlinear first-order plant gives a further test of the criterion.
In $y'=a(y)(w-y)$, the positive response rate $a(y)$ depends on the
output. Integrating a storage balance gives the same signed-area
identity and exact MIACE decomposition, with an intercept determined
by $\int_0^1 a(s)^{-1}\dd s$. A static output feedback with equilibrium
feedforward attains the bound for every positive locally Lipschitz
$a$, and every attaining response is monotone. This result, proved in
\cref{sec:nonlinear_balance}, uses the original error and centered
command areas. It concerns the specified nonlinear class without delay;
the delayed linear-system results retain their earlier assumptions.

%% file: sections/02_problem.tex
\section{Plant balance, normalization, and the PI family}
\label{sec:problem}

Table~\ref{tab:notation} summarizes the notation. All signal formulas in the main
analysis refer to a positive unit step from a zero operating point. The plant
initially has no delay, its parameters are known, and the controller and
actuator are ideal until explicit command bounds are imposed.
Section~\ref{sec:pure_delay} introduces a known pure delay and states
the changed balance and stability conditions separately.
Sections~\ref{sec:filtered_pid} and~\ref{sec:order_n} specify filtered PID
and positive cascades, respectively.

\begin{table}[!htbp]
\centering
\caption{Notation and its meaning. A prime denotes differentiation with
respect to normalized time $\theta$; a dot denotes differentiation with
respect to physical time $t$.}
\label{tab:notation}
\small
\begin{tabularx}{\linewidth}{@{}lX@{}}
\toprule
Symbol & Meaning\\
\midrule
$k,\tau$ & Positive plant gain and time constant in $G(s)=k/(\tau s+1)$.\\
$K_c,T_i$ & Positive proportional gain and integral time of the ideal PI controller.\\
$L,\delta$ & Pure delay and its normalized value $\delta=L/\tau$; zero unless stated.\\
$a$ & Nonzero reference-step amplitude; $a=1$ unless specified.\\
$y,e,u$ & Plant output, tracking error $e=1-y$, and command for a unit step.\\
$u_\infty,\widetilde u$ & Equilibrium command $1/k$ and its deviation $u-1/k$.\\
$\theta,p$ & Normalized time $t/\tau$ and normalized Laplace variable $\tau s$.\\
$x,r$ & Dimensionless PI coordinates $x=kK_c$, $r=T_i/\tau$.\\
$w,v$ & Plant-scaled command $w=ku$ and centered command $v=w-1$.\\
$D(p)$ & Closed-loop polynomial $rp^2+r(1+x)p+x$.\\
$\widehat E,\widehat V$ & Laplace transforms of $e(\theta)$ and $v(\theta)$.\\
$A_e,A_u$ & Dimensionless absolute areas of $e$ and $v$.\\
$\alpha,\Mcal_\alpha$ & Positive effort weight and normalized criterion
$\sqrt{A_e^2+\alpha A_u^2}$.\\
$z$ & Signed error area $r/x$.\\
$\lambda_1,\lambda_2$ & Positive decay rates when the closed-loop poles are real.\\
$A,\mathcal C_{\rm mon}$ & Command cap $ku_{\max}$ and complete class of monotone unit-step responses.\\
$m(x),\Delta$ & Boundary of that class and pole discriminant $(1+x)^2-4x/r$.\\
$N_e,N_v,\Ncal$ & Negative-part areas of $e$ and $v$, and their sum.\\
$m_\alpha,D_\alpha$ & MIACE lower bound $\sqrt{\alpha/(1+\alpha)}$ and area imbalance $A_e-\alpha A_u$.\\
$\mathcal V_y,B_y,P_y$ & Total output variation, total backward motion, and peak overshoot.\\
$E_2,U_2$ & Physical-time integrals of $e^2$ and $(u-u_\infty)^2$.\\
$\Ecal,\Ucal$ & Dimensionless integrals of $e^2$ and $v^2$ over $\theta\ge0$.\\
$Q,R$ & Positive weights, with units that make the two terms in $J$ comparable.\\
$q,F_q,s_q$ & Weight ratio $Qk^2/R$, objective $q\Ecal+\Ucal$, and $\sqrt{1+q}$.\\
$\kappa,\eta$ & Actual-to-nominal plant ratios $k/k_0$ and $\tau/\tau_0$.\\
$H,h$ & Physical and normalized finite horizons, with $h=H/\tau$.\\
$p_c,i_c,\beta$ & Parallel gains, with $p_c=x$, $i_c=x/r$ for PI, and quadratic weight $\beta=q$.\\
$d_c,\rho,\ell$ & Normalized derivative gain, filter time, and lead parameter of the filtered PID in \cref{sec:filtered_pid}.\\
\bottomrule
\end{tabularx}
\end{table}

With negative unity feedback and zero initial conditions, let
\begin{equation}
 G(s)=\frac{k}{\tau s+1},\qquad
 C(s)=K_c\left(1+\frac{1}{T_i s}\right),
 \qquad k,\tau,K_c,T_i>0.
 \label{eq:plant_controller}
\end{equation}
The corresponding time-domain equations are
\begin{equation}
 \tau\dot y+y=ku,\qquad
 u=K_c e+\frac{K_c}{T_i}\int_0^t e(\xi)\dd\xi,
 \qquad e=1-y.
 \label{eq:physical_loop}
\end{equation}
Introduce
\begin{equation}
 \theta=\frac t\tau,\qquad x=kK_c,\qquad r=\frac{T_i}{\tau},
 \qquad w=ku.
 \label{eq:normalization}
\end{equation}
Changing variables in the integral in \eqref{eq:physical_loop} gives
\begin{equation}
 y'+y=w,\qquad
 w=xe+\frac xr\int_0^\theta e(\eta)\dd\eta.
 \label{eq:normalized_loop}
\end{equation}
The integral state is initially zero. Thus $e(0^+)=1$ and the proportional
action gives $w(0^+)=x$. The plant state is continuous at the step, although
its input jumps. To obtain the error equation explicitly, substitute
$y=1-e$ into the plant equation and differentiate the PI law:
\[
 -e'+1-e=w,\qquad w'=xe'+\frac xr e.
\]
Differentiating the first relation gives $-e''-e'=w'$; inserting the second
and moving every term to the left yields
\begin{equation}
 e''+(1+x)e'+\frac xr e=0,
 \qquad e(0^+)=1,\quad e'(0^+)=-x.
 \label{eq:error_ode}
\end{equation}
Its characteristic polynomial is
\begin{equation}
 D(p)=rp^2+r(1+x)p+x.
 \label{eq:characteristic}
\end{equation}
A real second-order polynomial with positive leading coefficient is Hurwitz
when its other two coefficients are positive. Thus every $x,r>0$ gives an
asymptotically stable closed loop. In particular, $e\to0$, $y\to1$, and
$u\to u_\infty=1/k$.
This is internal stability of the two-state plant--integrator realization,
whose state matrix is displayed in \cref{app:peak_oracle}. It is established
before any transfer-function factor is cancelled.

Define
\begin{equation}
 \widetilde u=u-1/k,\qquad v=k\widetilde u=w-1.
 \label{eq:centered_command}
\end{equation}
The plant equation now has the useful form
\begin{equation}
 \boxed{e'=-e-v.}
 \label{eq:error_balance}
\end{equation}
Both $e$ and $v$ decay exponentially for positive PI gains. This makes
the absolute areas used by MIACE finite. More general commands will be
allowed in the response-shape theorems, with their integrability and
endpoint assumptions stated explicitly.

For a step $a\ne0$, use the normalized signals $(a-y)/a$ and $ku/a-1$.
The same dimensionless plant balance applies. The physical scaling of the
absolute areas is given in \eqref{eq:physical_nmiace_scaling}; the
quadratic energies considered later acquire a factor $a^2$.

\subsection{Transforms needed for the PI equality cases}

Let $\widehat E(p)=\int_0^\infty e^{-p\theta}e(\theta)\dd\theta$.
The transforms of the first two derivatives are
$p\widehat E-1$ and $p^2\widehat E-p+x$, respectively. Therefore
\eqref{eq:error_ode} becomes
\[
 (p^2\widehat E-p+x)+(1+x)(p\widehat E-1)+(x/r)\widehat E=0.
\]
The constant terms reduce to $-p-1$, giving
\[
 \left[p^2+(1+x)p+x/r\right]\widehat E(p)=p+1.
\]
Also, \eqref{eq:error_balance} gives
$\widehat V=1-(p+1)\widehat E$. Thus
\begin{equation}
 \widehat E(p)=\frac{r(p+1)}{D(p)},\qquad
 \widehat V(p)=\frac{r(x-1)p+x-r}{D(p)}.
 \label{eq:error_control_transforms}
\end{equation}
The constant one in $\widehat V=1-(p+1)\widehat E$ comes from $e(0^+)=1$;
it cancels the leading term, so $\widehat V$ is strictly proper.

%% file: sections/06_normalized_miace.tex
\section{Absolute areas and normalized MIACE}
\label{sec:miace}

The practical motivation for MIACE is to account for tracking error and
transient command excursion together. Its original development in 2003
and later finite-horizon formulation are described in the introduction.
Here both signals are normalized before integration, and the horizon is
infinite. These choices specify the criterion for which the exact
statements below hold. Behind the historical name, ``Module of Integral
Absolute Complex Error'', the criterion is a weighted Euclidean norm of
two nonnegative objectives.

The original aim was ``minimum inertia'': faster tracking without the
overshoot that might accompany larger control amplitudes. The two
integral objectives were pictured as the real and imaginary parts of a
complex number, with a smaller distance to the origin representing a
smaller combined transient burden. The weight expressed the ``price'' in
command area that the designer was willing to pay for faster tracking.
Taking absolute values before integration prevented cancellation of
positive and negative excursions and was intended to suppress overshoot
and oscillations and favour monotone responses. The author developed this
design intuition through exploratory testing of thousands of different
linear models in PIDBUL2003. The results below give precise statements
under explicit plant and controller assumptions.

Define the dimensionless infinite-horizon areas
\begin{equation}
 A_e=\int_0^\infty|e|\dd\theta,\qquad
 A_u=\int_0^\infty|v|\dd\theta,
 \label{eq:absolute_areas}
\end{equation}
and the normalized criterion
\begin{equation}
 \Mcal_\alpha=\sqrt{A_e^2+\alpha A_u^2},\qquad \alpha>0.
 \label{eq:nmiace}
\end{equation}
The complex-plane picture is made explicit by
\[
 \zeta_\alpha=A_e+\iu\sqrt{\alpha}\,A_u,
 \qquad |\zeta_\alpha|=\Mcal_\alpha.
\]
Thus the coordinates are absolute areas computed separately before they
are combined. Equal-MIACE contours are circles in the weighted complex
plane and ellipses in the unscaled $(A_e,A_u)$ plane. The origin represents
zero tracking-error area and zero command-excursion area about the new
equilibrium. Here ``inertia'' is a design metaphor for the combined
transient burden.

For a physical step $a\ne0$, let
$\operatorname{IAE}=\int_0^\infty|a-y(t)|\dd t$ and
$\operatorname{IAU}=\int_0^\infty|u(t)-a/k|\dd t$. Then
\begin{equation}
 \boxed{A_e=\frac{\operatorname{IAE}}{|a|\tau},\qquad
 A_u=\frac{k\operatorname{IAU}}{|a|\tau}.}
 \label{eq:physical_nmiace_scaling}
\end{equation}
Both denominators contain the step magnitude. The two quantities, and hence
$\alpha$, are dimensionless; changing the step amplitude alone does not
change this tuning preference.

There is an exact interpretation of the line that will form the efficient
front. Define the negative-part areas
\[
 N_e=\int_0^\infty\max\{-e,0\}\dd\theta,\qquad
 N_v=\int_0^\infty\max\{-v,0\}\dd\theta.
\]
For an integrable signal $f$, the pointwise identity
$|f|=f+2\max\{-f,0\}$ gives its absolute area from its signed area.
The transforms at $p=0$ give $\int e=r/x$ and $\int v=1-r/x$.
Writing $z=r/x$ therefore yields
\begin{equation}
 A_e=z+2N_e,\quad A_u=1-z+2N_v,\quad
 \boxed{A_e+A_u-1=2(N_e+N_v).}
 \label{eq:negative_area_gap}
\end{equation}
For a positive step, $N_e$ is the area of output overshoot, and $N_v$ is
the area by which the command falls below its equilibrium level. Either
effect moves the objective pair strictly above the line $A_e+A_u=1$.
The identity also holds for $z>1$, when $N_v$ is necessarily positive.
The question is which PI controllers can attain the line. A controller is
Pareto efficient if no other admissible controller makes both areas no
larger and at least one strictly smaller.

\begin{theorem}[Complete $L^1$ front and normalized MIACE optimum]
\label{thm:miace}
The complete efficient positive PI controller set for $(A_e,A_u)$ is
\begin{equation}
 r=1,\qquad x\ge1,
 \label{eq:l1_controller_front}
\end{equation}
with front
\begin{equation}
 \boxed{(A_e,A_u)=\left(\frac1x,1-\frac1x\right),\qquad x\ge1.}
 \label{eq:l1_pareto_front}
\end{equation}
For every finite $\alpha>0$, the unique global minimizer is
\begin{equation}
 \boxed{r_\alpha^\star=1,\qquad x_\alpha^\star=1+\frac1\alpha,}
 \label{eq:nmiace_optimum}
\end{equation}
and
\begin{equation}
 (\Mcal_\alpha^\star)^2=\frac\alpha{1+\alpha}.
 \label{eq:nmiace_minimum}
\end{equation}
\end{theorem}
\begin{proof}
Evaluation of \eqref{eq:error_control_transforms} at $p=0$ gives
\begin{equation}
 \int_0^\infty e\dd\theta=\frac rx=:z>0,\qquad
 \int_0^\infty v\dd\theta=1-z.
 \label{eq:signed_areas}
\end{equation}
Taking absolute values after integration cannot increase either area:
$A_e\ge z$ and $A_u\ge|1-z|$.
For $0<z\le1$, the controller $(x,r)=(1/z,1)$ produces
\begin{equation}
 e(\theta)=e^{-x\theta},\qquad
 v(\theta)=(x-1)e^{-x\theta}
 \label{eq:front_signals}
\end{equation}
and attains $(A_e,A_u)=(z,1-z)$. For $z>1$, both lower bounds show
that $(1,1)$, with areas $(1,0)$, dominates the original controller.
Also $A_e+A_u\ge z+|1-z|\ge1$. Therefore none of the attainable
points $(z,1-z)$ can be dominated.

To identify all efficient controllers, we must also exclude a different PI
controller attaining the same point. For $0<z<1$, equality requires
$e\ge0$ and $v\ge0$ for all $\theta\ge0$. The explicit sign argument in
\cref{app:miace_uniqueness} proves that these conditions force $r=1$.
For $z=1$, equality gives $A_u=0$ and hence $v\equiv0$; the plant equation
and $e'(0^+)=-x$ then give $x=r=1$.

Finally,
\begin{align}
 \Mcal_\alpha^2&\ge z^2+\alpha(1-z)^2\notag\\
 &=(1+\alpha)\left(z-\frac\alpha{1+\alpha}\right)^2
      +\frac\alpha{1+\alpha}.
 \label{eq:nmiace_certificate}
\end{align}
The minimum is attained at $z=\alpha/(1+\alpha)\in(0,1)$, and the equality
conditions already proved give \eqref{eq:nmiace_optimum} uniquely.
\end{proof}

The $L^1$ front is a half-open line segment: $(1,0)$ is attained, while
$(0,1)$ is approached only as $x\to\infty$.
Figure~\ref{fig:miace_complex_geometry} connects this front to the original
complex-plane picture. The area bound $A_e+A_u\ge1$ makes the origin an
unattainable reference point in the present first-order problem. The
closest admissible point in the weighted plane corresponds to the area
pair
\[
 (A_e^\star,A_u^\star)
 =\left(\frac{\alpha}{1+\alpha},\frac{1}{1+\alpha}\right).
\]
In the unscaled plane this is the tangency point of an equal-MIACE
ellipse with the front.
Theorem~\ref{thm:miace} proves that this point is attained and identifies
its unique positive PI controller.

Along this efficient PI family, the geometric trade-off has an exact
time-domain interpretation. For a relative settling tolerance
$0<\varepsilon<1$, \eqref{eq:front_signals} gives the normalized settling time
\[
 h_\varepsilon=\frac{\log(1/\varepsilon)}{x_\alpha^\star}
 =\frac{\alpha}{1+\alpha}\log(1/\varepsilon),
 \qquad v(0^+)=x_\alpha^\star-1=\frac1\alpha.
\]
The physical settling time is $\tau h_\varepsilon$. Decreasing $\alpha$
therefore shortens settling time while increasing both the command area
and its initial excursion; all these optimal responses remain monotone.
The plots use $\alpha=1/4,1,4$, giving $x_\alpha^\star=5,2,5/4$,
respectively. This settling-time interpretation applies to the specified
PI family. As \cref{sec:miace_response_shape} shows, other response shapes
can have the same two areas and the same minimum MIACE.

\begin{figure}[!htbp]
\centering
\includegraphics[width=\linewidth]{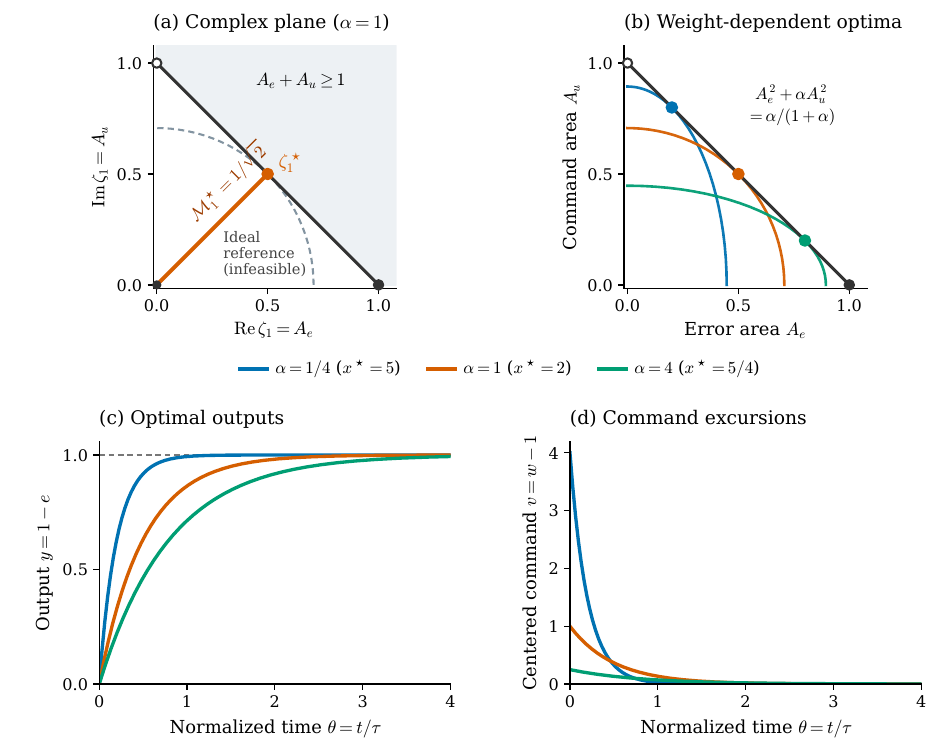}
\caption{The original MIACE picture and its exact first-order PI
realization. (a) At $\alpha=1$, MIACE is the distance from the origin in
the complex objective plane; shading indicates the necessary area bound
$A_e+A_u\ge1$.
(b) In unscaled area coordinates, equal-MIACE ellipses touch the front
at the optimal area pairs for three weights. (c,d) The corresponding
optimal output and equilibrium-centered command responses: smaller
$\alpha$ purchases faster monotone tracking with greater command
excursion. Open and filled endpoints denote limiting and attained
points of the front. The curves follow \cref{thm:miace} and
\eqref{eq:front_signals}.}

\label{fig:miace_complex_geometry}
\end{figure}

Absolute values also give the picture a direct connection to response
shape: \eqref{eq:negative_area_gap} measures the excess of total absolute
area above the efficient-front value caused by output overshoot or
command undershoot.
Theorems~\ref{thm:miace_shape_gap} and~\ref{thm:miace_variation} turn this
observation into an exact MIACE-excess identity and quantitative bounds
on backward output motion and peak overshoot under their stated
first-order assumptions.

For a finite horizon $H$, with $h=H/\tau$, the omitted normalized areas on
this front are
\begin{equation}
 A_{e,\infty}-A_{e,H}=\frac{e^{-xh}}x,\qquad
 A_{u,\infty}-A_{u,H}=\frac{x-1}{x}e^{-xh}.
 \label{eq:finite_horizon_tails}
\end{equation}
These expressions quantify truncation error for a fixed controller. They do
not identify the optimizer of a finite-horizon criterion, which need not
coincide with its infinite-horizon counterpart.

%% file: sections/06b_miace_shape.tex
\section{MIACE excess, monotonicity, and overshoot}
\label{sec:miace_response_shape}

\subsection{An exact penalty for sign reversals}
\label{sec:miace_shape}

The optimal PI response in \eqref{eq:front_signals} is monotone. A more
useful question for tuning is whether the value of MIACE also controls
how far a response can depart from that shape. This can be answered
without using the PI law. In this subsection retain the first-order plant,
but allow any command for which $e$ is locally absolutely continuous,
$e(0)=1$, $e(\infty)=0$, and $e,v\in L^1(0,\infty)$. The plant equation
$e'=-e-v$ holds almost everywhere. Every positive PI controller considered
above satisfies these assumptions, but the statements below also apply to
other controller structures producing such signals.

Write
\begin{equation}
 \Ncal=N_e+N_v,\qquad D_\alpha=A_e-\alpha A_u,\qquad
 m_\alpha=\sqrt{\frac{\alpha}{1+\alpha}}.
 \label{eq:miace_shape_quantities}
\end{equation}
Here $N_e$ measures the area by which the output exceeds its target,
and $N_v$ measures the area by which the command falls below its new
equilibrium. These two effects need not occur at the same time.

\begin{theorem}[Exact MIACE excess and shape of every minimizer]
\label{thm:miace_shape_gap}
For every response satisfying the assumptions above,
\begin{equation}
 \boxed{\Mcal_\alpha^2-m_\alpha^2
 =4m_\alpha^2\Ncal(1+\Ncal)+\frac{D_\alpha^2}{1+\alpha}.}
 \label{eq:miace_shape_gap}
\end{equation}
Consequently $\Mcal_\alpha\ge m_\alpha$. Equality holds precisely when
$e,v\ge0$ almost everywhere and $A_e=\alpha A_u$. Every response attaining
this bound has a nondecreasing output from zero to one and no overshoot.
The bound is attained by the PI controller in \eqref{eq:nmiace_optimum}.
\end{theorem}
\begin{proof}
Integrating the plant equation over $[0,h]$ gives
\[
 \int_0^h(e+v)\dd\theta=1-e(h).
\]
Pass to the limit using $e(h)\to0$ and absolute integrability. The
identity for the negative part of a real signal then gives
\begin{equation}
 A_e+A_u=\int_0^\infty(e+v)\dd\theta+2\Ncal=1+2\Ncal.
 \label{eq:general_area_balance}
\end{equation}
For any real $a,b$, direct expansion yields
\[
 \alpha(a+b)^2+(a-\alpha b)^2=(1+\alpha)(a^2+\alpha b^2).
\]
With $a=A_e,b=A_u$, this becomes
\begin{equation}
 \Mcal_\alpha^2=m_\alpha^2(A_e+A_u)^2
                  +\frac{D_\alpha^2}{1+\alpha}.
 \label{eq:miace_area_decomposition}
\end{equation}
Substitution of \eqref{eq:general_area_balance} proves
\eqref{eq:miace_shape_gap}.

Both terms on the right are nonnegative, and their sum is zero precisely
when $\Ncal=0,D_\alpha=0$. Vanishing integrals of the negative parts give
$e,v\ge0$ almost everywhere. Continuity gives $e\ge0$ everywhere, and
the plant equation gives $y'=e+v\ge0$ almost everywhere. Absolute
continuity therefore makes $y$ nondecreasing. Its initial and final values
are zero and one, so it cannot cross its target. Finally,
\eqref{eq:front_signals} with $x=1+1/\alpha$ attains the bound.
\end{proof}

The two terms have different meanings. The first is the exact contribution
of negative areas to the squared excess above the universal lower bound;
the second measures imbalance between the two objectives at the selected
weight. For small $\Ncal$, the first term is linear to leading order in $\Ncal$.
Geometrically, \eqref{eq:miace_area_decomposition} is Pythagoras in the
weighted complex plane of \cref{fig:miace_complex_geometry}: the signed
scalar components of $(A_e,\sqrt\alpha A_u)$ in unit directions normal
and parallel to the supporting line of the front are
$m_\alpha(A_e+A_u)$ and $D_\alpha/\sqrt{1+\alpha}$, respectively.
Even when the areas satisfy $D_\alpha=0$, any nonzero negative area still
increases the criterion. Monotonicity by itself is not sufficient for
optimality: the command condition and the weighted area balance must also
hold.

The controller-independent statement does not assert uniqueness outside
the positive PI family. For example, at $\alpha=1$, both the PI error
$e=e^{-2\theta}$ and the response
\[
 e=\tfrac14e^{-\theta}+\tfrac34e^{-3\theta},\qquad
 v=\tfrac32e^{-3\theta}
\]
satisfy the plant equation and have $A_e=A_u=1/2$. Both attain
$\Mcal_1=1/\sqrt2$ and both are monotone. The uniqueness in
\cref{thm:miace} concerns the two-parameter ideal PI family.

\subsection{A bound on total variation and peak overshoot}
\label{sec:miace_variation}

An output can reverse direction before reaching the setpoint, so its
overshoot alone does not measure all oscillatory motion. Define
\begin{equation}
 \mathcal V_y=\int_0^\infty|y'|\dd\theta,\qquad
 B_y=\int_0^\infty\max\{-y',0\}\dd\theta,\qquad
 P_y=\sup_{\theta\ge0}\max\{y(\theta)-1,0\}.
 \label{eq:output_variation}
\end{equation}
$B_y$ is the total distance travelled in the direction opposite to the
positive step. Total variation minus the net displacement is an established
measure of departure from monotonicity; for example,
\citet[Sec.~4, Eq.~(37)]{HubaEtAl2021} use its sampled-data counterpart
as a shape constraint. Here this quantity can be bounded directly by the
value of MIACE through the first-order plant equation.

\begin{theorem}[MIACE controls output variation and overshoot]
\label{thm:miace_variation}
Under the assumptions of \cref{thm:miace_shape_gap},
\begin{equation}
 \boxed{1\le\mathcal V_y=1+2B_y
       \le A_e+A_u=1+2\Ncal
       \le\frac{\Mcal_\alpha}{m_\alpha}.}
 \label{eq:miace_variation_bound}
\end{equation}
In particular,
\begin{equation}
 P_y\le B_y\le \Ncal\le
 \frac12\left(\frac{\Mcal_\alpha}{m_\alpha}-1\right).
 \label{eq:miace_peak_bound}
\end{equation}
If $\Mcal_\alpha\le(1+\varepsilon)m_\alpha$, with $\varepsilon\ge0$,
then
\begin{equation}
 \mathcal V_y-1\le\varepsilon,\qquad
 B_y\le\frac\varepsilon2,\qquad
 P_y\le\frac\varepsilon2,
 \qquad N_e+N_v\le\frac\varepsilon2.
 \label{eq:miace_near_optimal_shape}
\end{equation}
\end{theorem}
\begin{proof}
Because $y'=e+v$ and $e,v\in L^1$, the derivative is integrable and
\[
 \int_0^\infty y'\dd\theta=y(\infty)-y(0)=1.
\]
Apply $|f|=f+2\max\{-f,0\}$ to $f=y'$ to obtain
$\mathcal V_y=1+2B_y$. The triangle inequality gives
\[
 \mathcal V_y=\int_0^\infty|e+v|\dd\theta
 \le\int_0^\infty(|e|+|v|)\dd\theta=A_e+A_u.
\]
The last inequality in \eqref{eq:miace_variation_bound} follows from
\eqref{eq:miace_area_decomposition}. This proves the whole chain.

For every time $t$ at which $y(t)>1$, the subsequent net decrease is
\[
 y(t)-1=-\int_t^\infty y'(\theta)\dd\theta
 \le\int_t^\infty\max\{-y'(\theta),0\}\dd\theta\le B_y.
\]
Taking the supremum proves the peak bound. The remaining conclusions
follow by rearranging the chain and inserting the assumed MIACE value.
\end{proof}

For a positive unit step, a value within $1\%$ of the global lower bound
therefore guarantees peak overshoot at most $0.5\%$ of the step, and
bounds the sum of all backward output movements by the same amount.
For a step of magnitude $|a|$, the corresponding physical distances are
multiplied by $|a|$. The $1\%$ refers to $\Mcal_\alpha$, not its square;
if the squared criterion is within a factor $1+\varepsilon$, replace
$\varepsilon/2$ by $(\sqrt{1+\varepsilon}-1)/2$.
The comparison must use $m_\alpha$, not merely the best value found by
a numerical search. Extra constraints can make this bound unattainable.

These bounds control the accumulated size of reversals. They do not bound
the number or frequency of arbitrarily small oscillations. At the exact
minimum every reversal is excluded; close to it the allowed total reverse
motion is small. This is the precise sense in which the criterion favours
a well-damped response in the present model. It does not require a damping
ratio to be assigned to every response.

\begin{figure}[htbp]
\centering
\includegraphics[width=\linewidth]{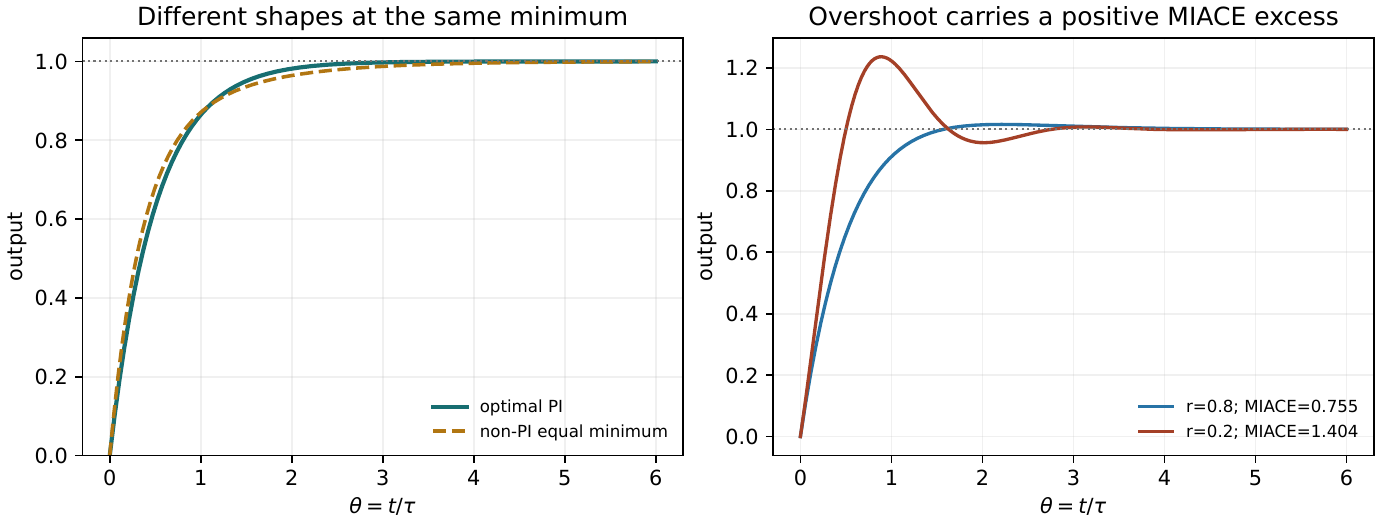}
\caption{Left: two different monotone responses attaining the same
MIACE lower bound. Right: PI responses with $x=2$ and shorter integral
times have overshoot and a strictly larger criterion.}

\label{fig:miace_shape_examples}
\end{figure}

\subsection{What remains true on a finite horizon}
\label{sec:miace_finite_shape}

Practical tuning evaluates a finite record. The endpoint must then remain
in the calculation. Append a subscript $h$ to areas, variation, and MIACE
when their integrals end at $h$, and put $e_h=e(h)$ and
$N_h=N_{e,h}+N_{v,h}$. The same two elementary expansions give
\begin{align}
 A_{e,h}+A_{u,h}&=1-e_h+2N_h,\label{eq:finite_area_balance}\\
 \Mcal_{\alpha,h}^2&=
 m_\alpha^2(1-e_h+2N_h)^2+
 \frac{(A_{e,h}-\alpha A_{u,h})^2}{1+\alpha}.
 \label{eq:finite_miace_identity}
\end{align}
Thus dropping the endpoint term would change the statement being proved.

\begin{proposition}[Finite-record overshoot bound]
\label{prop:finite_miace_shape}
For $P_{y,h}=\max_{0\le\theta\le h}(y(\theta)-1)_+$,
\begin{equation}
 P_{y,h}\le\frac12\left(
 \frac{\Mcal_{\alpha,h}}{m_\alpha}-1+|e_h|\right).
 \label{eq:finite_peak_bound}
\end{equation}
Consequently, $\Mcal_{\alpha,h}\le(1+\varepsilon)m_\alpha$ and
$|e_h|\le\varepsilon_e$ imply $P_{y,h}\le(\varepsilon+\varepsilon_e)/2$.
This is a bound on the observed interval, not on the unobserved tail.
\end{proposition}
\begin{proof}
Let $B_{y,h}=\int_0^h(-y')_+\dd\theta$. Its variation identity is
$\mathcal V_{y,h}=y(h)+2B_{y,h}$. Every earlier output value is at most
$y(h)+B_{y,h}$, because any subsequent decrease is included in
$B_{y,h}$. Therefore
\[
 P_{y,h}\le (y(h)-1)_++B_{y,h}
 =\frac{\mathcal V_{y,h}-1+|e_h|}{2}.
\]
The finite-interval triangle inequality and
\eqref{eq:miace_area_decomposition}, applied to the truncated areas,
give $\mathcal V_{y,h}\le\Mcal_{\alpha,h}/m_\alpha$.
These two inequalities prove \eqref{eq:finite_peak_bound}.
\end{proof}

A low truncated cost alone does not establish that the response has settled. For example,
holding $y=0,w=0$ throughout $[0,h]$ gives $e=1,v=-1$ and
$\Mcal_{\alpha,h}=h\sqrt{1+\alpha}$, which is arbitrarily small as
$h\downarrow0$ although no tracking has occurred. The endpoint remains
$e_h=1$. The example can be followed by a settling command after $h$;
it shows why the endpoint condition and a justified evaluation horizon
are necessary.

All shape conclusions above use the actual first-order plant balance and
an effort signal centred at the actual equilibrium. They are not a general
zero-overshoot theorem for higher-order or saturated PID loops. A known
pure delay requires the changed balance and attainment conditions derived
in \cref{sec:pure_delay}; the finite-record formulas above are for $L=0$.
The generalized signal statements do not establish internal stability of an
arbitrary controller architecture. Nor do they bound the variation of the command itself: rapid, small-area
command changes require a separate actuator or rate condition.

%% file: sections/06c_pure_delay.tex
\section{Pure delay: the area balance and an exact PI optimum}
\label{sec:pure_delay}

During a pure delay the controller already sees the new setpoint, but the
output cannot yet respond. The tracking error therefore accumulates before
any output motion takes place. At the same time the PI integrator continues
to increase the command. These two effects must both be included when
extending the preceding calculation. Shifting a delay-free response to the
right and retaining its old objective value would omit the first effect;
retaining its old PI gains would not in general produce that shifted response.

Replace the plant in \eqref{eq:plant_controller} by
\begin{equation}
 G_L(s)=\frac{k e^{-Ls}}{\tau s+1},\qquad
 \delta=\frac L\tau\ge0,
 \label{eq:delayed_plant}
\end{equation}
and retain the ideal PI law, the normalized variables, and the areas in
\eqref{eq:physical_nmiace_scaling}. The step is applied at time zero,
the plant and integrator start at zero, and $w(\theta)=0$ for $\theta<0$.
No command anticipates the step. The delay is known exactly; it is not
replaced by a rational approximation. All other sections use $L=0$
unless a delay is explicitly introduced.

\subsection{Separating the waiting interval from the plant motion}

The normalized plant balance becomes
\begin{equation}
 y'(\theta)+y(\theta)=w(\theta-\delta).
 \label{eq:delayed_balance}
\end{equation}
For $0\le\theta<\delta$, the right-hand side is zero. Hence $y=0$ and
$e=1$ throughout this interval. Define the shifted error
\begin{equation}
 f(\xi)=e(\xi+\delta),\qquad \xi\ge0.
 \label{eq:shifted_error}
\end{equation}
Substitution into \eqref{eq:delayed_balance}, at time $\xi+\delta$, gives
\begin{equation}
 f'=-f-v,\qquad f(0)=1,\qquad
 A_e=\delta+\int_0^\infty|f(\xi)|\dd\xi.
 \label{eq:delayed_shift_balance}
\end{equation}
The command in this equation is $v(\xi)=w(\xi)-1$, without a time shift.
This distinction avoids assigning a zero prehistory to the centered
command: its actual prehistory is $v=-1$.

For the signal-level results assume that $f$ is locally absolutely
continuous, $f(\infty)=0$, and $f,v\in L^1(0,\infty)$. Integrating
\eqref{eq:delayed_shift_balance} first on $[0,h]$ and then passing to
the limit gives $\int(f+v)=1$. Restoring the waiting interval yields
\begin{equation}
 \int_0^\infty e\dd\theta+\int_0^\infty v\dd\theta=1+\delta.
 \label{eq:delayed_signed_sum}
\end{equation}
Thus the unit intercept of the delay-free area balance increases by
exactly $\delta$. The extra term is an error area, not an approximation
to an additional plant time constant.

\subsection{MIACE excess and output motion}

Retain $N_e,N_v,\Ncal,D_\alpha,m_\alpha$ from \cref{sec:miace_shape}.
The waiting interval contributes no negative error area. Applying
$|g|=g+2(-g)_+$ to \eqref{eq:delayed_signed_sum} therefore gives
\begin{equation}
 A_e+A_u=1+\delta+2\Ncal.
 \label{eq:delayed_absolute_sum}
\end{equation}

\begin{theorem}[MIACE excess and shape with pure delay]
\label{thm:delayed_shape}
Under the signal assumptions above,
\begin{equation}
 \boxed{\Mcal_\alpha^2-(1+\delta)^2m_\alpha^2
 =4m_\alpha^2\Ncal(1+\delta+\Ncal)+\frac{D_\alpha^2}{1+\alpha}.}
 \label{eq:delayed_gap}
\end{equation}
In particular, $\Mcal_\alpha\ge(1+\delta)m_\alpha$. Equality holds
precisely when $e,v\ge0$ almost everywhere and $A_e=\alpha A_u$.
Every response attaining this lower bound is nondecreasing from zero
to one. Moreover,
\begin{equation}
 1\le\mathcal V_y=1+2B_y\le A_e+A_u-\delta=1+2\Ncal
 \le\frac{\Mcal_\alpha}{m_\alpha}-\delta,
 \label{eq:delayed_variation}
\end{equation}
and hence
\begin{equation}
 \boxed{P_y\le B_y\le \Ncal\le
 \frac12\left(\frac{\Mcal_\alpha}{m_\alpha}-1-\delta\right).}
 \label{eq:delayed_peak_bound}
\end{equation}
\end{theorem}
\begin{proof}
The weighted-square identity \eqref{eq:miace_area_decomposition},
which is algebraic and does not depend on the plant, gives
\[
 \Mcal_\alpha^2
 =m_\alpha^2(1+\delta+2\Ncal)^2+\frac{D_\alpha^2}{1+\alpha}.
\]
Subtracting $(1+\delta)^2m_\alpha^2$ and expanding the remaining
difference of squares proves \eqref{eq:delayed_gap}. Both terms on its
right are nonnegative and vanish together exactly under the stated
equality conditions. Before the delay the output is constant. Afterwards,
\eqref{eq:delayed_shift_balance} gives
$y'(\xi+\delta)=f(\xi)+v(\xi)\ge0$ at equality. Absolute continuity
and the initial and final values then establish monotonicity.

For a general admissible response, the derivative is integrable because
$f,v\in L^1$. Its signed integral is one, so
$\mathcal V_y=1+2B_y$. The triangle inequality must now be applied only
after the delay:
\[
 \mathcal V_y=\int_0^\infty|f+v|\dd\xi
 \le\int_0^\infty|f|\dd\xi+\int_0^\infty|v|\dd\xi
 =A_e-\delta+A_u.
\]
Combining this with \eqref{eq:delayed_absolute_sum} and
$A_e+A_u\le\Mcal_\alpha/m_\alpha$ proves
\eqref{eq:delayed_variation}. Any overshoot must subsequently be
undone before the output settles at one; the same integral argument as
in \cref{thm:miace_variation} gives $P_y\le B_y$. Rearranging the
chain proves \eqref{eq:delayed_peak_bound}.
\end{proof}

Consequently, a value satisfying
$\Mcal_\alpha\le(1+\varepsilon)(1+\delta)m_\alpha$ implies
$P_y,B_y\le(1+\delta)\varepsilon/2$. The comparison is with the
analytical lower bound. If a restricted controller class cannot attain
that bound, closeness to its own best value does not give the same
bounds in terms of relative excess. The next calculation determines exactly
when the positive ideal PI family can attain it.

\subsection{The cancellation family and its delay limits}

The error transform follows directly from the feedback interconnection.
Writing the step transform as $1/p$ and the loop transfer as
$x(rp+1)e^{-\delta p}/[rp(p+1)]$ gives
\begin{equation}
 \widehat E(p)=\frac{r(p+1)}{
 rp(p+1)+x(rp+1)e^{-\delta p}}.
 \label{eq:delayed_error_transform}
\end{equation}
Whenever the PI loop settles with absolutely integrable error and
centered command, evaluation at zero and
\eqref{eq:delayed_signed_sum} give
\begin{equation}
 \int_0^\infty e\dd\theta=\frac rx,
 \qquad\int_0^\infty v\dd\theta=1+\delta-\frac rx.
 \label{eq:delayed_pi_signed_areas}
\end{equation}
These are signed areas. They cannot be substituted for absolute areas
until their signs have been established.

At $r=1$ the full characteristic equation factorizes as
\begin{equation}
 (p+1)(p+xe^{-\delta p})=0,
 \qquad\widehat E(p)=\frac1{p+xe^{-\delta p}}.
 \label{eq:delayed_cancellation}
\end{equation}
The cancelled internal mode $p=-1$ is stable. The remaining modes belong
to the scalar delay equation
\begin{equation}
 e'(\theta)=-xe(\theta-\delta),\qquad
 e(0)=1,\quad e(\theta)=0\quad(\theta<0).
 \label{eq:delayed_error_dde}
\end{equation}
Here the zero error prehistory describes the time before the reference
step. For $0\le\theta\le\delta$, the solution is instead $e=1$.

\begin{lemma}[Stability and monotonicity on $r=1$]
\label{lem:delayed_cancellation}
For $x>0$, the internal PI loop at $r=1$ is asymptotically stable exactly
when $x\delta<\pi/2$. Its unit-step output is nondecreasing from zero
to one exactly when $x\delta\le1/\mathrm e$, where $\mathrm e$ is
Euler's number. Under the latter condition, $e>0$ at every finite
nonnegative time. If also $x\ge1$, then $v\ge0$.
\end{lemma}
\begin{proof}
For $\delta=0$, $e=e^{-x\theta}$ and the assertions follow immediately.
For $\delta>0$, rescale the characteristic roots by $z=\delta p$.
The remaining equation is $z+ae^{-z}=0$ with $a=x\delta>0$.
The classical scalar stability criterion gives $0<a<\pi/2$
\citep{Hayes1950}. The first boundary can also be checked directly:
putting $p=i\omega$ in \eqref{eq:delayed_cancellation} gives
$x\cos(\omega\delta)=0$ and
$\omega=x\sin(\omega\delta)$, whose first positive crossing is
$\omega=x$, $x\delta=\pi/2$. It is a crossing into the right half-plane,
since implicit differentiation with respect to $\delta$ gives
$\operatorname{Re}(\partial p/\partial\delta)
=x^2/[1+(x\delta)^2]>0$ there. The scalar criterion excludes an
earlier unstable root. The additional internal root $-1$ does not
alter stability.

For positivity, suppose $x\delta\le1/\mathrm e$. The function
$\lambda e^{-\lambda\delta}$ has maximum $1/(\mathrm e\delta)$
at $\lambda=1/\delta$. Thus there is a positive $\lambda$ with
$\lambda=xe^{\lambda\delta}$. Set
$h(\theta)=e^{\lambda\theta}e(\theta)$, with zero prehistory. Then
\[
 h'=\lambda[h-h(\theta-\delta)],\qquad
 h(\theta)=1+\lambda\int_{\max(0,\theta-\delta)}^\theta h(s)\dd s.
\]
If $h$ had a first zero, the last expression would be strictly positive
there. Therefore $h>0$, and hence $e>0$. Equation
\eqref{eq:delayed_error_dde} then gives $e'\le0$.
For the converse, the classical constant-coefficient oscillation
criterion states that every nontrivial real solution of
$g'(t)+xg(t-\delta)=0$ changes sign when $x\delta>1/\mathrm e$
\citep[Introduction, Eqs.~(17)--(20)]{Benekas2020}.
It applies to \eqref{eq:delayed_error_dde} from time $\delta$ onward,
with continuous history $e=1$ on $[0,\delta]$.
Thus a globally nonnegative error is impossible above this threshold.
The positivity boundary agrees with the scalar cases discussed by
\citet{Domoshnitsky2014}; it is not a new delay-equation criterion.

Finally, evaluating \eqref{eq:delayed_shift_balance} and
\eqref{eq:delayed_error_dde} at matching times gives
\begin{equation}
 v(\theta)=xe(\theta)-e(\theta+\delta).
 \label{eq:delayed_command_identity}
\end{equation}
Since $e$ is nonnegative and nonincreasing,
$v\ge(x-1)e\ge0$ when $x\ge1$.
\end{proof}

The two thresholds answer different questions. Between $1/\mathrm e$
and $\pi/2$ the output eventually settles but crosses its target.
At $x\delta=\pi/2$ the loop no longer decays asymptotically; the
infinite-horizon absolute areas used here are not finite.

\subsection{Attaining the lower bound and excluding other PI settings}

\begin{theorem}[Exact delayed PI optimum and its attainment range]
\label{thm:delayed_optimum}
Let $\alpha>0$ and define
\begin{equation}
 \delta_{\max}(\alpha)=
 \min\left\{\frac1\alpha,
 \frac{\alpha}{\mathrm e(\alpha+1)-\alpha}\right\}.
 \label{eq:delayed_attainment_range}
\end{equation}
A positive ideal PI controller attains the bound
$(1+\delta)m_\alpha$ if and only if
$0\le\delta\le\delta_{\max}(\alpha)$. In this range its unique
global minimizing setting is
\begin{equation}
 \boxed{r^\star=1,\qquad
 x^\star=\frac{1+1/\alpha}{1+\delta},\qquad
 \Mcal_\alpha^\star=(1+\delta)m_\alpha.}
 \label{eq:delayed_pi_optimum}
\end{equation}
The minimization is over all positive PI controllers with finite
infinite-horizon areas, without an imposed monotonicity constraint.
In physical units,
\begin{equation}
 \boxed{T_i^\star=\tau,\qquad
 K_c^\star=\frac{\tau(1+1/\alpha)}{k(\tau+L)}.}
 \label{eq:delayed_physical_optimum}
\end{equation}
\end{theorem}
\begin{proof}
First suppose $r=1$, $x\ge1$, and $x\delta\le1/\mathrm e$.
By \cref{lem:delayed_cancellation}, the loop is stable and $e,v\ge0$.
Consequently \eqref{eq:delayed_pi_signed_areas} becomes
$A_e=1/x$, $A_u=1+\delta-1/x$. The remaining equality condition
in \cref{thm:delayed_shape} is
\[
 \frac1x=\alpha\left(1+\delta-\frac1x\right).
\]
Solving gives $x=x^\star$ in \eqref{eq:delayed_pi_optimum}.
The requirement $x^\star\ge1$ is equivalent to $\delta\le1/\alpha$.
The second requirement is
\[
 \frac{(\alpha+1)\delta}{\alpha(1+\delta)}\le\frac1{\mathrm e},
\]
which is equivalent to the second bound in
\eqref{eq:delayed_attainment_range}. Its denominator is positive.
Thus the stated range is sufficient for attainment of the universal
lower bound and therefore for global optimality.

To prove uniqueness and necessity, allow an arbitrary $r>0$ and suppose
that its PI response attains the bound. Equality in
\eqref{eq:delayed_gap} requires $e,v\ge0$. For the shifted error this
gives $f'=-f-v\le-f$, $f(0)=1$, and hence
\begin{equation}
 0\le f(\xi)\le e^{-\xi}.
 \label{eq:delayed_equality_decay}
\end{equation}
The transform $\widehat E(p)$ is consequently analytic for
$\operatorname{Re}p>-1$. The identity obtained by multiplying
\eqref{eq:delayed_error_transform} by its denominator extends throughout
that half-plane, because it holds on the positive real half-plane and
both sides are analytic. In particular the displayed quotient applies
at real $p=-a$, $0<a<1$: its numerator is nonzero, so its denominator
cannot vanish there.

If $r\ne1$, that denominator tends, as $a\uparrow1$, to
$x(1-r)e^\delta\ne0$, while the numerator tends to zero. It would
follow that $\widehat E(-a)\to0$. On the other hand,
\[
 \widehat E(-a)=\int_0^\infty e^{a\theta}e(\theta)\dd\theta
 \ge\int_0^\infty e(\theta)\dd\theta=\frac rx>0,
\]
a contradiction. Therefore $r=1$. Since $v(0^+)=x-1$ and $v$ is
right-continuous, $v\ge0$ requires $x\ge1$. Nonnegativity of $e$
requires $x\delta\le1/\mathrm e$ by
\cref{lem:delayed_cancellation}. Finally $A_e=\alpha A_u$ fixes
$x=x^\star$. This proves necessity of the range as well as uniqueness.
\end{proof}

The theorem identifies when the lower bound is attained. It does not
describe the full delayed Pareto front or solve the remaining PI
optimization outside this range. In particular, failure of attainment
does not prove that every optimum outside the range must overshoot.
A monotone output may coexist with a command that falls below its
equilibrium level, making $N_v>0$ even when $N_e=0$.

\subsection{The initial command peak and the meaning of a smooth response}

The delayed optimum keeps the integral time but reduces the gain.
There is a useful consequence for the command amplitude. Before the
output starts moving, the PI law at $r=1$ gives
$w(\theta)=x(1+\theta)$ for $0\le\theta\le\delta$.
After this interval, differentiation gives
\[
 w'=x(e'+e)=x[e(\theta)-xe(\theta-\delta)]\le0
\]
when $x\ge1$ and the error is nonincreasing. Therefore the command has
one initial rise and then decreases to its equilibrium. At the optimum,
\begin{equation}
 \boxed{\sup_{\theta\ge0}w(\theta)
 =x^\star(1+\delta)=1+1/\alpha.}
 \label{eq:delayed_peak_command}
\end{equation}
This is the same peak as for the delay-free optimum. It is a property
of the unconstrained optimum in \cref{thm:delayed_optimum}, not a
solution of a delayed amplitude-constrained optimization problem.

For $\alpha=1$, the theorem applies up to
$\delta=1/(2\mathrm e-1)=0.225399674\ldots$.
The old setting $x=2,r=1$ loses monotonicity earlier, at
$\delta=1/(2\mathrm e)=0.183939721\ldots$, although it remains
stable for $\delta<\pi/4=0.785398163\ldots$.
At $\delta=0.2$, the corrected gain is $x^\star=5/3$ and both
areas are $0.6$. Thus $\Mcal_1^\star=1.2/\sqrt2$ and the
response is exactly monotone. The old gain gives a small overshoot,
visible in the enlarged view in \cref{fig:delay_miace_responses}.
Numerical areas and independent checks appear in
\cref{sec:delay_numerical_checks}.

\begin{figure}[htbp]
 \centering
 \includegraphics[width=\linewidth]{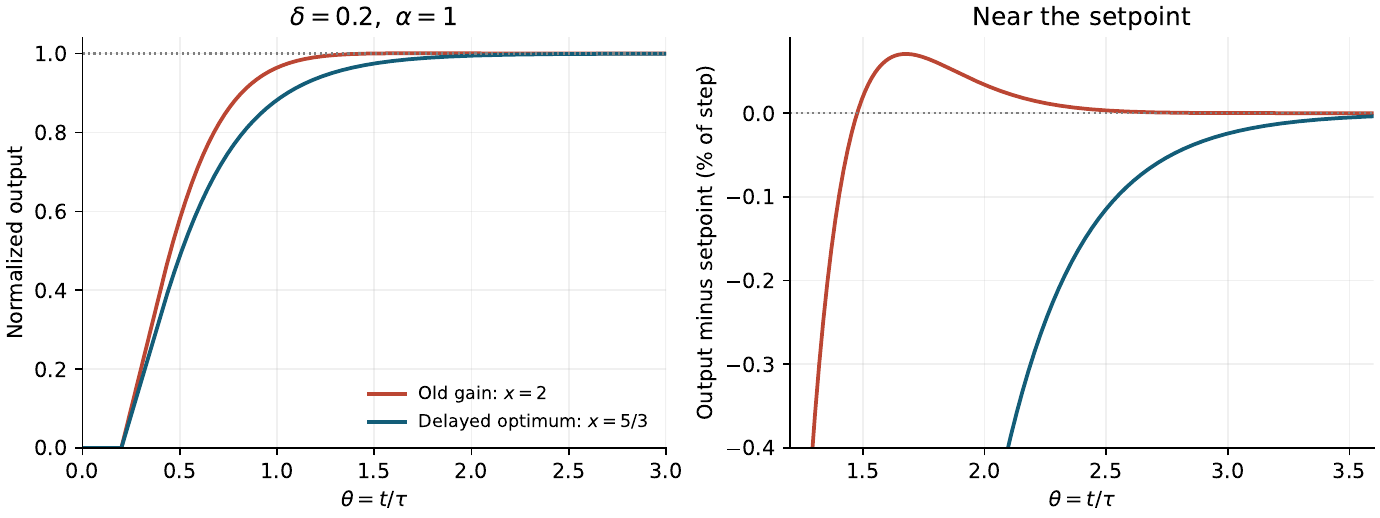}
 \caption{A pure delay $L=0.2\tau$ at $\alpha=1$. Both designs use
 $T_i=\tau$. Red: the delay-free gain $kK_c=2$ retained after adding
 the delay. Blue: the delayed MIACE optimum $kK_c=5/3$.
 The right panel enlarges the neighborhood of the setpoint;
 its vertical scale is percent of the step.}

 \label{fig:delay_miace_responses}
\end{figure}

Monotonicity is the precise output-shape property established here.
It does not mean differentiability at every instant. The output remains
zero before $\theta=\delta$, while its right derivative there is
$x$, so its slope has a corner at the end of the delay. The command
also retains the ideal proportional jump at the reference step.
Neither feature is an oscillation, but neither is removed by minimizing
the infinite-horizon areas.

%% file: sections/06d_filtered_pid.tex
\section{Filtered PID: attaining the same bound beyond the PI range}
\label{sec:filtered_pid}

The preceding lower bound belongs to the plant and the two absolute areas.
Its failure to be attained by PI at a given delay therefore leaves a
specific question: can another controller keep both the error and the
centered command nonnegative while supplying the required area ratio?
A derivative term gives an additional way of shaping the initial command.
To use it in the present problem, its realization must first be specified.

\subsection{Controller realization and the necessary equality conditions}

Consider a parallel PID controller with a first-order derivative filter,
acting on the full tracking error:
\begin{equation}
 C_F(s)=K_p+\frac{K_i}{s}+\frac{K_d s}{1+T_f s},\qquad
 K_p,K_i,K_d,T_f>0.
 \label{eq:pid_physical_law}
\end{equation}
The filter time $T_f$ is an independent tuning parameter. The reference
step, plant, integrator, filter, and command prehistory have the same rest
conditions as in \cref{sec:pure_delay}. In dimensionless variables define
\begin{equation}
 p_c=kK_p,\quad i_c=k\tau K_i,\quad d_c=\frac{kK_d}{\tau},\quad
 \rho=\frac{T_f}{\tau},\qquad
 \overline C_F(p)=p_c+\frac{i_c}{p}+\frac{d_c p}{1+\rho p}.
 \label{eq:pid_normalization}
\end{equation}
Thus $K_p=p_c/k$, $K_i=i_c/(k\tau)$, $K_d=d_c\tau/k$, and
$T_f=\rho\tau$. In the standard time-constant notation,
$T_i=p_c\tau/i_c$ and $T_d=d_c\tau/p_c$.
The proportional and integral symbols extend the parallel PI notation
used in the numerical examples below; they do not reuse the step amplitude.

Filtering makes the command at a reference step an ordinary function
with a finite initial jump. An unfiltered derivative of that step would
produce a Dirac impulse, outside the ordinary $L^1$ signal class used to
define $A_u$. A criterion on measures would require a separate definition.
Differentiating the measured output instead of the full error avoids the
setpoint derivative kick, but changes the closed-loop numerator and is a
different problem. These implementation distinctions are standard in PID
design \citep{AstromHagglund1995}. Filtered PID design and real-pole tuning
are also established topics \citep{Belai2019}. The assertions below concern
exactly \eqref{eq:pid_physical_law}, without a prescribed ratio $T_d/T_f$.

Combining the controller terms over a common denominator gives
\begin{align}
 Q_F(p)&=(p_c\rho+d_c)p^2+(p_c+i_c\rho)p+i_c,\\
 \widehat E(p)&=\frac{(p+1)(1+\rho p)}{
 p(p+1)(1+\rho p)+Q_F(p)e^{-\delta p}}.
 \label{eq:pid_full_error}
\end{align}
For a settled response with finite absolute areas, evaluation at zero
and \eqref{eq:delayed_signed_sum} yield
\begin{equation}
 \int_0^\infty e\dd\theta=\frac1{i_c},\qquad
 \int_0^\infty v\dd\theta=1+\delta-\frac1{i_c}.
 \label{eq:pid_signed_areas}
\end{equation}
Consequently every PID controller attaining the MIACE lower bound must have
\begin{equation}
 \boxed{i_c=i_c^\star:=\frac{1+1/\alpha}{1+\delta}.}
 \label{eq:pid_integral_gain}
\end{equation}
This fixes the integral gain, not the entire PID controller. Indeed,
equality requires $A_e=\alpha A_u$, so that
$1/i_c=\alpha(1+\delta-1/i_c)$. It also requires $e,v\ge0$;
the signed-area calculation alone does not supply these signs.

There is a controller-independent necessary restriction. Since $e=1$
throughout the waiting interval and the output is continuous when that
interval ends, $A_e>\delta$. At equality,
$A_e=\alpha(1+\delta)/(1+\alpha)$, and therefore
\begin{equation}
 \delta<\alpha.
 \label{eq:pid_general_necessary_delay}
\end{equation}
This restriction is not asserted to be sufficient.

\begin{proposition}[Structure of an attaining genuine PID controller]
\label{prop:pid_necessary_structure}
If a controller in \eqref{eq:pid_normalization}, with all four parameters
strictly positive, attains $(1+\delta)m_\alpha$, then
\begin{equation}
 0<\rho<1,\qquad p_c>i_c,\qquad
 d_c=(1-\rho)(p_c-i_c).
 \label{eq:pid_necessary_structure}
\end{equation}
Equivalently, for a parameter $\ell>\rho$,
\begin{align}
 p_c&=i_c(1+\ell-\rho),&
 d_c&=i_c(\ell-\rho)(1-\rho),\\
 \overline C_F(p)&=
 \frac{i_c(p+1)(1+\ell p)}{p(1+\rho p)}.
 \label{eq:pid_cancellation_form}
\end{align}
\end{proposition}
\begin{proof}
At equality, \cref{thm:delayed_shape} gives $e,v\ge0$ and hence
$0\le f(\xi)\le e^{-\xi}$ from \eqref{eq:delayed_equality_decay}.
Thus $\widehat E$ is analytic for $\operatorname{Re}p>-1$ and is
strictly positive on the real interval $(-1,\infty)$. The identity
obtained by multiplying \eqref{eq:pid_full_error} by its denominator
extends throughout that half-plane.

If $\rho>1$, the point $p=-1/\rho$ lies inside this interval.
The numerator is zero there, whereas
$Q_F(-1/\rho)=d_c/\rho^2>0$. The identity would give
$\widehat E(-1/\rho)=0$, a contradiction. If $\rho=1$, the
denominator tends to $d_c e^\delta>0$ as $p$ tends to $-1$ from
the right, while the numerator tends to zero. This would imply
$\widehat E(p)\to0$, although
$\widehat E(p)\ge\widehat E(0)=1/i_c>0$ for $-1<p\le0$.
Therefore $0<\rho<1$.

The same limit now requires $Q_F(-1)=0$. Expanding it gives
$d_c+(\rho-1)(p_c-i_c)=0$. Positivity of $d_c$ and $1-\rho$
implies $p_c>i_c$. Finally put
$\ell=\rho+(p_c-i_c)/i_c$ and factor $Q_F$ to obtain
\eqref{eq:pid_cancellation_form}.
\end{proof}

The cancellation is therefore a consequence of attainment. It was not
imposed to reduce the optimization problem. For this structure write
\begin{equation}
 H_F(p)=\frac{1+\ell p}{1+\rho p},\qquad
 \Delta_F(p)=p(1+\rho p)+i_c(1+\ell p)e^{-\delta p}.
 \label{eq:pid_reduced_denominator}
\end{equation}
Then $\widehat E=(1+\rho p)/\Delta_F$ and the full characteristic
expression is $(p+1)\Delta_F$. The internal mode $p=-1$ is stable;
stability of the remaining delay equation still has to be proved.

\subsection{A positivity argument that also proves stability}

For causal functions write $(K*g)(\theta)=\int_0^\theta
K(\theta-s)g(s)\dd s$. We shall represent the error and centered
command by the same nonnegative convolution kernel.

\begin{lemma}[Nonnegative convolution resolvent]
\label{lem:pid_positive_resolvent}
Suppose $K\ge0$, $K\in L^1(0,\infty)$, and
$q_K=\int_0^\infty K\dd\theta<1$. For $g\ge0$ in $L^1$, the
equation $z=g+K*z$ has the unique $L^1$ solution
\begin{equation}
 z=g+K*g+K*K*g+\cdots\ge0.
 \label{eq:pid_renewal_series}
\end{equation}
If $K$ has a finite positive exponential moment, then a sufficiently
small $\eta>0$ satisfies $\int e^{\eta\theta}K(\theta)\dd\theta<1$.
The same series converges in the exponentially weighted $L^1$ norm
whenever $g$ belongs to that space.
\end{lemma}
\begin{proof}
Young's inequality bounds the norm of the $n$th term by
$q_K^n\|g\|_1$, so the series converges absolutely in $L^1$ and
solves the equation. If two solutions existed, the norm of their
difference would be at most $q_K$ times itself, forcing equality.
Each summand is nonnegative. Continuity of the exponential moment at
$\eta=0$ follows by dominated convergence, using any smaller weight
than the given finite moment. Its value at zero is strictly below one.
Replacing each function by its product with $e^{\eta\theta}$ then
gives the same convolution and norm argument.
\end{proof}

We first construct a kernel with compact support. Choose two real numbers
$1<\lambda<\mu$ for which
\begin{equation}
 i_c(1-\ell\lambda)e^{\lambda\delta}
   =\lambda(1-\rho\lambda),\qquad
 i_c(1-\ell\mu)e^{\mu\delta}
   =\mu(1-\rho\mu).
 \label{eq:pid_two_roots}
\end{equation}
These equations place zeros of $\Delta_F$ at $-\lambda$ and $-\mu$.
They do not assume that these are the dominant characteristic roots.
For chosen $i_c,\delta,\lambda,\mu$, they are solved explicitly by
\begin{align}
 \rho&=\frac{e^{-\lambda\delta}-e^{-\mu\delta}
                  +i_c(1/\mu-1/\lambda)}
               {\lambda e^{-\lambda\delta}-\mu e^{-\mu\delta}},
 \label{eq:pid_two_rate_filter}\\
 \ell&=\frac1\lambda+
          \frac{(\rho\lambda-1)e^{-\lambda\delta}}{i_c}.
 \label{eq:pid_two_rate_lead}
\end{align}
The denominator in \eqref{eq:pid_two_rate_filter} must be nonzero.
The resulting values must also satisfy $0<\rho<1$ and $\ell>\rho$.

Define
\begin{equation}
 A_g=\frac{1/\rho-\lambda}{\mu-\lambda},\quad B_g=1-A_g,
 \qquad g(\theta)=A_g e^{-\lambda\theta}+B_g e^{-\mu\theta}.
 \label{eq:pid_two_rate_g}
\end{equation}
Its transform is
$\widehat g=(1+\rho p)/[\rho(p+\lambda)(p+\mu)]$.
The following sufficient restriction is convenient:
\begin{equation}
 \mu\le\frac1\rho\le\lambda+\mu-1.
 \label{eq:pid_rate_restriction}
\end{equation}
It implies $A_g\ge1$, $B_g\le0$, and $g>0$, since
$g=e^{-\lambda\theta}[A_g-(A_g-1)e^{-(\mu-\lambda)\theta}]$.
More precisely,
\[
 -g'-g=e^{-\lambda\theta}
 \left[(\lambda-1)A_g-(\mu-1)(A_g-1)e^{-(\mu-\lambda)\theta}\right]
 \ge0.
\]
The bracket is smallest at zero, where it equals
$\lambda+\mu-1/\rho-1$. Thus $-g'\ge g>0$.

Partial fractions and \eqref{eq:pid_two_roots} give the useful identity
\begin{equation}
 \mathcal L^{-1}\{i_cH_F\widehat g\}(\theta)=-g'(\theta+\delta).
 \label{eq:pid_shifted_g_identity}
\end{equation}
For example, the residue at $p=-\lambda$ on the left is
$i_c(1-\ell\lambda)/[\rho(\mu-\lambda)]$; the first equation
in \eqref{eq:pid_two_roots} makes it
$\lambda A_g e^{-\lambda\delta}$, the corresponding coefficient
on the right. The residue at $-\mu$ is verified in the same way.
Since both transforms are strictly proper and have only these two
simple poles, these two comparisons prove the identity.

Now put
\begin{equation}
 K(\theta)=
 \begin{cases}-g'(\theta),&0\le\theta<\delta,\\
 0,&\theta\ge\delta.
 \end{cases}
 \label{eq:pid_compact_kernel}
\end{equation}
Using $g(0)=1$ and integrating its derivative up to $\delta$ gives
\begin{equation}
 \widehat K=1-(p+i_cH_F e^{-\delta p})\widehat g,
 \quad \widehat E=\frac{\widehat g}{1-\widehat K},
 \quad \|K\|_1=1-g(\delta)<1.
 \label{eq:pid_compact_resolvent}
\end{equation}
In particular $e\ge0$ by \cref{lem:pid_positive_resolvent}.
The stronger stability statement follows from the exact factorization
\begin{equation}
 \Delta_F(p)=\rho(p+\lambda)(p+\mu)[1-\widehat K(p)].
 \label{eq:pid_stability_factor}
\end{equation}
For a sufficiently small $\eta<\min\{1,\lambda\}$, the weighted
kernel norm is below one, so
$|\widehat K(p)|<1$ for $\operatorname{Re}p\ge-\eta$.
There are no zeros of the full characteristic expression in that strip.
To check that no internal mode has been omitted, take the physical states
$y,z,\zeta$ with $z'=e$, $\rho\zeta'=e-\zeta$, and
$w=p_c e+i_cz+(d_c/\rho)(e-\zeta)$.
The homogeneous three-state retarded system has characteristic determinant
\[
 \det[pI-A_0-A_1e^{-\delta p}]
 =\frac{p(p+1)(1+\rho p)+Q_F(p)e^{-\delta p}}{\rho}
 =\frac{(p+1)\Delta_F(p)}{\rho}.
\]
The characteristic-root criterion for a finite-dimensional retarded
system therefore gives internal exponential stability. Equivalently,
its resolvent uses the exponentially integrable convolution series above
and stable rational factors. This verifies the internal dynamics, in
addition to positivity of the particular step response.

\subsection{Keeping the centered command nonnegative}

The error sign alone does not prove MIACE optimality. From
\eqref{eq:pid_cancellation_form}, the controller relation
$\widehat W=\overline C_F\widehat E$, and
$1=(p+i_cH_F e^{-\delta p})\widehat E$, subtraction of the new
equilibrium command gives
\begin{equation}
 \widehat V=
 \left[i_cH_F\left(1+\frac{1-e^{-\delta p}}p\right)-1\right]
 \widehat E=\frac{\widehat R}{1-\widehat K}.
 \label{eq:pid_command_transform}
\end{equation}
Here \eqref{eq:pid_shifted_g_identity} yields
\begin{equation}
 R(\theta)=-g'(\theta+\delta)-g(\theta+\delta)
                  +g(\max\{\delta,\theta\})-g(\theta).
 \label{eq:pid_command_forcing}
\end{equation}
Indeed, the factor $(1-e^{-\delta p})/p$ integrates
$-g'(s+\delta)$ from $\max(0,\theta-\delta)$ to $\theta$.
The integral is $g(\max\{\delta,\theta\})-g(\theta+\delta)$,
which proves the displayed formula after including the remaining terms.

For $\theta\ge\delta$, this reduces to
$R=-g'(\theta+\delta)-g(\theta+\delta)\ge0$.
On the initial interval define
$c_z=(z-1)e^{-z\delta}-1$. Then
\begin{equation}
 R(\theta)=g(\delta)+A_gc_\lambda e^{-\lambda\theta}
                         +B_gc_\mu e^{-\mu\theta},
 \qquad 0\le\theta\le\delta.
 \label{eq:pid_initial_forcing}
\end{equation}
A sufficient check, requiring only three scalar signs, is
\begin{equation}
 c_\lambda\ge0,\qquad c_\mu\ge0,\qquad
 g(\delta)-(A_g-1)c_\mu\ge0.
 \label{eq:pid_simple_forcing_check}
\end{equation}
The positive term in \eqref{eq:pid_initial_forcing} may then be
discarded for a lower bound, and the magnitude of its negative term
is at most $(A_g-1)c_\mu$.
If this sufficient check fails, the exact minimum of $R$ on the initial
interval can still be checked at its endpoints and at the single possible
interior stationary point
\begin{equation}
 \theta_0=\frac1{\mu-\lambda}
       \log\!\left(-\frac{\mu B_gc_\mu}{\lambda A_gc_\lambda}\right).
 \label{eq:pid_forcing_stationary}
\end{equation}
This point is included only when its logarithm has a positive argument
and $0<\theta_0<\delta$. If a coefficient vanishes, the remaining
single exponential is monotone or constant, so endpoints suffice.

\begin{theorem}[A two-rate sufficient condition for the delayed PID optimum]
\label{thm:pid_two_rate_optimum}
Let $i_c=i_c^\star$ in \eqref{eq:pid_integral_gain}. Suppose
\eqref{eq:pid_two_rate_filter}--\eqref{eq:pid_two_rate_lead} give
$0<\rho<1$, $\ell>\rho$, and satisfy
\eqref{eq:pid_rate_restriction}. If $R\ge0$ on $[0,\delta]$,
as guaranteed in particular by \eqref{eq:pid_simple_forcing_check},
then \eqref{eq:pid_cancellation_form} is internally exponentially
stable and attains
\begin{equation}
 \boxed{A_e=\frac{\alpha(1+\delta)}{1+\alpha},\qquad
 A_u=\frac{1+\delta}{1+\alpha},\qquad
 \Mcal_\alpha=(1+\delta)\sqrt{\frac\alpha{1+\alpha}}.}
 \label{eq:pid_optimal_areas}
\end{equation}
Its output is nondecreasing from zero to one. The minimum is global over
all admissible commands of the same plant with the stipulated initial
conditions, and consequently over the specified filtered PID family.
\end{theorem}
\begin{proof}
The preceding kernel argument proves stability and $e\ge0$.
For $R\ge0$, \cref{lem:pid_positive_resolvent} applied to
\eqref{eq:pid_command_transform} gives $v\ge0$ and absolute
integrability. The signed areas in \eqref{eq:pid_signed_areas}
therefore equal the absolute areas. The choice $i_c=i_c^\star$
balances them in the required ratio. All equality conditions of
\cref{thm:delayed_shape} are satisfied, which proves the minimum
and the output monotonicity.
\end{proof}

\subsection{Equal weights: an exact construction beyond the PI boundary}

Take $\alpha=1$, $\delta=1/4$, $i_c=8/5$, $\lambda=5$, and $\mu=6$.
The formulas become
\begin{equation}
 \rho=\frac{e^{-5/4}-e^{-3/2}-4/75}{5e^{-5/4}-6e^{-3/2}},\qquad
 \ell=\frac15+\frac58(5\rho-1)e^{-5/4}.
 \label{eq:pid_balanced_exact}
\end{equation}
For orientation, their decimal values and the resulting gains are
\begin{align}
 \rho&\simeq0.1071152072,& \ell&\simeq0.1168376916,\\
 p_c&\simeq1.6155559750,& i_c&=1.6,\qquad d_c\simeq0.0138896935.
 \label{eq:pid_balanced_decimals}
\end{align}
The theorem concerns the exact expressions and
\eqref{eq:pid_cancellation_form}; rounded gains do not preserve an exact
cancellation or area equality by definition.

All signs can be checked with short bounds. The exact formulas imply
\[
 0.10711<\rho<0.10712,\qquad
 0.11683<\ell<0.11685,\qquad 4.33<A_g<4.34.
\]
Thus $6<1/\rho<10$, proving \eqref{eq:pid_rate_restriction} strictly.
Also
\[
 0.28650<e^{-5/4}<0.28651,\qquad
 0.22313<e^{-3/2}<0.22314.
\]
These bounds give $c_5>0.146$ and $c_6>0.11565$. Finally,
\begin{align*}
 g(1/4)&>4.33(0.28650)-3.34(0.22314)=0.4952574,\\
 (A_g-1)c_6&<3.34(5(0.22314)-1)=0.386438.
\end{align*}
Their difference exceeds $0.1088$, so the command sign check is strict.
The exponential enclosures, and all subsequent interval operations,
are also verified in exact rational arithmetic in the supplied code.
The even Taylor sum through degree 36 is an upper bound and the odd sum
through degree 37 a lower bound for each exponential used here.

Consequently,
\begin{equation}
 A_e=A_u=\frac58,\qquad
 \Mcal_1=\frac5{4\sqrt2}=0.8838834765\ldots.
 \label{eq:pid_balanced_value}
\end{equation}
The exact PI boundary from \cref{thm:delayed_optimum} is
$1/(2\mathrm e-1)=0.225399674\ldots<1/4$.
Thus this PID attains a value which no positive ideal PI controller with
finite cost attains for the same problem. The extrapolated PI setting
$r=1,x=1.6$ is stable, but $x\delta=0.4>1/\mathrm e$ and its
output is not monotone.

Unlike the PI optimum, the PID optimum is not unique. The strict
inequalities above persist under small changes of $\lambda$ and $\mu$.
Equations \eqref{eq:pid_two_rate_filter}--\eqref{eq:pid_two_rate_lead}
then give a continuous family of attaining settings. This family cannot
reduce to a single controller: its characteristic function would otherwise
vanish at every $-\lambda$ in an interval, contrary to analyticity and
$\Delta_F(0)=i_c>0$. The separate choice $\lambda=5$, $\mu=7$
also passes the exact rational interval checks. Additional controller
parameters therefore permit different command shapes with the same two
optimal areas.

\subsection{An explicit family for larger effort weights}

An additional interval of attainable delays can be described without
selecting two rates. Suppose $\alpha>\mathrm e$ and
\begin{equation}
 \frac1\alpha<\delta\le
 \min\left\{\frac1{\mathrm e},
              \frac{\mathrm e}{\alpha+1-\mathrm e}\right\}.
 \label{eq:pid_explicit_range}
\end{equation}
Set
\begin{equation}
 i_c=\frac{1+1/\alpha}{1+\delta},\qquad
 \rho=1-i_c,\qquad \ell=\frac\rho{i_c},\qquad
 \boxed{p_c=i_c+\rho^2,\qquad d_c=i_c\rho^2.}
 \label{eq:pid_explicit_tuning}
\end{equation}
These coefficients satisfy \eqref{eq:pid_cancellation_form}. All four
controller parameters are positive. At the excluded left endpoint,
$\delta=1/\alpha$, they have the PI limit $i_c=1$, $\rho=d_c=0$;
the genuine filtered PID uses the strict lower inequality.

\begin{theorem}[An explicit additional attainment interval]
\label{thm:pid_explicit_family}
The controller \eqref{eq:pid_explicit_tuning}, on
\eqref{eq:pid_explicit_range}, is internally exponentially stable,
attains \eqref{eq:pid_optimal_areas}, and has a monotone step output.
No positive ideal PI controller attains the same lower bound at any
delay in this interval.
\end{theorem}
\begin{proof}
Put $\lambda=1/\delta$. The second upper bound in
\eqref{eq:pid_explicit_range} is equivalent to
\begin{equation}
 i_c\ge\frac\lambda{\lambda+1-e^{-1}}.
 \label{eq:pid_explicit_gain_condition}
\end{equation}
Therefore $\ell=(1-i_c)/i_c\le(1-e^{-1})/\lambda<1/\lambda$,
and $0<\rho<\ell<1/\lambda$.

Use $g=e^{-\lambda\theta}$ and
$\widehat K=(\lambda-i_cH_F e^{-\delta p})/(p+\lambda)$.
Again $\widehat E=\widehat g/(1-\widehat K)$. Define
\[
 A_0=\frac{1-\lambda\ell}{1-\lambda\rho}>0,\qquad
 B_0=\frac{\ell-\rho}{\rho(1-\lambda\rho)}>0.
\]
Partial fractions give
$\mathcal L^{-1}\{H_F/(p+\lambda)\}
=A_0e^{-\lambda\theta}+B_0e^{-\theta/\rho}$, and hence
\begin{equation}
 K(\theta)=
 \begin{cases}
 \lambda e^{-\lambda\theta},&0\le\theta<\delta,\\
 (\lambda-i_cA_0e^{\lambda\delta})e^{-\lambda\theta}
       -i_cB_0e^{-(\theta-\delta)/\rho},&\theta\ge\delta.
 \end{cases}
 \label{eq:pid_explicit_kernel}
\end{equation}
Since $i_c\ell/\rho=1$, its value immediately after the delay is
$K(\delta^+)=\lambda e^{-\lambda\delta}-1
=1/(\mathrm e\delta)-1\ge0$. The negative exponential in the
second branch decays faster than the positive one because $1/\rho>\lambda$.
Thus this branch stays nonnegative, as does the first. Its norm is
$\widehat K(0)=1-i_c/\lambda<1$.

The same command relation \eqref{eq:pid_command_transform}, now with
this $g$ and $K$, gives on $0\le\theta\le\delta$
\begin{equation}
 R(\theta)=\frac{i_c}{\lambda(1-\lambda\rho)}
 \left[1-\lambda\rho-e^{-\lambda\theta}
                   +\lambda\rho e^{-\theta/\rho}\right].
 \label{eq:pid_explicit_R_initial}
\end{equation}
It starts at zero and has derivative
$R'=i_c(e^{-\lambda\theta}-e^{-\theta/\rho})/(1-\lambda\rho)\ge0$.
For $\theta\ge\delta$, direct integration of the two exponentials gives
\begin{align}
 R(\theta)&=C_0e^{-\lambda\theta}
     +i_cB_0[1+\rho(e^{\delta/\rho}-1)]e^{-\theta/\rho},\\
 C_0&=i_cA_0\left[1+\frac{e^{\lambda\delta}-1}{\lambda}\right]-1.
 \label{eq:pid_explicit_R_tail}
\end{align}
The second coefficient is positive. To verify the first without a sign
guess, substitute $\rho=1-i_c$ and $\lambda\delta=1$. Then
\[
 i_c A_0=\frac{i_c-\lambda(1-i_c)}{1-\lambda\rho},\qquad
 C_0=\frac{i_c(\mathrm e-1)-\lambda\mathrm e(1-i_c)}
                 {\lambda(1-\lambda\rho)}.
\]
The denominator is positive. Its numerator is nonnegative exactly when
$i_c\ge\lambda\mathrm e/(\lambda\mathrm e+\mathrm e-1)$, which
is \eqref{eq:pid_explicit_gain_condition}. Thus $R\ge0$ everywhere.
The convolution series proves $e,v\ge0$ and absolute integrability.
Here $K$ has an exponentially decaying tail, so its norm remains below
one for a sufficiently small positive exponential weight. The factorization
$\Delta_F=(p+\lambda)(1+\rho p)(1-\widehat K)$ and the full
characteristic determinant used above then prove internal exponential
stability. The prescribed integral gain gives
\eqref{eq:pid_optimal_areas}, and \cref{thm:delayed_shape} proves
monotonicity and global optimality.

For $\alpha>\mathrm e$, the smaller PI delay limit in
\eqref{eq:delayed_attainment_range} is $1/\alpha$: the strict
comparison reduces to $(\alpha+1)(\alpha-\mathrm e)>0$.
Every delay in \eqref{eq:pid_explicit_range} exceeds that limit.
\end{proof}

For $\alpha=4$ the additional interval established above is
$1/4<\delta\le1/\mathrm e=0.367879441\ldots$.
At $\delta=1/3$, all controller coefficients are rational:
\begin{equation}
 \boxed{p_c=\frac{241}{256},\quad i_c=\frac{15}{16},\quad
 d_c=\frac{15}{4096},\quad \rho=\frac1{16}},\qquad \ell=\frac1{15}.
 \label{eq:pid_rational_example}
\end{equation}
The areas and criterion are exactly
\begin{equation}
 A_e=\frac{16}{15},\qquad A_u=\frac4{15},\qquad
 \Mcal_4=\frac{8\sqrt5}{15}=1.192569588\ldots.
 \label{eq:pid_rational_value}
\end{equation}
The initial command satisfies $w(0^+)=p_c+d_c/\rho=1$ and
$w'(0^+)=i_c-d_c/\rho^2=0$. Before the output can respond,
\begin{equation}
 w(\theta)=p_c+i_c\theta+\frac{d_c}{\rho}e^{-\theta/\rho},
 \qquad 0\le\theta\le\delta.
 \label{eq:pid_initial_command}
\end{equation}
It rises from the equilibrium value and ultimately returns to it.
The extrapolated PI setting with $x=i_c=15/16$, $r=1$, has a
monotone output because $x\delta=5/16<1/\mathrm e$, but starts
below the equilibrium command. Its negative command area prevents equality.

\begin{figure}[htbp]
 \centering
 \includegraphics[width=\linewidth]{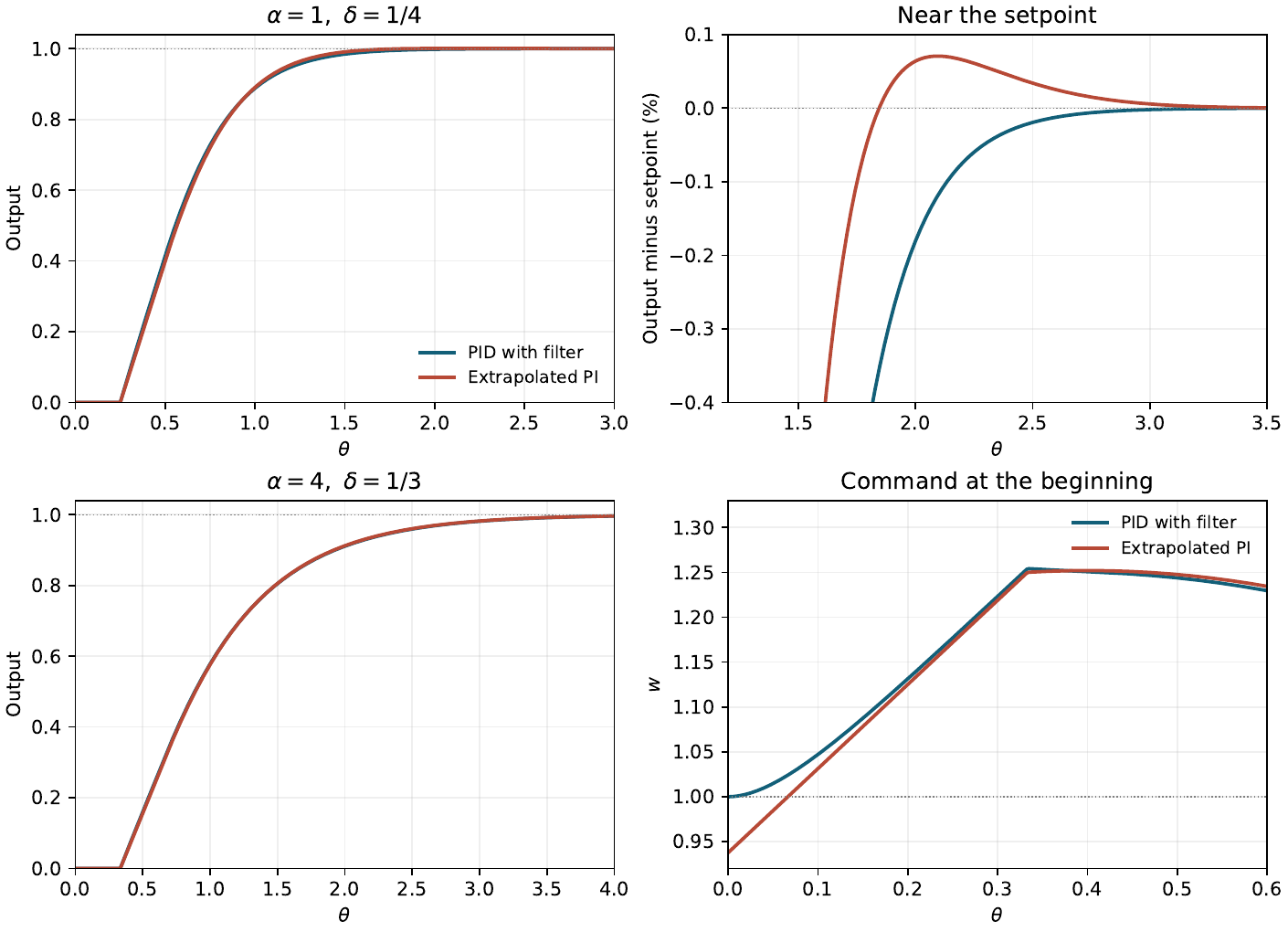}
 \caption{Two filtered PID constructions beyond the PI attainment range.
 At $\alpha=1$, $\delta=1/4$, the PID of
 \eqref{eq:pid_balanced_exact} is monotone, while the extrapolated
 PI setting $x=1.6,r=1$ slightly overshoots. At $\alpha=4$,
 $\delta=1/3$, both plotted outputs are monotone, but the PI command
 starts below equilibrium. The PID of \eqref{eq:pid_rational_example}
 keeps both error and centered command nonnegative.}

 \label{fig:pid_delay_comparison}
\end{figure}

These constructions establish additional attainment, not a complete
filtered PID delay boundary or Pareto front. The necessary conditions
\eqref{eq:pid_general_necessary_delay}, \eqref{eq:pid_integral_gain}, and
\eqref{eq:pid_necessary_structure} do not become sufficient without a
sign and stability argument. Nor does output monotonicity imply a minimum
command peak, freedom from a setpoint jump, or robustness to delay and
parameter errors. The derivative is filtered, but the output can still
have a slope corner at the end of the waiting interval. The precise
conclusion is the absence of output overshoot and backward motion under
the stated nominal model and controller realization.

%% file: sections/06e_arbitrary_order.tex
\section{Positive cascades of arbitrary order with pure delay}
\label{sec:order_n}

The first-order balance makes the shape argument particularly direct:
when $e$ and $v$ are nonnegative, $y'=e+v$ cannot be negative.
Additional storage elements change this last step. Their intermediate
states can move differently, and the derivative of the final output is
no longer the sum of the two signals used in MIACE. The area balance
survives, but monotonicity requires a separate argument about the plant.

Consider a cascade of stable first-order lags,
\begin{equation}
 G_L(s)=\frac{k e^{-Ls}}{\prod_{j=1}^N(1+\tau_js)},
 \qquad k>0,\quad \tau_j>0,\quad L\ge0.
 \label{eq:n_physical}
\end{equation}
Choose $\tau_* =\max_j\tau_j$ and order the lags so that
$\tau_1=\tau_*$. The variables in this section are
\begin{gather}
 p=\tau_*s,\quad \theta=t/\tau_*,\quad
 \delta=L/\tau_*,\quad \sigma_j=\tau_j/\tau_*,\qquad
 \sigma_1=1,\quad 0<\sigma_j\le1,\\
 A(p)=\prod_{j=1}^N(1+\sigma_jp),\qquad H(p)=\frac1{A(p)}.
 \label{eq:n_normalization}
\end{gather}
The same unit step, zero initial states, zero command prehistory, and
equilibrium-centered command $v=ku-1$ are retained. The areas in MIACE
are now integrated in $\theta=t/\tau_*$. Their dimensional counterparts,
when the command is still normalized as $w=ku$, are multiplied by $\tau_*$.

A controller of order $N$ means a proper rational transfer function with
McMillan degree $N$, including its integrator. A finite feedthrough gain
is allowed; an ideal derivative producing an input impulse is not.
The pure delay is retained exactly and is not counted as a rational
state. None of the controllers constructed below contains an internal
delay block. The symbol $\Ncal=N_e+N_v$ continues to denote the total
negative area; $N$ denotes the plant order.

\subsection{The general signed-area identity}
\label{sec:n_balance}

For this identity alone, $H$ need not be a positive cascade. Let it be
any stable strictly proper real rational function with $H(0)=1$.
Assume $e,v\in L^1(0,\infty)$ and settling to $y=w=1$. Define
\begin{equation}
 \mu=-H'(0),\qquad S=\mu+\delta.
 \label{eq:n_mu}
\end{equation}
For \eqref{eq:n_physical},
\begin{equation}
 \mu=\sum_{j=1}^N\sigma_j,\qquad
 S=\delta+\sum_{j=1}^N\sigma_j.
 \label{eq:n_S}
\end{equation}
For a general normalized $H=B_0/A_0$, with $A_0(0)=B_0(0)=1$,
the corresponding expression is $\mu=A_0'(0)-B_0'(0)$. It is a
low-frequency moment, not the order of the plant. It need not be
positive when plant zeros are admitted. The lower-bound and equality
statements below assume $S>0$, as is automatic for the positive cascade.

\begin{theorem}[MIACE excess for a general stable plant]
\label{thm:n_gap}
Under these assumptions,
\begin{gather}
 \int_0^\infty e\dd\theta+\int_0^\infty v\dd\theta=S,
 \qquad A_e+A_u=S+2\Ncal,\label{eq:n_balance}\\
 \boxed{\Mcal_\alpha^2-S^2m_\alpha^2
 =4m_\alpha^2\Ncal(S+\Ncal)
   +\frac{(A_e-\alpha A_u)^2}{1+\alpha}.}
 \label{eq:n_gap}
\end{gather}
Consequently $\Mcal_\alpha\ge S m_\alpha$. Equality holds precisely
when $e,v\ge0$ almost everywhere and $A_e=\alpha A_u$. The equality
areas are
\begin{equation}
 A_e^*=\frac{\alpha S}{1+\alpha},\qquad
 A_u^*=\frac{S}{1+\alpha}.
 \label{eq:n_areas}
\end{equation}
\end{theorem}
\begin{proof}
Write $E,V,W,Y$ for the Laplace transforms of the corresponding signals.
Since $W=1/p+V$ and $Y=H(p)e^{-\delta p}W$,
\begin{equation}
 E(p)=\frac{1-H(p)e^{-\delta p}}p-H(p)e^{-\delta p}V(p).
 \label{eq:n_E}
\end{equation}
Let $p\to0^+$. Absolute integrability gives $E(0)=\int e$ and
$V(0)=\int v$, while the first term tends to
$-[H(p)e^{-\delta p}]'_{p=0}=\mu+\delta$. This proves the signed
identity in \eqref{eq:n_balance}. For a real integrable function $g$,
$\int|g|=\int g+2\int\max\{-g,0\}$, proving the absolute identity.

Apply the algebraic decomposition \eqref{eq:miace_area_decomposition}
to $A_e+A_u=S+2\Ncal$. Subtracting $S^2m_\alpha^2$ leaves
\eqref{eq:n_gap}. Since $S>0$, its two nonnegative terms vanish
exactly when $\Ncal=0$ and $A_e=\alpha A_u$. Solving the sum and
ratio equations yields \eqref{eq:n_areas}.
\end{proof}

Thus a higher-order plant preserves the exact MIACE penalty for output
overshoot and command deficit. It does not yet preserve the conclusion
about backward output motion. If the controller has a simple integrator,
$\overline C(p)=b/p+O(1)$, $b>0$, settling of $w$ gives
$1=\lim_{p\to0}pW(p)=bE(0)$. Every attaining controller of this kind
must therefore have
\begin{equation}
 b=b_*:=\frac{1+1/\alpha}{S}.
 \label{eq:n_integral_gain}
\end{equation}
This fixes one coefficient. It does not prove that a stable controller
with the required signal signs exists.

\subsection{Why a positive cascade is monotone at equality}
\label{sec:n_monotonic}

Nonnegative impulse responses and monotone step responses are established
objects of linear-systems theory
\citep{LinFang1997,Darbha2003,DrummondTurner2023}.
Here the issue is whether the two MIACE equality signs force that shape
for the plant being controlled. The following state argument answers
this question without assuming a particular controller.

\begin{lemma}[One-way crossing through a cascade]
\label{lem:n_crossing}
Let $x_j(0)=0$ and
\begin{equation}
 \sigma_jx_j'+x_j=x_{j-1},\qquad j=1,\ldots,N,
 \qquad x_0=w\ge1,
 \label{eq:n_states}
\end{equation}
where $w$ is an ordinary locally integrable function. Each output $x_j$
is nondecreasing while it is below one. Once it reaches one, it cannot
return below one.
\end{lemma}
\begin{proof}
For the first lag, $x_1'=(w-x_1)/\sigma_1\ge0$ while $x_1\le1$.
If it reaches one, variation of constants with $w\ge1$ keeps it at
least one at every later time.

Suppose the assertion has been proved for $x_{j-1}$. Until that state
first reaches one, it is nonnegative and nondecreasing. Therefore
\begin{equation}
 x_j(\theta)=\frac1{\sigma_j}\int_0^\theta
 e^{-(\theta-r)/\sigma_j}x_{j-1}(r)\dd r
 \le x_{j-1}(\theta).
\end{equation}
It follows that $x_j'\ge0$ on this interval. After $x_{j-1}$ reaches
one, it stays at least one by the induction hypothesis. The equation
then gives $x_j'\ge0$ whenever $x_j\le1$. If $x_j$ reaches one,
the same variation-of-constants argument prevents a later return
below it. This completes the induction.
\end{proof}

\begin{theorem}[Shape of every attaining cascade response]
\label{thm:n_monotonic}
For \eqref{eq:n_physical}, any admissible controller attaining
$S m_\alpha$ has a nondecreasing output from zero to one. In particular,
it has neither overshoot nor backward output motion.
\end{theorem}
\begin{proof}
Equality in \cref{thm:n_gap} gives $w=1+v\ge1$ and $y\le1$.
The output is zero on $[0,\delta)$. Shift time by $\delta$; the
remaining response is the last state of \eqref{eq:n_states}, driven
by $w(\theta)$. By \cref{lem:n_crossing}, that state can decrease
only after exceeding one. The error sign excludes such an excursion.
All cascade states are nonnegative because each lag has a nonnegative
convolution kernel. Hence $0\le y\le1$ and $y'\ge0$ almost everywhere.
\end{proof}

This theorem concerns attainment of the universal lower bound. If a
restricted controller family cannot attain that bound, its own minimizing
response is not characterized by this equality argument.

\subsection{A quantitative extension under an additional positivity condition}
\label{sec:n_variation}

For some higher-order plants, a value near the lower bound also limits
the amount of backward motion. Let $H$ be as in \cref{thm:n_gap},
with $\mu>0$, and define the balance operator
\begin{equation}
 Z(p)=\frac{pH(p)}{1-H(p)}.
 \label{eq:n_Z}
\end{equation}
Assume its causal impulse response is a nonnegative finite measure $q$.
Its mass is then $Z(0)=1/\mu$. A feedthrough mass at zero is allowed
in this auxiliary operator; the physical command remains an ordinary
function.

Put $f(\theta)=e(\theta+\delta)$. The undelayed output is driven
by $W=1/p+V$, and
$F+V=(1-H)W$. Consequently
\begin{equation}
 \mathcal L\{y'(\theta+\delta)\}=Z(p)[F(p)+V(p)].
\end{equation}
The convolution inequality and $\int|f|=A_e-\delta$ give
\begin{equation}
 \mathcal V_y\le\frac{A_e-\delta+A_u}{\mu}
 =1+\frac{2\Ncal}{\mu}.
\end{equation}
Since $\mathcal V_y=1+2B_y$ and $P_y\le B_y$,
\begin{equation}
 \boxed{P_y\le B_y\le\frac{\Ncal}{\mu}
 \le\frac{\Mcal_\alpha/m_\alpha-S}{2\mu}.}
 \label{eq:n_variation_bound}
\end{equation}
The final inequality follows from
$\Mcal_\alpha\ge m_\alpha(A_e+A_u)$.
In particular, $\Mcal_\alpha\le(1+\eta)S m_\alpha$ implies
$P_y,B_y\le\eta S/(2\mu)$.

The positivity assumption is essential. For
$H=(1+a_1p+a_2p^2)^{-1}$ with $a_1,a_2>0$, it holds because
$Z=(a_1+a_2p)^{-1}$. For
$H=(1+a_1p+a_2p^2+a_3p^3)^{-1}$, positive coefficients and
$a_2^2\ge4a_1a_3$ are sufficient. In contrast, $H=(1+p)^{-3}$
has
\begin{equation}
 Z(p)=\frac1{p^2+3p+3},\qquad
 q(\theta)=\frac2{\sqrt3}e^{-3\theta/2}
                   \sin\frac{\sqrt3\theta}{2},
\end{equation}
which changes sign. Thus \eqref{eq:n_variation_bound} is not an
automatic all-order consequence of \eqref{eq:n_gap}.
The exact-equality conclusion of \cref{thm:n_monotonic} still applies
to this three-lag cascade, by its separate state proof.

\subsection{A necessary weight--delay inequality}
\label{sec:n_necessary}

The balance fixes the required command area, but the slowest plant pole
limits how much nonnegative extra command can be applied without
eventually crossing the target.

\begin{theorem}[Necessary condition for attaining the bound]
\label{thm:n_necessary}
If an ordinary nonimpulsive command attains $S m_\alpha$ for
\eqref{eq:n_physical}, then
\begin{equation}
 \boxed{\alpha>\mu+\delta-1.}
 \label{eq:n_necessary}
\end{equation}
Equivalently,
\begin{equation}
 L<\alpha\tau_*-\sum_{j=2}^N\tau_j.
 \label{eq:n_necessary_dimensional}
\end{equation}
\end{theorem}
\begin{proof}
Let $h_A$ be the impulse response of $1/A$. Its step-error tail is
\begin{equation}
 F_0(\theta)=1-\int_0^\theta h_A(r)\dd r
            =\int_\theta^\infty h_A(r)\dd r>0.
\end{equation}
After shifting the output by $\delta$,
\begin{equation}
 f(\theta)=F_0(\theta)-(h_A*v)(\theta).
 \label{eq:n_tail_identity}
\end{equation}
At equality, $f,v\ge0$, and hence $(h_A*v)(\theta)/F_0(\theta)\le1$.

The slowest pole is $-1$. If its multiplicity is $q_*$, partial fractions
give
\begin{equation}
 h_A(\theta)\sim d\theta^{q_*-1}e^{-\theta},\qquad
 F_0(\theta)\sim d\theta^{q_*-1}e^{-\theta},
 \qquad
 d=\frac1{(q_*-1)!\prod_{\sigma_j<1}(1-\sigma_j)}>0.
\end{equation}
For each fixed $r\ge0$, therefore,
$h_A(\theta-r)/F_0(\theta)\to e^r$ as $\theta\to\infty$.
Extend $h_A$ by zero at negative arguments and apply Fatou's lemma
to the nonnegative convolution integrands. This yields
\begin{equation}
 \int_0^\infty e^r v(r)\dd r\le1.
 \label{eq:n_exponential_moment}
\end{equation}
The equality area $\int v=S/(1+\alpha)$ is strictly positive.
An ordinary function cannot concentrate this area at the single time
$r=0$. Since $e^r>1$ for $r>0$,
\begin{equation}
 \frac{S}{1+\alpha}=\int_0^\infty v(r)\dd r
 <\int_0^\infty e^r v(r)\dd r\le1.
\end{equation}
Rearranging proves \eqref{eq:n_necessary}.
\end{proof}

For equal time constants, the condition is
$\alpha>N-1+\delta$. In particular, $\alpha=1$ cannot attain the
universal bound for $N\ge2$, even at zero delay. This statement does
not force the actual optimum above the bound to overshoot. Its value
and shape remain a different optimization problem. The strict inequality
also depends on excluding impulsive commands.

\subsection{An exact finite-support remainder for the delay}
\label{sec:n_kernel}

For the construction take $N\ge2$ and write
\begin{equation}
 m=N-1,\quad B(p)=\prod_{j=2}^N(1+\sigma_jp),\quad
 A(p)=(1+p)B(p),\quad \varepsilon_N=(-1)^m.
 \label{eq:n_AB}
\end{equation}
Let $h_B$ be the impulse response of $1/B$ at nonnegative times.
As an exponential polynomial it extends analytically to negative times.
Only this analytic extension is used to define the following kernel;
the physical plant remains causal:
\begin{equation}
 J(t)=(-1)^{m-1}h_B(t-\delta),\quad 0\le t\le\delta,
 \qquad J(t)=0\quad(t>\delta).
 \label{eq:n_J}
\end{equation}
The kernel is nonnegative. To establish its sign also for repeated
lags, let $r_j=1/\sigma_j$, $j=2,\ldots,N$. Repeated convolution of
the exponentials $r_je^{-r_jt}$ gives
\begin{equation}
 h_B(t)=\left(\prod_{j=2}^Nr_j\right)t^{m-1}
 \int_{\Delta_{m-1}}
   \exp\left(-t\sum_{j=2}^Nr_ju_j\right)\dd u.
 \label{eq:n_simplex}
\end{equation}
Here $u_j\ge0$, $\sum u_j=1$, and the simplex measure is normalized
to have volume $1/(m-1)!$. When $m=1$, the integral is the single
exponential. The identity extends analytically to real $t$; its
integral is positive for negative $t$ as well. Thus the sign there is
$(-1)^{m-1}$. This simplex representation belongs to the classical
divided-difference and Hermite--Genocchi framework
\citep[Sec.~9]{deBoor2005}; it is not a new interpolation principle.

Let $I_\delta$ be the unique polynomial of degree at most $m-1$
interpolating $e^{-\delta p}$ at the roots of $B$, with multiplicities.
At a root $z$ of multiplicity $q$ it satisfies
\begin{equation}
 I_\delta^{(\ell)}(z)=(-\delta)^\ell e^{-\delta z},
 \qquad \ell=0,\ldots,q-1.
 \label{eq:n_Hermite}
\end{equation}
These are ordinary confluent Hermite conditions \citep{deBoor2005}.
What is useful for the delayed control problem is the exact time-domain
form of their remainder.

\begin{lemma}[Exact delay remainder]
\label{lem:n_remainder}
With \eqref{eq:n_J} and \eqref{eq:n_Hermite},
\begin{equation}
 \boxed{e^{-\delta p}=I_\delta(p)+
       \varepsilon_N B(p)\widehat J(p),\qquad
 \widehat J(p)=\int_0^\delta e^{-pt}J(t)\dd t.}
 \label{eq:n_remainder}
\end{equation}
\end{lemma}
\begin{proof}
The Hermite conditions make the principal parts of $I_\delta/B$
and $e^{-\delta p}/B$ identical at every root of $B$. Inverting the
partial fractions of $I_\delta/B$ gives the analytically continued
$h_B(t-\delta)$ for $t\ge0$. In the residue formula the derivatives
act on $e^{p(t-\delta)}$, producing the required powers of $t-\delta$
when roots repeat.

The causal inverse of $e^{-\delta p}/B$ is zero before $\delta$ and
equals $h_B(t-\delta)$ afterwards. The difference of these two
inverses is therefore $-h_B(t-\delta)=\varepsilon_NJ(t)$ before
$\delta$ and zero afterwards. Taking transforms and multiplying
by $B$ proves the formula. Both sides are entire, so it also holds
at the roots of $B$ by continuation.
\end{proof}

No rational approximation of the delay has been made. The polynomial
$I_\delta$ and the compactly supported remainder are both retained.
Define the ordinary moments and the exponential moment separately:
\begin{gather}
 j_0=\int_0^\delta J(t)\dd t,\qquad
 j_1=\int_0^\delta tJ(t)\dd t,\qquad
 \Phi(\lambda)=\int_0^\delta e^{\lambda t}J(t)\dd t,\\
 K(t)=\int_t^\delta J(r)\dd r\quad(0\le t\le\delta),
 \qquad K(t)=0\quad(t>\delta).
 \label{eq:n_moments}
\end{gather}
Then $K\ge0$, $\int K=j_1$, and
\begin{equation}
 \widehat K(p)=\frac{j_0-\widehat J(p)}p.
 \label{eq:n_K}
\end{equation}

\subsection{A proper rational controller of the same order}
\label{sec:n_controller}

Set
\begin{equation}
 U_* =\frac{S}{1+\alpha},\qquad
 a=1-U_*+\varepsilon_N j_1,\qquad
 c=1-\varepsilon_N j_0,
 \label{eq:n_ac}
\end{equation}
and define
\begin{equation}
 Q(p)=B(p)(ap+c)-I_\delta(p),\qquad
 \boxed{\overline C_N(p)=\frac{A(p)}{Q(p)},\qquad
 C_N(s)=\frac1k\overline C_N(\tau_*s).}
 \label{eq:n_controller}
\end{equation}
These coefficients depend on the known delay, but the controller itself
is rational. Its integrator is a consequence of the construction.
At zero, \eqref{eq:n_remainder} gives $I_\delta(0)=c$; differentiation
gives
\begin{equation}
 I_\delta'(0)=-\delta-\varepsilon_N B'(0)j_0
                         +\varepsilon_N j_1.
\end{equation}
Using $B'(0)=\mu-1$,
\begin{align}
 Q(0)&=0,\notag\\
 Q'(0)&=a+cB'(0)-I_\delta'(0)
       =a+B'(0)+\delta-\varepsilon_N j_1\notag\\
       &=S-U_*=\frac{\alpha S}{1+\alpha}>0.
 \label{eq:n_Qprime}
\end{align}
The zero at the origin is simple and the integral residue is $b_*$
from \eqref{eq:n_integral_gain}.

For $a>0$, numerator and denominator have degree $N$, so the
controller is proper with feedthrough $1/a$. Its minimal order is
at most $N$. At any root $z$ of $B$,
$Q(z)=-I_\delta(z)=-e^{-\delta z}\ne0$, excluding cancellation
there. The only remaining possible common factor is $p+1$ when
$-1$ is not already a root of $B$. Thus $Q(-1)\ne0$ suffices for
minimal order exactly $N$. This will hold throughout the explicit
all-order family and throughout the small-delay existence interval.

The characteristic expression has the exact factorization
\begin{gather}
 Q(p)+e^{-\delta p}=B(p)R(p),\\
 R(p)=ap+c+\varepsilon_N\widehat J(p)
     =1+ap-\varepsilon_Np\widehat K(p).
 \label{eq:n_R}
\end{gather}
Consequently the reference-to-output transfer and the step-command
transforms are
\begin{gather}
 T_y(p)=\frac{e^{-\delta p}}{B(p)R(p)},\qquad
 W(p)=\frac{1+p}{pR(p)},\\
 V(p)=\frac{1-a+\varepsilon_N\widehat K(p)}{R(p)}.
 \label{eq:n_closed}
\end{gather}
The remaining question is whether the inverse of $R$ is stable and
nonnegative, and whether its combination in $V$ has the required sign.
The two parities of $N$ lead to different sufficient conditions.

\begin{theorem}[Sufficient conditions for global attainment]
\label{thm:n_construction}
Let $N\ge2$, $0<a<1$, and let the controller be
\eqref{eq:n_controller}. Assume the corresponding condition below:
\begin{align}
 N\text{ odd}:&\quad
 \exists\lambda>0:\quad a\lambda\ge c+\Phi(\lambda);
 \label{eq:n_odd}\\
 N\text{ even}:&\quad
 \gamma=c/a,\qquad
 \int_0^\delta e^{\gamma t}K(t)\dd t\le1-a.
 \label{eq:n_even}
\end{align}
The physical closed loop is internally exponentially stable, its
centered command is nonnegative, and its output is nondecreasing
from zero to one. It attains
\begin{equation}
 \Mcal_\alpha^{\min}=S\sqrt{\frac{\alpha}{1+\alpha}},\qquad
 A_e=A_e^*,\quad A_u=A_u^*.
 \label{eq:n_attainment}
\end{equation}
This is a global minimum over all admissible controllers satisfying
the signal assumptions, not only over the constructed rational family.
\end{theorem}
\begin{proof}
For odd $N$, $\varepsilon_N=1$. Write
\begin{equation}
 \frac1R=\frac1{a(p+\lambda)}\frac1{1-\widehat k_0(p)},\qquad
 \widehat k_0(p)=\frac{\lambda-c/a-\widehat J(p)/a}{p+\lambda}.
\end{equation}
Its convolution kernel is
\begin{equation}
 k_0(t)=e^{-\lambda t}
 \left[\lambda-\frac ca-\frac1a\int_0^t e^{\lambda r}J(r)\dd r\right].
 \label{eq:n_odd_kernel}
\end{equation}
Condition \eqref{eq:n_odd} makes the bracket nonnegative for every
$t\ge0$. Since $c+j_0=1$,
\begin{equation}
 \|k_0\|_1=\widehat k_0(0)=1-\frac1{a\lambda}<1.
 \label{eq:n_kernel_mass}
\end{equation}
Nonnegativity also implies $a\lambda\ge1$. The convolution Neumann
series therefore converges and gives
\begin{equation}
 r(t):=\mathcal L^{-1}\{1/R\}(t)
 =\frac{e^{-\lambda t}}a*
 \left(\delta_0+\sum_{j=1}^{\infty}k_0^{*j}\right)(t)\ge0.
 \label{eq:n_resolvent}
\end{equation}
Here $\delta_0$ is the identity measure for convolution; it does not
represent an impulse in the physical command. Such resolvent expansions
are standard in Volterra theory \citep{Gripenberg1990}.

The kernel $J$ has finite support and $k_0$ has an exponentially
decaying tail. By continuity of its exponential moment, there is
$\rho>0$, smaller than $\lambda$, for which
$\int e^{\rho t}k_0(t)\dd t<1$. The same geometric series converges
in exponentially weighted $L^1$. It also bounds
$\sup_{t\ge0}e^{\rho t}|r(t)|$, because the initial kernel
$e^{-\lambda t}/a$ has bounded weighted supremum and convolution
is bounded by the product of this norm and the weighted $L^1$ norm.
Thus $r$ decays exponentially. From \eqref{eq:n_closed},
\begin{equation}
 v=(1-a)r+K*r\ge0.
 \label{eq:n_v_odd}
\end{equation}

For even $N$, $\varepsilon_N=-1$ and $c=1+j_0>0$. Now
\begin{equation}
 \frac1R=\frac1{ap+c}\frac1{1-\widehat J/(ap+c)}.
\end{equation}
The base kernel $g(t)=e^{-\gamma t}/a$ is nonnegative, and the
feedback kernel $g*J$ has mass $j_0/c<1$. The same weighted-series
argument proves nonnegative exponentially decaying $r$. Its equation
for $t>0$ is
\begin{equation}
 r'(t)=-\gamma r(t)+\frac1a(J*r)(t)\ge-\gamma r(t),
 \qquad r(0^+)=1/a.
\end{equation}
Integration between $t-s$ and $t$ gives
$r(t-s)\le e^{\gamma s}r(t)$ for $0\le s\le t$. Hence
\begin{align}
 v(t)&=(1-a)r(t)-(K*r)(t)\notag\\
 &\ge\left[1-a-\int_0^\delta e^{\gamma s}K(s)\dd s\right]r(t)
 \ge0.
 \label{eq:n_v_even}
\end{align}
Interchanging the two integrals gives a convenient evaluation formula,
\begin{equation}
 \int_0^\delta e^{\gamma t}K(t)\dd t
 =\frac{\Phi(\gamma)-j_0}{\gamma}.
 \label{eq:n_weighted_K}
\end{equation}

In either parity the impulse response of $T_y$ is the delayed
nonnegative convolution $h_B*r$. Its mass is $T_y(0)=1$, since
$B(0)=R(0)=1$. Therefore the step output increases monotonically
to one and $e\ge0$.

To check internal stability, keep the plant and controller states
before cancelling factors. Their characteristic determinant, up to
a nonzero constant, is
\begin{equation}
 A(p)Q(p)+A(p)e^{-\delta p}=A(p)B(p)R(p).
 \label{eq:n_full_characteristic}
\end{equation}
All roots of $A$ and $B$ lie in $\operatorname{Re}p\le-1$.
For a sufficiently small positive $\rho$, the weighted kernel norm
used above is less than one throughout $\operatorname{Re}p\ge-\rho$,
which makes $R$ zero-free there. The unreduced state equation is
retarded: the nondelayed characteristic polynomial has degree $2N$
and the polynomial multiplying the delay has degree $N<2N$.
The resulting strictly negative spectral bound implies exponential
stability for this finite-dimensional retarded realization
\citep{HaleLunel1993}. The stable cancelled modes are thus included;
positivity of the reference response alone was not used to infer
internal stability.

Finally, \eqref{eq:n_Qprime} and the simple integrator give
$\int e=\alpha S/(1+\alpha)$. The area balance gives
$\int v=S/(1+\alpha)$. Both signals have been proved nonnegative,
so their signed and absolute areas coincide. Equality in the universal
bound proves \eqref{eq:n_attainment} and global optimality.
\end{proof}

Conditions \eqref{eq:n_odd} and \eqref{eq:n_even} are sufficient.
They are not asserted to describe every attaining controller or the
largest possible delay interval.

\subsection{Existence for arbitrary positive time constants}
\label{sec:n_small_delay}

\begin{corollary}[An attaining order-$N$ controller for small delays]
\label{cor:n_small_delay}
Fix $N\ge2$, the positive time constants in \eqref{eq:n_physical},
and a weight $\alpha>\mu-1$. There is $\delta_0>0$ such that for
every $0\le\delta<\delta_0$, \eqref{eq:n_controller} has minimal
order exactly $N$ and satisfies the conclusions of
\cref{thm:n_construction}.
\end{corollary}
\begin{proof}
At $\delta=0$, $J=0$, $I_0=1$, $c=1$, and
\begin{equation}
 U_0=\frac\mu{1+\alpha}\in(0,1),\qquad
 a_0=1-U_0\in(0,1),\qquad R(p)=1+a_0p.
\end{equation}
For odd $N$, choose a fixed $\lambda>1/a_0$. Then
$a_0\lambda>1=c+\Phi(\lambda)$ at zero delay. The coefficients
and moments depend continuously on $\delta$, so the strict inequality
and $0<a<1$ persist for all sufficiently small positive delays.
For even $N$, the left side of \eqref{eq:n_even} is zero at
$\delta=0$, whereas its right side is $1-a_0=U_0>0$.
The same continuity argument applies.

At zero delay,
\begin{equation}
 Q(-1)=B(-1)U_0-1<0,
 \qquad B(-1)=\prod_{j=2}^N(1-\sigma_j)\in[0,1).
\end{equation}
This strict inequality also persists. Since cancellation at the other
plant roots was already excluded, the controller keeps minimal order
exactly $N$. Taking the smallest of these positive neighborhoods
provides $\delta_0$.
\end{proof}

The theorem covers unequal as well as repeated time constants, but does
not give one common delay limit for every cascade. Very small
$\sigma_j$ can make the continued exponential kernel grow rapidly
with $\delta$. The moment conditions provide the object-specific check.
At zero delay the construction reduces to the transparent formulas
\begin{equation}
 T_y(p)=\frac1{B(p)(1+a_0p)},\qquad
 w(\theta)=1+\left(\frac1{a_0}-1\right)e^{-\theta/a_0}.
\end{equation}

\subsection{One explicit family valid for every order}
\label{sec:n_all_orders}

When $\tau_j=\tau$ for all $j$, the interpolation and kernel become
\begin{gather}
 B(p)=(1+p)^{N-1},\qquad \varepsilon_N=(-1)^{N-1},\\
 I_\delta(p)=e^\delta\sum_{j=0}^{N-2}
          \frac{[-\delta(p+1)]^j}{j!},\qquad
 J(t)=\frac{e^{\delta-t}(\delta-t)^{N-2}}{(N-2)!}\quad(0\le t\le\delta).
 \label{eq:n_equal_kernel}
\end{gather}
Although $I_\delta$ is a truncated Taylor polynomial about $p=-1$,
the remainder in \eqref{eq:n_remainder} is still retained exactly.

\begin{theorem}[Uniform construction for every $N\ge2$]
\label{thm:n_all_orders}
For
\begin{equation}
 G_L(s)=\frac{k e^{-Ls}}{(1+\tau s)^N},\qquad
 N\ge2,\quad \alpha=2N-1,\quad 0\le\delta=L/\tau\le\frac14,
 \label{eq:n_all_orders_conditions}
\end{equation}
use \eqref{eq:n_equal_kernel} and set
\begin{gather}
 a=1-\frac{N+\delta}{2N}+\varepsilon_Nj_1,
 \qquad c=1-\varepsilon_Nj_0,\\
 \boxed{C_N(s)=\frac1k\,
 \frac{(1+\tau s)^N}
 {(1+\tau s)^{N-1}(a\tau s+c)-I_\delta(\tau s)}.}
 \label{eq:n_explicit_controller}
\end{gather}
This controller has minimal order exactly $N$, is internally
stabilizing, and attains
\begin{gather}
 \boxed{\Mcal_{2N-1}^{\min}
       =(N+\delta)\sqrt{\frac{2N-1}{2N}},}\\
 A_e=\frac{(2N-1)(N+\delta)}{2N},\qquad
 A_u=\frac{N+\delta}{2N}.
 \label{eq:n_explicit_minimum}
\end{gather}
Its output is nondecreasing from zero to one.
\end{theorem}
\begin{proof}
We verify the sufficient conditions simultaneously for every order.
For $m=N-1$, direct integration bounds give
\begin{equation}
 j_0\le\frac{e^\delta\delta^m}{m!},\qquad
 j_1\le\frac{e^\delta\delta^{m+1}}{(m+1)!}.
\end{equation}
For the second bound, substitute $t=\delta u$ and use
$\int_0^1u(1-u)^{m-1}\dd u=1/[m(m+1)]$.
At $\delta\le1/4$, the elementary inequalities
$e^{1/4}<4/3$ and $e<3$ imply, for every $m\ge1$,
\begin{equation}
 j_0\le\frac13,\qquad j_1\le\frac1{24}.
 \label{eq:n_uniform_moments}
\end{equation}
The exponential inequalities follow by comparison with geometric
series; for the second one use $n!\ge2^{n-1}$ for $n\ge1$,
with strict inequality for some terms.

For even $N$, $\varepsilon_N=-1$ and
\begin{equation}
 U_*\le\frac9{16},\qquad
 \frac{19}{48}\le a=1-U_*-j_1\le\frac12,
 \qquad c\le\frac43.
\end{equation}
Thus $\gamma=c/a\le64/19$ and $\gamma\delta\le16/19<1$.
Since $\int K=j_1$,
\begin{equation}
 \int_0^\delta e^{\gamma t}K(t)\dd t
 \le e^{\gamma\delta}j_1<\frac18<\frac12\le1-a.
\end{equation}
This proves \eqref{eq:n_even} with a strict positive margin.

For odd $N\ge3$, $m\ge2$ improves the first bound to $j_0\le1/24$.
Also,
\begin{equation}
 U_*\le\frac{13}{24},\qquad
 \frac{11}{24}\le a=1-U_*+j_1\le\frac{13}{24}.
\end{equation}
Choose $\lambda=3$. Then
\begin{equation}
 c+\Phi(3)=1+\Phi(3)-j_0
 \le1+(e^{3\delta}-1)j_0
 <\frac{13}{12}<\frac{11}{8}\le3a.
\end{equation}
This proves \eqref{eq:n_odd}, and both cases have $0<a<1$.
Finally, $A$ has only the root $-1$ and
$Q(-1)=-I_\delta(-1)=-e^\delta\ne0$. Hence the minimal controller
order is exactly $N$. Apply \cref{thm:n_construction} with
$S=N+\delta$ to obtain the minimum and both areas.
\end{proof}

The weight $2N-1$ gives an explicit sufficient family. It is not the
only admissible weight, nor is $\delta\le1/4$ asserted to be the
largest attainable delay interval. For $N=2$, $\alpha=3$ and
$\delta=1/4$, the formula is especially simple:
\begin{gather}
 a=\frac{27}{16}-e^{1/4}>0,\quad c=e^{1/4},\qquad
 \overline C_2(p)=
 \frac{(1+p)^2}{p[(27/16-e^{1/4})p+27/16]},\\
 A_e=\frac{27}{16},\qquad A_u=\frac9{16},\qquad
 \Mcal_3^{\min}=\frac{9\sqrt3}{8}.
 \label{eq:n_second_order_example}
\end{gather}
This example can be checked directly without solving an interpolation
system. \Cref{fig:n_responses} shows four orders from the same family.
The proof covers every integer $N\ge2$ and every delay in the stated
interval; the plotted cases only illustrate the resulting waveforms.

\begin{figure}[htbp]
 \centering\includegraphics[width=\linewidth]{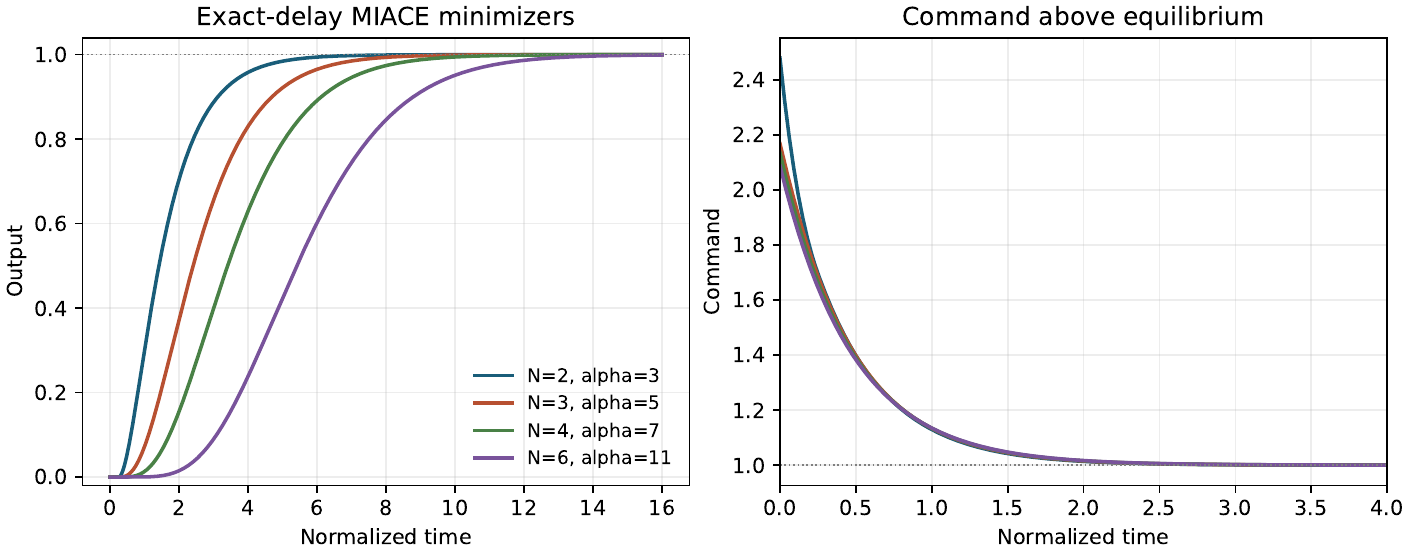}
 \caption{Attaining controllers for equal-lag plants of orders
 $N=2,3,4,6$, with $\delta=1/4$ and $\alpha=2N-1$.
 The output is monotone and the centered command remains nonnegative.
 The exact values are given by \eqref{eq:n_explicit_minimum};
 the curves are obtained from the full physical delayed loop.}

 \label{fig:n_responses}
\end{figure}

\subsection{Smoothness and the boundary of the plant class}
\label{sec:n_scope}

The relative degree of \eqref{eq:n_physical} is $N$. The constructed
proper controller has finite feedthrough $1/a$, so $w(0^+)=1/a$ is
a finite jump rather than an impulse. Its locally bounded command,
filtered through $N$ lags, gives
\begin{equation}
 y\in C^{N-1},\qquad
 y(\delta+\xi)=\frac{\xi^N}{aN!\prod_{j=1}^N\sigma_j}
                   +O(\xi^{N+1})\quad(\xi\downarrow0).
 \label{eq:n_smoothness}
\end{equation}
Indeed, the plant impulse begins as
$h_A(\xi)=\xi^{N-1}/[(N-1)!\prod\sigma_j]+O(\xi^N)$,
and convolution with $w(\xi)=1/a+O(\xi)$ gives the displayed
coefficient. The first $N-1$ output derivatives match the zero waiting
segment continuously. The $N$th derivative can jump. Monotonicity
excludes output reversals; $v\ge0$ only keeps the command above
equilibrium and does not assert that the command itself is monotone.

Plant order alone cannot imply these conclusions. A real right-half-plane
zero is a classical obstruction to monotonic tracking
\citep{Darbha2003,Taghavian2026}. For example, take
\begin{equation}
 G_L(s)=\frac{(1-s/z)e^{-Ls}}{(1+\tau_1s)(1+\tau_2s)},
 \qquad z>0,\quad L>0.
\end{equation}
An internally stable proper unity-feedback controller cannot cancel
this zero with a hidden unstable pole, so $T_y(z)=0$. A monotone
unit step would have $y'\ge0$, nonzero and of mass one, implying
\begin{equation}
 T_y(z)=\int_0^\infty e^{-zt}y'(t)\dd t>0,
\end{equation}
a contradiction. The delay does not remove the zero. This argument
is included to delimit the theorem, not as a new zero-location result.

Even minimum phase alone does not make MIACE equality imply
monotonicity. The distinction between nonovershooting and monotone
responses is classical \citep{LinFang1997}; the following exact
example shows it at the MIACE global minimum. The example has zero
delay, and makes no claim about persistence at positive delay. Set
\begin{gather}
 a_* =\frac{209}{260},\qquad
 F(p)=p^2+\frac{19}{5}p+\frac{418}{5},\\
 D_0(p)=[(p+2)^2+100][(1+a_*)p+1],\qquad
 B_0(p)=\frac65p^2+\frac{122}{5}p+104,\\
 H_3(p)=\frac{B_0(p)}{D_0(p)},\qquad
 \overline C_3(p)=\frac{D_0(p)}{pF(p)}.
 \label{eq:n_counterexample}
\end{gather}
The plant is stable, minimum phase, and of order three. The proper
controller also has order three. Before cancellation the characteristic
polynomial is
\begin{equation}
 D_0(p)[pF(p)+B_0(p)]=D_0(p)(p+1)[(p+2)^2+100],
\end{equation}
so all internal modes are stable. Direct transformation gives
\begin{equation}
 e(t)=\frac45e^{-t}+\frac15e^{-2t}\cos10t,\qquad
 v(t)=a_*e^{-t}.
 \label{eq:n_counterexample_signals}
\end{equation}
Both are positive, with $e(t)\ge(3/5)e^{-t}$, and
\begin{equation}
 A_e=\frac45+\frac15\frac2{4+100}=a_*=A_u,
 \qquad \mu=2a_*.
\end{equation}
Thus at $\alpha=1$ the universal bound is attained exactly:
$\Mcal_1=\sqrt2a_*=\mu/\sqrt2$. But at $t_0=3\pi/20$,
\begin{equation}
 e'(t_0)=-\frac45e^{-t_0}+2e^{-2t_0}>0,
\end{equation}
and the output moves backwards. The equality signs prevent overshoot;
the additional cascade structure was what prevented backward motion
in \cref{thm:n_monotonic}.

\begin{figure}[htbp]
 \centering\includegraphics[width=\linewidth]{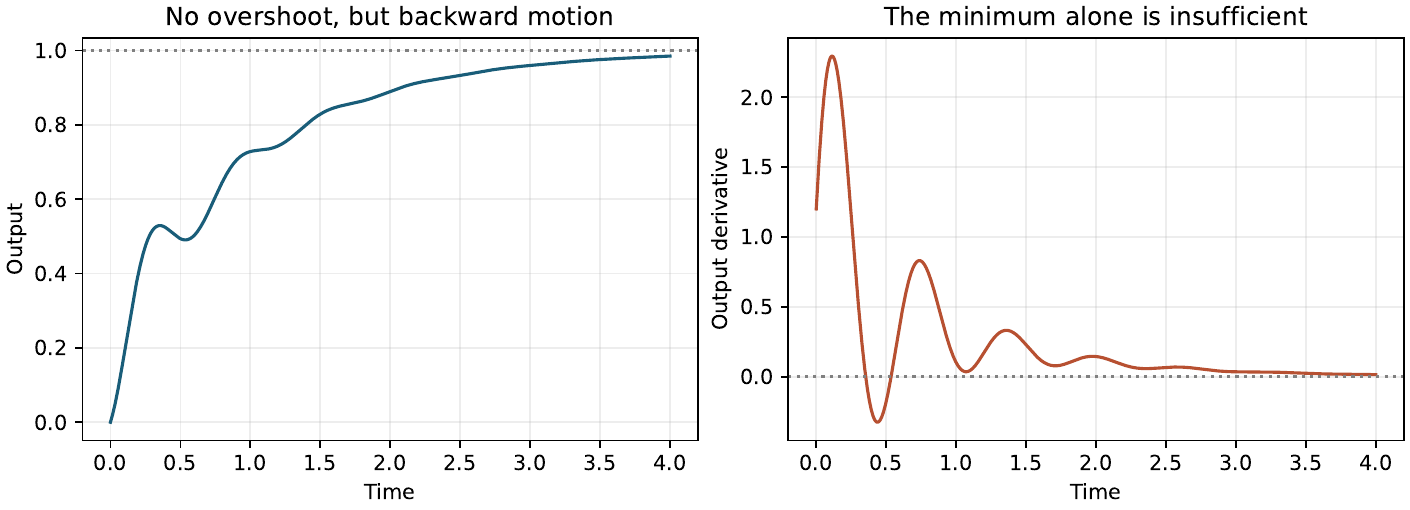}
 \caption{The exact minimum-phase third-order counterexample
 \eqref{eq:n_counterexample}. MIACE attains its lower bound and the
 output never exceeds one, but its derivative is negative on some
 intervals. The assertions follow from the explicit signals, not
 from a sampled sign test.}

 \label{fig:n_counterexample}
\end{figure}

The higher-order optimum is nominal. It uses known parameters and
stable pole cancellation. None of the preceding equalities establishes
the same optimality under parameter errors, uncertain delay, saturation,
measurement noise, or sampling. These limitations concern implementation
and robustness; they do not alter the exact attainment proof for the
stated plant and controller.

%% file: sections/06f_complete_attainment.tex
\section{Complete attainment with an internal model and a sharp command cap}
\label{sec:complete_attainment}

The plant’s pure delay has been retained exactly throughout the preceding analysis. The controllers constructed so far were restricted to proper rational transfer functions, without delay elements inside the controller. We now allow an internal model containing a copy of the plant’s delay, while retaining the same physical plant and tracking problem. In this enlarged controller class, the necessary condition of \cref{thm:n_necessary} becomes sufficient for attaining the MIACE lower bound.

Use the normalization in \cref{sec:order_n}, now also allowing $N=1$:
\begin{equation}
 P(p)=\frac{e^{-\delta p}}{(1+p)B(p)},\qquad
 B(p)=\prod_{j=2}^N(1+\sigma_jp),\qquad
 S=\delta+1+\sum_{j=2}^N\sigma_j.
 \label{eq:complete_plant}
\end{equation}
An empty product is one. The physical gain is removed by $w=ku$.
All responses start from rest with zero command prehistory. Commands are
ordinary functions, with finite jumps allowed. An internal model below
contains the exact causal transport delay; it never advances a signal.

\subsection{An exact exponential-moment characterization}

\begin{theorem}[The nonovershooting command cone]
\label{thm:command_cone}
Let $v=w-1\ge0$ be an admissible locally integrable command. For the
positive cascade \eqref{eq:complete_plant}, the following are equivalent:
\begin{enumerate}[label=(\roman*)]
 \item $y(\theta)\le1$ for every $\theta\ge0$;
 \item $y$ is nondecreasing from zero to one;
 \item $\displaystyle J_v:=\int_0^\infty e^\theta v(\theta)\dd\theta\le1$.
\end{enumerate}
For every such command,
\begin{equation}
 U:=A_u=\int_0^\infty v\dd\theta<1,
 \qquad A_e=S-U.
 \label{eq:complete_front}
\end{equation}
Every $0\le U<1$ is realized by an internally stable causal feedback
controller, if internal delays are allowed.
\end{theorem}
\begin{proof}
The slowest-pole tail and Fatou argument in
\cref{thm:n_necessary} proves (i)$\Rightarrow$(iii) without using
the area-ratio equality. It includes repeated slowest poles. To prove
the converse, place the slowest lag first in an equivalent cascade
realization, and temporarily remove the final transport delay. Its
first state satisfies
\begin{equation}
 x_1(\theta)=1-e^{-\theta}
       +e^{-\theta}\int_0^\theta e^r v(r)\dd r.
 \label{eq:complete_first_state}
\end{equation}
If $J_v\le1$, then $0\le x_1\le1$ and
$x_1'=1+v-x_1\ge0$ almost everywhere. Each remaining stable lag
preserves a nondecreasing input and its bounds. Restoring the delay
therefore proves (ii), which implies (i). This statement about
intermediate states uses the reordered realization; it is not a claim
about every ordering of the physical storage elements.

The signed-area identity gives $A_e=S-U$. If $U>0$, then
$U<J_v\le1$, because an ordinary function cannot put positive area
only at zero. For $U=0$ the strict upper bound is immediate.
For any chosen $0<U<1$, put $a=1-U$ and take
\begin{equation}
 v_a(\theta)=\frac{1-a}{a}e^{-\theta/a},\qquad 0<a<1.
 \label{eq:complete_v}
\end{equation}
Its area is $1-a=U$ and $J_{v_a}=1$. The case $U=0$ uses $v=0$.
The feedback realization and its internal stability are proved next.
\end{proof}

Thus the exponential moment previously used only as an obstruction
is also sufficient. After normalization by the slowest time constant,
neither the other lag ratios nor the transport delay changes this
command cone. Those parameters still change $S$ and the required
area ratio. In physical time its condition is
\begin{equation}
 \frac1{\tau_*}\int_0^\infty e^{t/\tau_*}[ku(t)-1]\dd t\le1.
 \label{eq:complete_moment_dimensional}
\end{equation}

\subsection{Closing the necessary--sufficient gap}

For a stable causal filter $F$ with $F(0)=1$, consider the familiar
internal-model realization
\begin{equation}
 \widehat y=P w,\qquad w=F(r-y+\widehat y).
 \label{eq:complete_imc}
\end{equation}
Here $r$ is the reference, not the PI integral-time ratio used in
earlier sections. Since $P$ is strictly proper, this interconnection
has no algebraic loop. Its equivalent error-feedback controller and
nominal closed-loop maps are
\begin{equation}
 \overline C_F=\frac{F}{1-PF},\qquad
 \frac{W}{R}=F,\qquad \frac{Y}{R}=PF.
 \label{eq:complete_imc_maps}
\end{equation}
The architecture and this parameterization are established internal
model control (IMC), not a new predictor principle
\citep{GarciaMorari1982,Morari1983}. What is determined here is the
filter that meets the MIACE equality conditions and its exact feasibility
range. The sign and peak results below do not follow from nominal
stability of IMC alone.

\begin{theorem}[Exact attainment when internal delays are allowed]
\label{thm:complete_attainment}
For \eqref{eq:complete_plant} and $\alpha>0$, an internally stable
causal controller with an ordinary command attains $S m_\alpha$
if and only if
\begin{equation}
 \boxed{\alpha>S-1.}
 \label{eq:complete_iff}
\end{equation}
A sufficient controller, for every delay satisfying this inequality, is
\begin{gather}
 U_* =\frac{S}{1+\alpha},\qquad a=1-U_*\in(0,1),\qquad
 F_a(p)=\frac{1+p}{1+ap},\label{eq:complete_filter}\\
 \boxed{\overline C_a(p)=
 \frac{(1+p)B(p)}{(1+ap)B(p)-e^{-\delta p}}.}
 \label{eq:complete_controller}
\end{gather}
Its output transfer, areas, and initial command peak are
\begin{gather}
 T_a(p)=\frac{e^{-\delta p}}{B(p)(1+ap)},\qquad
 A_e=\frac{\alpha S}{1+\alpha},\qquad A_u=\frac{S}{1+\alpha},
 \label{eq:complete_areas}\\
 \sup_{\theta\ge0}w(\theta)=\frac1a.
 \label{eq:complete_peak}
\end{gather}
The physical controller is $C_a(s)=\overline C_a(\tau_*s)/k$.
\end{theorem}
\begin{proof}
Necessity is \cref{thm:n_necessary}. Under the strict inequality,
$0<U_*<1$. The unit-step command of $F_a$ is
$w=1+(1/a-1)e^{-\theta/a}$. Its centered part is
\eqref{eq:complete_v}, so \cref{thm:command_cone} proves the signal
signs and monotonicity. Alternatively, $T_a$ is a delayed cascade
of positive lags and has unit DC gain. The signed-area identity gives
\eqref{eq:complete_areas}; its prescribed ratio proves global attainment
over all admissible controllers, not only over the displayed family.

Internal stability must be checked in the realization, not inferred
from the apparent cancellation in $T_a$. Both the actual plant and
its model are stable cascades with finite delays. At exact parameters
their state difference evolves autonomously and decays exponentially;
any initial delay-line history disappears after a finite interval.
The input of the stable proper filter $F_a$ is $r$ plus this decaying
difference. The command, plant states, model states, and filter state
therefore have exponentially decaying homogeneous responses. The same
argument with additive disturbances uses only stable cascades of
$P$ and $F_a$. In transfer form the four standard internal maps are
$1-PF_a$, $F_a$, $P(1-PF_a)$, and $PF_a$, all stable and causal.
This establishes internal exponential stability while retaining the
model states. Finally $a>0$ makes the command jump finite and gives
\eqref{eq:complete_peak}.
\end{proof}

Expanding the controller at zero gives
$\overline C_a(p)=[(S-U_*)p]^{-1}+O(1)$, in agreement with
\eqref{eq:n_integral_gain}. At zero delay it reduces to a rational
controller already contained in the earlier construction. For positive
delay it is generally nonrational and includes a delay block. Therefore
\cref{thm:complete_attainment} does not enlarge the proved range for
PI, filtered PID, or fixed-order rational controllers without delays.
For identical lags it gives the exact condition
$L/\tau<\alpha-N+1$ in the larger causal class.

The strict boundary also has a concrete meaning. As $U_*\uparrow1$,
$a\downarrow0$ and the command peak diverges. The limiting centered
command is a unit impulse, which is excluded. Within the sign-preserving
class $e,v\ge0$, the full achievable area segment is exactly
\eqref{eq:complete_front}; its endpoint $U=1$ belongs only to its closure.
When $\alpha\le S-1$, the infimum in this sign-preserving class is
$\sqrt{(S-1)^2+\alpha}$ and is not attained by an ordinary command.
This last statement is not a solution of the unrestricted problem,
which permits negative error or command-deficit areas.

\subsection{The least command peak and the exact capped front}

\begin{theorem}[Sharp amplitude-cap criterion]
\label{thm:complete_cap}
Impose $1\le w\le M$ almost everywhere, with $M>1$, and write
\begin{equation}
 c=M-1,\qquad U_M=c\log(1+1/c).
 \label{eq:complete_capacity}
\end{equation}
Among responses with $e,v\ge0$, the attainable command areas are
exactly $0\le U\le U_M$. Consequently the universal MIACE bound
is attainable under this cap if and only if
\begin{equation}
 \boxed{\frac{S}{1+\alpha}\le U_M,\qquad\text{or}\qquad
 \alpha\ge\frac{S}{U_M}-1.}
 \label{eq:complete_cap_iff}
\end{equation}
For every feasible $U$, a realizing command is
\begin{equation}
 v(\theta)=c\,\mathbf1_{[0,T)}(\theta),\qquad T=U/c.
 \label{eq:complete_pulse}
\end{equation}
It is generated by \eqref{eq:complete_imc} with
\begin{equation}
 F_{c,T}(p)=1+c-ce^{-Tp},\qquad
 \overline C_{c,T}(p)=\frac{F_{c,T}(p)}{1-P(p)F_{c,T}(p)}.
 \label{eq:complete_cap_controller}
\end{equation}
This is an internally stable causal feedback realization with finite
command jumps and finite delay blocks. At $U=U_M$, the command
\eqref{eq:complete_pulse} is unique almost everywhere among commands
having that area and satisfying both the cap and nonovershoot.
\end{theorem}
\begin{proof}
Fix an area $U$ and let $v_*=c\mathbf1_{[0,U/c)}$.
For $0\le v\le c$ with the same area, put $T=U/c$. Then
\begin{equation}
 \int_0^\infty e^\theta(v-v_*)\dd\theta
 =\int_0^\infty(e^\theta-e^T)(v-v_*)\dd\theta\ge0.
 \label{eq:complete_rearrangement}
\end{equation}
The integrand is nonnegative on both sides of $T$: before $T$ both
factors are nonpositive; afterwards both are nonnegative. Hence
\begin{equation}
 J_v\ge c(e^{U/c}-1).
 \label{eq:complete_min_moment}
\end{equation}
By \cref{thm:command_cone}, nonovershoot requires this quantity
to be at most one, which is precisely $U\le U_M$.
The rectangular command realizes equality in this moment estimate
and is feasible for every such $U$. Its transfer from reference to
command is $F_{c,T}$, a stable finite-delay filter. The internal-stability
argument for \eqref{eq:complete_imc} applies unchanged. At the boundary,
$J_v\le1=c(e^{U_M/c}-1)$ forces equality in
\eqref{eq:complete_rearrangement}. Strict increase of $e^\theta$
then forces $v=v_*$ almost everywhere.
\end{proof}

This is a design that respects a cap, not a nonlinear saturation or
anti-windup theorem. The cap is proved for the nominal unit-step
experiment. The initial and terminal command jumps do not satisfy
a finite slew-rate constraint.

For a fixed equality area $0<U_*<1$, the smallest possible command
peak is $M_{\min}=1+c_*$, where
\begin{equation}
 \boxed{c_*\log(1+1/c_*)=U_*.}
 \label{eq:complete_min_peak}
\end{equation}
There is exactly one positive solution: the left side increases
continuously from zero to one, since
$\log(1+1/c)-1/(1+c)>0$. The unique command at this least peak
is \eqref{eq:complete_pulse} with $c=c_*$.
For $U_*=1/2$, $c_*=0.3979525473\ldots$, so the least peak is
$1.3979525473\ldots$; the exponential design has peak $2$.
The two designs have exactly the same MIACE and the same two areas.

More generally, the exact minimum over the capped sign-preserving
class is
\begin{equation}
 \widehat U=\min\left\{\frac{S}{1+\alpha},U_M\right\},\qquad
 \boxed{\Mcal_{\alpha,M}^{+}
 =\sqrt{(S-\widehat U)^2+\alpha\widehat U^2}.}
 \label{eq:complete_cap_optimum}
\end{equation}
This follows by minimizing a strictly convex quadratic on
$[0,U_M]$. If \eqref{eq:complete_cap_iff} holds, it is also the
global optimum over all controllers respecting the cap, because
the universal bound is reached. Otherwise global optimality has
been established only in the stated sign-preserving class.

\subsection{Stability and performance bounds under a small delay error}

Exact model matching is nominal. A perturbation estimate is useful
because the controller remains fixed when the physical delay changes.
For a causal stable convolution operator $K$, write $\|K\|_{\rm tv}$
for the total variation of its impulse measure, including any point
masses. This is its induced $L^1$ norm. The following result concerns
the two explicit attaining filters above, with $F=F_a$ or $F_{c,T}$.

\begin{theorem}[Delay-error stability and MIACE deviation]
\label{thm:complete_robust}
Keep the nominal internal model $P=e^{-\delta p}H$ in
\eqref{eq:complete_imc}, but let the actual plant be
$P_\varepsilon=e^{-(\delta+\varepsilon)p}H$, where
$\delta+\varepsilon\ge0$. Set $d=|\varepsilon|$ and
\begin{equation}
 f=\|F\|_{\rm tv}=
 \begin{cases}2/a-1,&F=F_a,\\1+2c,&F=F_{c,T},\ T>0,
 \end{cases}
 \qquad q_d=2f(1-e^{-d}).
 \label{eq:complete_robust_q}
\end{equation}
If $q_d<1$, the perturbed realization is internally exponentially
stable. For the same unit step, with nominal signals denoted by a
subscript zero,
\begin{gather}
 \|w_\varepsilon-w_0\|_1\le\frac{f^2d}{1-q_d},\qquad
 \|e_\varepsilon-e_0\|_1\le\frac{2fd}{1-q_d},
 \label{eq:complete_robust_signals}\\
 \boxed{|\Mcal_{\alpha,\varepsilon}-\Mcal_{\alpha,0}|
 \le\frac{fd\sqrt{4+\alpha f^2}}{1-q_d}.}
 \label{eq:complete_robust_cost}
\end{gather}
The result establishes stability and bounds the performance deviation;
it does not establish nonovershoot, the original cap, or exact equality
after the delay changes.
\end{theorem}
\begin{proof}
Let $\Delta=P_\varepsilon-P$. The actual command map is
$W/R=F/(1+F\Delta)$. Since $H=(1+p)^{-1}/B$, contraction by
the nonnegative unit-mass impulse response of $1/B$ gives
\begin{equation}
 \|\Delta\|_{\rm tv}\le2(1-e^{-d}),\qquad
 \|\Delta/p\|_1=d.
 \label{eq:complete_delay_norms}
\end{equation}
The first expression before contraction is the exact $L^1$ distance
between an exponential impulse response and its shift by $d$.
For the second, the two step responses are ordered because $H$
is positive. Integrating their difference over all time gives $d$
times the unit impulse mass, for either sign of $\varepsilon$.

Thus $\|F\Delta\|_{\rm tv}\le q_d<1$. The Neumann series
for $(1+F\Delta)^{-1}$ has norm at most $(1-q_d)^{-1}$.
All participating impulse measures have an exponential moment;
by continuity the same strict inequality holds with a sufficiently
small positive exponential weight. This proves exponential stability
of the inverse and of the full stable-block realization, including
initial-state responses.

For a step input the exact signal differences are
\begin{equation}
 W_\varepsilon-W_0=-\frac{F^2}{1+F\Delta}\frac\Delta p,
 \qquad
 E_\varepsilon-E_0=-\frac{F(1-PF)}{1+F\Delta}\frac\Delta p.
 \label{eq:complete_robust_transforms}
\end{equation}
The nominal transfer $PF$ has a nonnegative unit-mass impulse
response by \cref{thm:command_cone}, so $\|1-PF\|_{\rm tv}\le2$.
Convolution norm inequalities prove \eqref{eq:complete_robust_signals}.
The actual DC gain remains one, so the equilibrium-centered command
is still $w-1$. Finally, changes of absolute areas are bounded by
the corresponding signal $L^1$ differences. The reverse triangle
inequality for the weighted Euclidean norm proves
\eqref{eq:complete_robust_cost}.
\end{proof}

The stability condition is conservative and only sufficient. In particular,
the numerical response may remain satisfactory beyond its radius
$d<-\log(1-1/(2f))$. Its assumptions retain the exact lag constants
and plant gain. A gain error would change the command equilibrium
and requires recentering before any infinite-horizon area comparison.

\subsection{Independent response checks}
\label{sec:complete_checks}

The script \texttt{code/complete\_attainment\_audit.py} integrates the
physical lag cascade for the exponential and rectangular commands,
and compares it with a separate matrix-exponential response formula.
It checks orders one, two, three, and five, repeated and unequal
lags, and delays outside the earlier sufficient rational-construction
interval. A separate method-of-steps integration retains both the
plant and its internal model for nonzero delay errors. Delays are
time shifts, not Pad\'e approximations. Tolerances, finite horizons,
residuals, and the exact formulas used for comparison are saved with
the numerical results. These checks support the calculations; the
infinite-horizon statements rely on the proofs.

Across the ten nominal runs, the maximum discrepancy from the independent
response formula was $1.11\times10^{-10}$ and the largest tracking-area
discrepancy was $2.09\times10^{-11}$. The three delay-error runs use
$\varepsilon=-0.02,0.02,0.05$ in the three-lag example; their MIACE
values are respectively $4.7545220$, $4.7643377$, and $4.8092761$,
against nominal $\sqrt{22.5}=4.7434165$. All satisfy the conservative
bounds in \cref{thm:complete_robust}. Halving the maximum integration
step changes the last MIACE value by less than $10^{-7}$.

\begin{table}[H]
 \centering
 \caption{Nominal equality designs and the reduction in command peak.
 The unequal lag vectors are $(1,0.3)$ in row two and $(1,0.6,0.15)$
 in row four; the other rows have identical lags.}
 \input{data/complete_attainment_table}
 \label{tab:complete_attainment}
\end{table}

\begin{figure}[H]
 \centering
 \includegraphics[width=\linewidth]{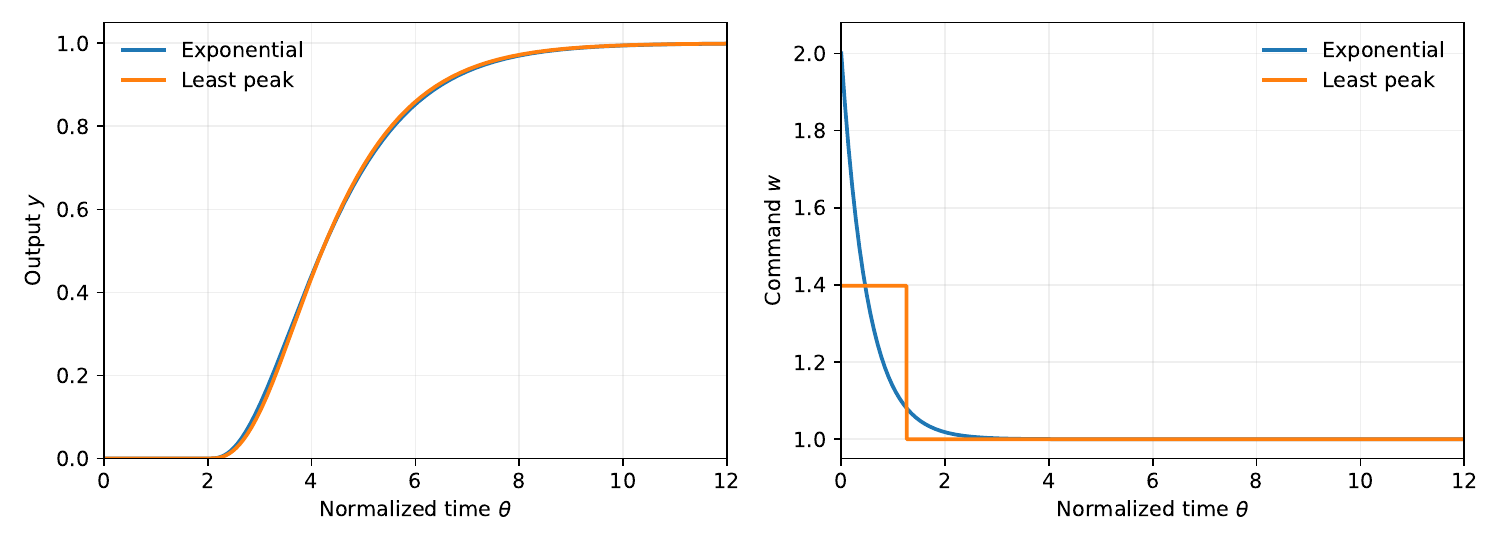}
 \caption{Three identical lags, $\delta=2$, and $\alpha=9$.
 Both designs attain $A_e=4.5$, $A_u=0.5$, and
 $\Mcal_9=\sqrt{22.5}$. The rectangular command has the least
 possible peak among all attaining commands. Internal delay blocks
 are allowed in both feedback realizations.}

 \label{fig:complete_attainment}
\end{figure}


%% file: data/complete_attainment_table.tex
\begin{tabular}{rrrrrr}
\toprule
$N$ & $\delta$ & $\alpha$ & $U_*$ & $\sup w_a$ & $M_{\min}$\\
\midrule
1 & 1.00 & 3 & 0.5 & 2.00000 & 1.39795\\
2 & 1.40 & 5 & 0.45 & 1.81818 & 1.31479\\
3 & 2.00 & 9 & 0.5 & 2.00000 & 1.39795\\
3 & 1.80 & 4.462 & 0.65 & 2.85714 & 1.80494\\
5 & 3.00 & 15 & 0.5 & 2.00000 & 1.39795\\
\bottomrule
\end{tabular}

%% file: sections/03_exact_energies.tex
\section{Quadratic comparison: exact tracking and command energies}
\label{sec:energies}

We first return to the delay-free plant, $L=0$. The MIACE results concern absolute areas. To see what changes when large
excursions receive a squared penalty, consider the centered quadratic
criterion
\begin{equation}
 J=Q\int_0^\infty e^2(t)\dd t+
 R\int_0^\infty\bigl(u(t)-1/k\bigr)^2\dd t,\qquad Q,R>0.
 \label{eq:physical_criterion}
\end{equation}
The weights have units that make the terms comparable. The stationary
PI formula for this criterion is given in
\citet[Sec.~2.
1, pp.~25--28]{Tomov2011}. We retain its derivation and
proof of global optimality as a comparison with MIACE.

Both $e$ and $v$ decay exponentially. Hence the energies
\begin{equation}
 \Ecal=\int_0^\infty e^2\dd\theta,\qquad
 \Ucal=\int_0^\infty v^2\dd\theta
 \label{eq:normalized_energies_definition}
\end{equation}
are finite, with physical scaling
\begin{equation}
 E_2=\tau\Ecal,\qquad U_2=\frac\tau{k^2}\Ucal.
 \label{eq:physical_scaling}
\end{equation}
Consequently,
\begin{equation}
 J=\frac{R\tau}{k^2}F_q,
 \qquad F_q=q\Ecal+\Ucal,
 \qquad q=\frac{Qk^2}{R}>0.
 \label{eq:dimensionless_weight}
\end{equation}
The two energies can be evaluated without separating real and complex poles.
We use a standard second-order energy formula, whose short Gramian derivation
is given in \cref{app:h2}. Here the squared $H^2$ norm of a stable strictly
proper transfer function means the integral of the square of its impulse
response.

\begin{lemma}[Second-order energy formula]
\label{lem:h2}
For $a,b,c>0$ and real $d,e_0$, let
$H(p)=(dp+e_0)/(ap^2+bp+c)$. Then
\begin{equation}
 \norm{H}_2^2=\frac{ae_0^2+cd^2}{2abc}.
 \label{eq:h2_lemma}
\end{equation}
\end{lemma}

\begin{proposition}[Exact PI energies]
\label{prop:energies}
For every $x,r>0$,
\begin{align}
 \Ecal(x,r)&=\frac{x+r}{2x(1+x)},
 \label{eq:tracking_energy}\\
 \Ucal(x,r)&=\frac{(x-r)^2+xr(x-1)^2}{2rx(1+x)}.
 \label{eq:effort_energy}
\end{align}
\end{proposition}
\begin{proof}
In \cref{lem:h2}, take $(a,b,c)=(r,r(1+x),x)$.
For $\widehat E$, use $(d,e_0)=(r,r)$. Its numerator in
\eqref{eq:h2_lemma} is $r^3+xr^2=r^2(x+r)$, and its denominator is
$2r^2x(1+x)$. For $\widehat V$, use
$(d,e_0)=(r(x-1),x-r)$. The numerator is
$r(x-r)^2+xr^2(x-1)^2$; cancellation of one factor $r$ gives
\eqref{eq:effort_energy}.
\end{proof}

The formulas remain regular at repeated poles. Moreover, $\Ucal=0$ requires
$x=r=1$. This controller applies the equilibrium command immediately and
keeps it constant for $t>0$. The plant then follows its natural first-order
response. Zero centered effort therefore does not mean zero command: the
initial command jump and the steady operating level are outside this
particular integral penalty.

For the optimization, collecting the terms gives
\begin{equation}
 2x(1+x)F_q=(1+q)r+\frac{x^2}{r}
              +x(q+x^2-2x-1).
 \label{eq:expanded_objective}
\end{equation}

%% file: sections/04_global_optimum.tex
\section{The global quadratic optimum}
\label{sec:global}

Before minimizing the rational expression, it helps to see why a simple
answer is possible. Multiply \eqref{eq:error_balance} by $e$ and integrate.
Since $e(0^+)=1$ and $e(\infty)=0$,
\begin{equation}
 -\frac12=-\Ecal-\int_0^\infty ev\dd\theta,
 \qquad \int_0^\infty ev\dd\theta=\frac12-\Ecal.
 \label{eq:cross_energy}
\end{equation}
Put $s_q=\sqrt{1+q}$. Expanding a square and using
$q=(s_q-1)^2+2(s_q-1)$ gives, term by term,
\begin{align*}
 \int_0^\infty[v-(s_q-1)e]^2\dd\theta
 &=\Ucal-2(s_q-1)(1/2-\Ecal)+(s_q-1)^2\Ecal\\
 &=\Ucal+\bigl[2(s_q-1)+(s_q-1)^2\bigr]\Ecal-(s_q-1)\\
 &=F_q-(s_q-1).
\end{align*}
Consequently,
\begin{equation}
 \boxed{F_q-(s_q-1)
 =\int_0^\infty\bigl[v-(s_q-1)e\bigr]^2\dd\theta.}
 \label{eq:signal_certificate}
\end{equation}
The cost cannot be smaller than $s_q-1$. Equality requires
$v=(s_q-1)e$, so the plant equation becomes $e'=-s_qe$.
Thus the optimal error is $e^{-s_q\theta}$. The following parameter identity
shows exactly which PI controller produces it and how much any other
positive PI controller loses.

\begin{theorem}[Unique global quadratic optimum]
\label{thm:global}
For $q>0$, $s_q=\sqrt{1+q}$, and every $x,r>0$,
\begin{equation}
 \boxed{F_q(x,r)-(s_q-1)
 =\frac{(s_qr-x)^2}{2rx(1+x)}
  +\frac{(x-s_q)^2}{2(1+x)}.}
 \label{eq:sos_certificate}
\end{equation}
The unique global minimizer and minimum are
\begin{equation}
 r^\star=1,\qquad x^\star=s_q,\qquad F_q^\star=s_q-1.
 \label{eq:global_normalized_solution}
\end{equation}
In physical variables,
\begin{equation}
 \boxed{T_i^\star=\tau,\qquad
 K_c^\star=\frac1k\sqrt{1+\frac{Qk^2}{R}},}
 \label{eq:physical_global_controller}
\end{equation}
and, for a unit step,
\begin{equation}
 J^\star=\frac{R\tau}{k^2}
       \left(\sqrt{1+\frac{Qk^2}{R}}-1\right).
 \label{eq:physical_global_value}
\end{equation}
\end{theorem}
\begin{proof}
In \eqref{eq:expanded_objective}, write
\[
 s_q^2r+\frac{x^2}{r}
 =\frac{(s_qr-x)^2}{r}+2s_qx.
\]
Subtract $2x(1+x)(s_q-1)$ from the numerator of
\eqref{eq:expanded_objective}. After removing $(s_qr-x)^2/r$, the
remaining numerator is
\[
 x\{2s_q+s_q^2-1+x^2-2x-1-2(1+x)(s_q-1)\}
 =x(x-s_q)^2.
\]
Dividing by $2x(1+x)$ proves
\eqref{eq:sos_certificate}. Both denominators are positive. The second
square vanishes only at $x=s_q$, and the first then vanishes only at $r=1$.
Equations \eqref{eq:physical_global_controller} and
\eqref{eq:physical_global_value} follow from the scaling of $F_q$.
\end{proof}

The controller zero cancels the stable plant pole at this optimum, but
cancellation has not been assumed in the minimization. The two squares give
an explicit measure of the loss caused by changing the gain and integral
time. For $k=\tau=Q=R=1$, they give $K_c^\star=\sqrt2$, $T_i^\star=1$, and
$J^\star=\sqrt2-1$.

Equation \eqref{eq:signal_certificate} is the scalar LQR completion of
squares for $e'=-e-v$. Its Riccati equation is $P^2+2P-q=0$, with
$P=s_q-1$ \citep{AndersonMoore1990}. In fact, that signal identity holds for
any absolutely continuous error with $e(0)=1$, $e(\infty)=0$, and
$e,v\in L^2$ satisfying the plant equation. The PI family attains this
broader signal-level minimum. Accordingly, global quadratic optimality here
is a specialization of classical optimal control; the explicit PI gap and
the subsequent controller-space characterizations are the results needed
for the present tuning problem.

The LQR feedback is written for the centered command $v$, whereas the PI
law is written for $w=v+1$. Their agreement can be checked directly. With
$x=s_q,r=1$,
\[
 w=s_qe^{-s_q\theta}+s_q\int_0^\theta e^{-s_q\xi}\dd\xi
   =1+(s_q-1)e^{-s_q\theta}.
\]
Thus the integral contribution supplies the equilibrium offset, and the
transient part is exactly the scalar optimal feedback. This statement is
about the stipulated step and initial state; it does not assert that the PI
and the scalar feedback have identical responses to arbitrary disturbances.

%% file: sections/04b_delayed_quadratic.tex
\subsection{Why the quadratic setting also changes under delay}
\label{sec:delayed_quadratic}

The delayed MIACE formula does not provide the delayed quadratic optimum.
Even on the cancellation family, the latter cost depends on the delay
through a different expression. It is useful to calculate it exactly,
because it shows directly why the delay-free gain cannot simply be kept.

\begin{proposition}[Delayed quadratic energies on the cancellation family]
\label{prop:delayed_quadratic}
For $r=1$ and $0\le x\delta<\pi/2$, define
\begin{equation}
 \Psi(z)=\frac{1+\sin z}{\cos z}.
 \label{eq:delay_energy_function}
\end{equation}
The error and centered-command energies are
\begin{equation}
 \Ecal_\delta(x)=\frac{\Psi(x\delta)}{2x},\qquad
 \Ucal_\delta(x)=(x^2+1)\Ecal_\delta(x)-\delta-1.
 \label{eq:delayed_quadratic_energies}
\end{equation}
Consequently, for $q>0$,
\begin{equation}
 F_q(x,1;\delta)
 =\frac12\left(x+\frac{q+1}{x}\right)\Psi(x\delta)-\delta-1.
 \label{eq:delayed_quadratic_cost}
\end{equation}
If $\delta>0$ and the old gain $s_q=\sqrt{1+q}$ still satisfies
$s_q\delta<\pi/2$, then
\begin{equation}
 \left.\frac{\partial F_q(x,1;\delta)}{\partial x}\right|_{x=s_q}
 =s_q\delta \Psi'(s_q\delta)>0.
 \label{eq:delayed_old_quadratic_derivative}
\end{equation}
Thus the old quadratic optimum is no longer stationary even on $r=1$.
\end{proposition}
\begin{proof}
For $\delta=0$, \eqref{eq:delayed_quadratic_energies} reduces to the
previous cancellation formulas because $\Psi(0)=1$.
Now let $\delta>0$. Exponential stability makes the following
integrations and differentiations legitimate. Put
\[
 R_e(s)=\int_0^\infty e(t)e(t+s)\dd t,\qquad s\ge0.
\]
For $0<s<\delta$, \eqref{eq:delayed_error_dde} and its zero
prehistory give
\begin{align*}
 R_e'(s)
 &=-x\int_{\delta-s}^\infty e(t)e(t+s-\delta)\dd t\\
 &=-x\int_0^\infty e(u+\delta-s)e(u)\dd u
 =-xR_e(\delta-s).
\end{align*}
Differentiating this equality yields
$R_e''(s)=xR_e'(\delta-s)=-x^2R_e(s)$.
The initial derivative is
\[
 R_e'(0^+)=\int_0^\infty ee'\dd t
 =\tfrac12[e(\infty)^2-e(0)^2]=-\tfrac12.
\]
Taking $s\downarrow0$ in $R_e'(s)=-xR_e(\delta-s)$ also gives
$R_e(\delta)=1/(2x)$. Since $R_e(0)=\Ecal_\delta$,
\[
 R_e(s)=\Ecal_\delta\cos(xs)-\frac{\sin(xs)}{2x}.
\]
Evaluating at $s=\delta$ proves the first energy formula; division
by $\cos(x\delta)$ is valid in the specified stable interval.

Next square \eqref{eq:delayed_command_identity} and integrate.
The waiting interval has $e=1$, so
$\int_0^\infty e(t+\delta)^2\dd t=\Ecal_\delta-\delta$.
The cross term is $R_e(\delta)=1/(2x)$. Therefore
\[
 \Ucal_\delta=x^2\Ecal_\delta+(\Ecal_\delta-\delta)
              -2xR_e(\delta)
 =(x^2+1)\Ecal_\delta-\delta-1.
\]
This proves \eqref{eq:delayed_quadratic_cost} as well.
At $x=s_q$, the derivative of $x+(q+1)/x$ is zero, while that
factor itself equals $2s_q$. Finally
$\Psi'(z)=(1+\sin z)/\cos^2z>0$ for $0\le z<\pi/2$.
These observations prove \eqref{eq:delayed_old_quadratic_derivative}.
\end{proof}

Decreasing the old gain slightly therefore lowers the delayed quadratic
cost whenever the old setting remains stable. This comparison does not
solve the two-parameter delayed quadratic problem or establish that its
optimum must still have $r=1$. It also shows why the delay-free
weight-matching relation derived below should not be carried over to
a delayed plant without a separate calculation.

%% file: sections/05_pareto_geometry.tex
\section{The complete quadratic trade-off}
\label{sec:pareto}

For the delay-free plant, changing $q$ changes the preferred balance between tracking and effort.
To describe all such choices, we determine the Pareto front. A controller is
Pareto efficient if no other admissible controller has both objectives no
larger and at least one strictly smaller.

Cauchy--Schwarz applied to \eqref{eq:cross_energy} gives
\[
 \Ecal\Ucal\ge(1/2-\Ecal)^2.
\]
This bound comes directly from the plant. The PI energy formulas give its
exact residual.

\begin{theorem}[Complete $L^2$ Pareto front]
\label{thm:l2_pareto}
For all $x,r>0$,
\begin{equation}
 \boxed{4\Ecal(\Ucal+1-\Ecal)-1
 =\frac{x(r-1)^2}{r(1+x)^2}\ge0.}
 \label{eq:pareto_identity}
\end{equation}
Thus
\begin{equation}
 \Ucal\ge\phi(\Ecal):=\Ecal+\frac{1}{4\Ecal}-1.
 \label{eq:pareto_lower_bound}
\end{equation}
The complete efficient controller set is
\begin{equation}
 r=1,\qquad x\ge1,
 \label{eq:pareto_controller_set}
\end{equation}
and its objective front is
\begin{equation}
 \boxed{\Ecal=\frac1{2x},\qquad
 \Ucal=\frac{(x-1)^2}{2x},\qquad x\ge1.}
 \label{eq:pareto_parametrization}
\end{equation}
\end{theorem}
\begin{proof}
The elimination can be checked with one intermediate expression:
\[
 \Ucal+1-\Ecal=\frac{x(x+1/r)}{2(1+x)}.
\]
Multiplication by $4\Ecal$ gives
$(x+r)(x+1/r)/(1+x)^2$. Subtracting one leaves
$x(r+1/r-2)/(1+x)^2$, which is \eqref{eq:pareto_identity}.
Equality holds exactly when $r=1$.
For any controller define
\begin{equation}
 \bar x=\frac1{2\Ecal(x,r)}=\frac{x(1+x)}{x+r},\qquad \bar r=1.
 \label{eq:same_tracking_controller}
\end{equation}
It has the same tracking energy and effort $\phi(\Ecal)$. If $r\ne1$,
its effort is strictly smaller, so the original controller is dominated.
On $r=1$, gains $x<1$ are dominated by $(1,1)$, which has objectives
$(1/2,0)$.

It remains to show that every point with $r=1$, $x\ge1$ is efficient.
On $0<\Ecal\le1/2$, $\phi$ is strictly decreasing. If another controller
has a strictly smaller tracking energy, \eqref{eq:pareto_lower_bound}
forces a strictly larger effort. At equal tracking energy, no smaller effort
is possible. The endpoint $(1/2,0)$ is efficient because zero effort is
attained only at $(1,1)$. This proves both completeness and efficiency.
\end{proof}

The distinction between a lower boundary and an efficient front matters.
The cancellation curve also exists for $x<1$, but that part is dominated.
The efficient front has $0<\Ecal\le1/2$ and is strictly convex because
$\phi''(\Ecal)=1/(2\Ecal^3)>0$. Its interior points correspond one-to-one
to positive weights:
\begin{equation}
 q=x^2-1,\qquad x=\sqrt{1+q}>1.
 \label{eq:weight_controller_bijection}
\end{equation}
The endpoint $x=1$ corresponds to $q=0$, or to the limit $q\downarrow0$;
no positive weight selects it exactly.

\begin{figure}[t]
 \centering
 \includegraphics[width=0.91\linewidth]{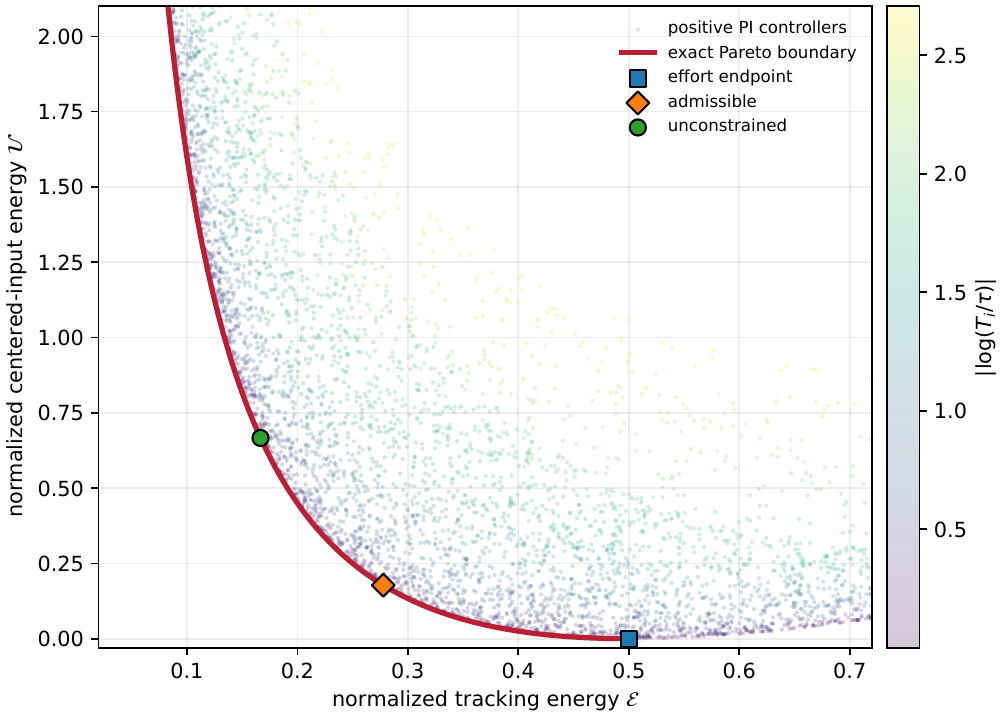}
 \caption{A sample of $12{,}000$ positive PI controllers and the exact
 efficient front. Color represents $|\log(T_i/\tau)|$, the logarithmic
 distance from the cancellation line. The sample illustrates the identity;
 it is not used to establish the front.}

 \label{fig:pareto_cloud}
\end{figure}

\subsection{Matching the quadratic and MIACE weights}

The quadratic and absolute-area
criteria select the same efficient controller when
\begin{equation}
 \sqrt{1+q}=1+\frac1\alpha,\qquad
 q=\frac2\alpha+\frac1{\alpha^2}.
 \label{eq:matched_weights}
\end{equation}
This equates their selected gains, not their values away from the optimum.

\begin{figure}[t]
 \centering
 \includegraphics[width=0.90\linewidth]{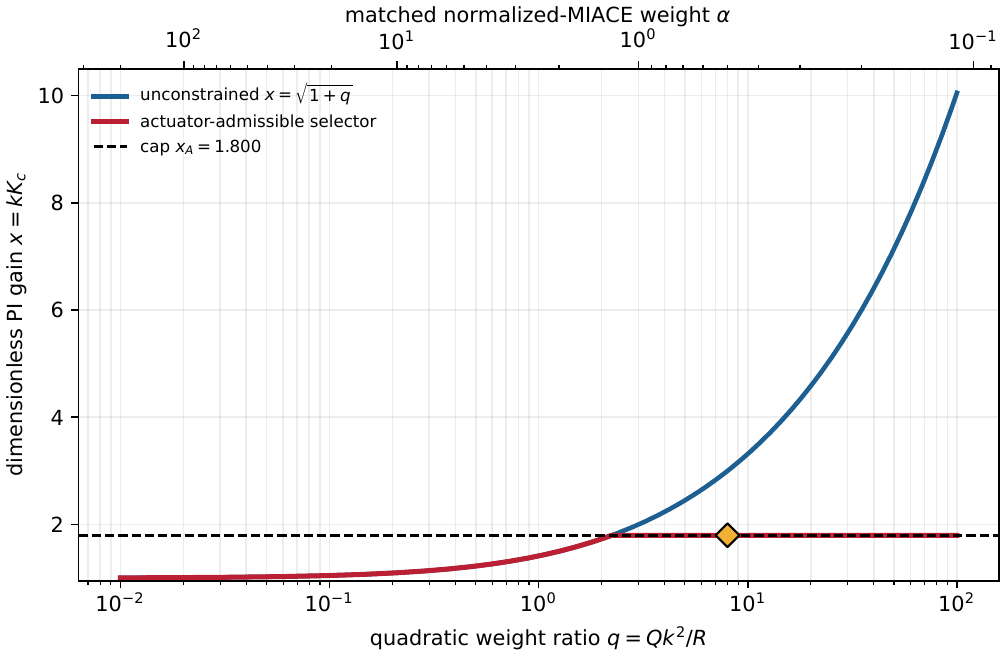}
 \caption{The quadratic weight, matched normalized MIACE weight, and selected
 gain on the common efficient family. The capped curve restricts selection
 to the part of that family passing the stated amplitude and post-jump-rate
 bounds. It does not solve the unrestricted constrained PI problem.}

 \label{fig:weight_gain_map}
\end{figure}

%% file: sections/07_actuator_design.tex
\section{Command bounds and the complete nonovershooting class}
\label{sec:actuator}

The domain and constrained optima in this section concern the delay-free
plant, $L=0$.

A small integrated effort does not exclude a large command peak. A
command-amplitude constraint must therefore be imposed separately, as in
earlier constrained PI designs
\citep{KlanGorez2008,Shi2009,ThuLee2013}. For a positive unit step from a
zero operating point, let
\begin{equation}
 0\le u(t)\le u_{\max},\qquad A=ku_{\max}\ge1.
 \label{eq:actuator_bound}
\end{equation}
The lower bound on $A$ follows from $ku_\infty=1$. We first identify every
positive PI controller whose output approaches one without crossing it.
A restriction to $x,r\ge1$ would discard valid monotone responses with $x<1$.

\subsection{Which positive PI controllers avoid overshoot?}

Set $\Delta=(1+x)^2-4x/r$. When $\Delta>0$, let
\[
 \lambda_1=\frac{1+x-\sqrt\Delta}{2},\qquad
 \lambda_2=\frac{1+x+\sqrt\Delta}{2}.
\]
They are positive and the closed-loop poles are $-\lambda_1,-\lambda_2$.
Partial fractions of $\widehat E$ give
\begin{equation}
 e=a_1e^{-\lambda_1\theta}+a_2e^{-\lambda_2\theta},\qquad
 a_1=\frac{1-\lambda_1}{\lambda_2-\lambda_1},\quad
 a_2=\frac{\lambda_2-1}{\lambda_2-\lambda_1}.
 \label{eq:positive_residues}
\end{equation}
These coefficients sum to one, and
$\lambda_1a_1+\lambda_2a_2=x$, as required by the initial conditions.
They need not both be positive.

\begin{theorem}[Complete nonovershooting domain]
\label{thm:monotone}
For the positive PI domain and zero initial conditions, the following are
equivalent: $y(\theta)\le1$ for all $\theta\ge0$; $y$ is nondecreasing
from zero to one; and
\begin{equation}
 (x,r)\in\mathcal C_{\rm mon}
 :=\{x>0,\ r\ge m(x)\},\qquad
 m(x)=
 \begin{cases}
 4x/(1+x)^2,&0<x\le1,\\
 1,&x\ge1.
 \end{cases}
 \label{eq:nonovershoot_class}
\end{equation}
Every controller in this class satisfies
\begin{equation}
 0\le w(\theta)\le\max\{1,x\},\qquad
 \sup_{\theta\ge0}w(\theta)=\max\{1,x\}.
 \label{eq:command_peak_general}
\end{equation}
The supremum may be approached only as $\theta\to\infty$.
\end{theorem}
\begin{proof}
If $\Delta<0$, the real error is an exponentially decaying sinusoid with
nonzero amplitude, since $e(0)=1$. It takes negative values and hence
produces overshoot.

Suppose $\Delta>0$. If $\lambda_1>1$, then $a_1<0$, and its slower
exponential eventually dominates the error. Thus $e<0$ eventually.
If $\lambda_1\le1$, then $a_1\ge0$. For $a_2\ge0$, both $e\ge0$ and
$-e'\ge0$ follow immediately. If $a_2<0$, factor out the slower
exponential:
\[
 e=e^{-\lambda_1\theta}
       [a_1+a_2e^{-(\lambda_2-\lambda_1)\theta}].
\]
The bracket starts at one and increases. Similarly the bracket in
\[
 -e'=e^{-\lambda_1\theta}
 [\lambda_1a_1+\lambda_2a_2e^{-(\lambda_2-\lambda_1)\theta}]
\]
starts at $x>0$ and increases. Hence $e\ge0$ and $y'=-e'\ge0$.
For distinct real poles, nonovershoot and monotonicity are therefore
equivalent to $\lambda_1\le1$.

For $\Delta=0$, put $\lambda=(1+x)/2$. The initial conditions give
\begin{equation}
 e=e^{-\lambda\theta}[1+(1-\lambda)\theta],\qquad
 -e'=e^{-\lambda\theta}[x+\lambda(1-\lambda)\theta].
 \label{eq:repeated_error}
\end{equation}
Both conditions hold exactly when $\lambda\le1$, or $x\le1$.
For $x\le1$, the real-pole condition $\Delta\ge0$ already implies
$\lambda_1\le(1+x)/2\le1$ and is equivalent to $r\ge4x/(1+x)^2$.
For $x>1$, the condition $\lambda_1\le1$ is equivalent to
$\sqrt\Delta\ge x-1$, which after squaring gives $r\ge1$.
This proves \eqref{eq:nonovershoot_class}, including its boundary.

Nonnegativity of $w$ follows from the PI law, since $e\ge0$ and its
integral is nonnegative. To locate its maximum, first consider $r>1$.
The polynomial $\lambda^2-(1+x)\lambda+x/r$ is negative at one, so
$0<\lambda_1<1<\lambda_2$ and both residues in
\eqref{eq:positive_residues} are positive. Because $w=1-e'-e$,
\[
 e^{\lambda_1\theta}w'
 =a_1\lambda_1(1-\lambda_1)
  -a_2\lambda_2(\lambda_2-1)e^{-(\lambda_2-\lambda_1)\theta}.
\]
The right side increases strictly. An interior stationary point can
therefore only be a minimum. The supremum is the larger of $w(0^+)=x$
and $w(\infty)=1$. For $r=1$, the signal
$w=1+(x-1)e^{-x\theta}$ gives the same result.

The remaining case is $r<1$ in $\mathcal C_{\rm mon}$, which requires
$x<1$ and real poles below one. For distinct poles,
$a_1>0,a_2<0$, and
\[
 v=e^{-\lambda_1\theta}
 [(\lambda_1-1)a_1+(\lambda_2-1)a_2
                         e^{-(\lambda_2-\lambda_1)\theta}].
\]
The second coefficient is positive, so the bracket decreases from
$x-1<0$. Hence $w<1$. At a repeated pole, substitution in
\eqref{eq:repeated_error} gives
$v=-e^{-\lambda\theta}[(1-x)+(1-\lambda)^2\theta]<0$.
Thus the supremum is one in these cases as well.
\end{proof}

\subsection{Exact clipping over the whole class}

\begin{theorem}[Amplitude-constrained optima without overshoot]
\label{thm:amplitude}
For $A\ge1$, the unique optima over
$\mathcal C_{\rm mon}$ subject to \eqref{eq:actuator_bound} are
\begin{align}
 x_{q,A}^\star&=\min\{\sqrt{1+q},A\},&
 r_{q,A}^\star&=1,\label{eq:amplitude_quadratic_optimum}\\
 x_{\alpha,A}^\star&=\min\{1+1/\alpha,A\},&
 r_{\alpha,A}^\star&=1.\label{eq:amplitude_miace_optimum}
\end{align}
The constraint is equivalent to $x\le A$ within this class.
\end{theorem}
\begin{proof}
The peak formula proves feasibility. We first dispose of
the additional branch $0<x<1$, rather than presuming it irrelevant.
Let $h=1+x<2$. The domain condition and the energy formula give
\begin{equation}
 \Ecal\ge\frac{1+4/h^2}{2h}>\frac12,\qquad
 A_e=\frac rx\ge\frac4{h^2}>1.
 \label{eq:low_gain_dominance}
\end{equation}
The first strict inequality follows from $1/h>1/2$ and $4/h^3>1/2$.
Both effort objectives are nonnegative. The feasible controller $(1,1)$
has $(\Ecal,\Ucal)=(1/2,0)$ and $(A_e,A_u)=(1,0)$, and therefore
strictly improves both scalar criteria on this entire branch.

It remains to consider $1\le x\le A,r\ge1$. If the unconstrained
quadratic optimum is feasible, its uniqueness has already been proved.
Otherwise $s_q>A$, and
\[
 \frac{\partial F_q}{\partial r}
 =\frac{s_q^2-x^2/r^2}{2x(1+x)}>0.
\]
Thus $r=1$ at the constrained minimum. On that line,
$F_q=x/2-1+s_q^2/(2x)$ has derivative
$(1-s_q^2/x^2)/2<0$ for $x\le A<s_q$, so its unique minimum is $x=A$.

For normalized MIACE, $z=r/x\ge1/A$. If the unconstrained selector is
infeasible, $z_\alpha=\alpha/(1+\alpha)<1/A$. The lower bound
$z^2+\alpha(1-z)^2$ has derivative
$2(1+\alpha)(z-z_\alpha)>0$ on $z\ge1/A$.
It is attained at $x=A,r=1$. Equality also requires $r/x=1/A$;
together with $r\ge1,x\le A$, this forces those parameters.
The argument includes $A=1$, where both optima are $(1,1)$.
\end{proof}

For example, $(x,r)=(1/2,9/10)$ is monotone although $r<1$.
It has $\Ecal=14/15>1/2$ and $A_e=9/5>1$.
It is included in \cref{thm:monotone} and then excluded from optimality
by \eqref{eq:low_gain_dominance}. This explains why enlarging the class
does not change the clipped answer.

\begin{figure}[tb]
 \centering
 \includegraphics[width=.86\linewidth]{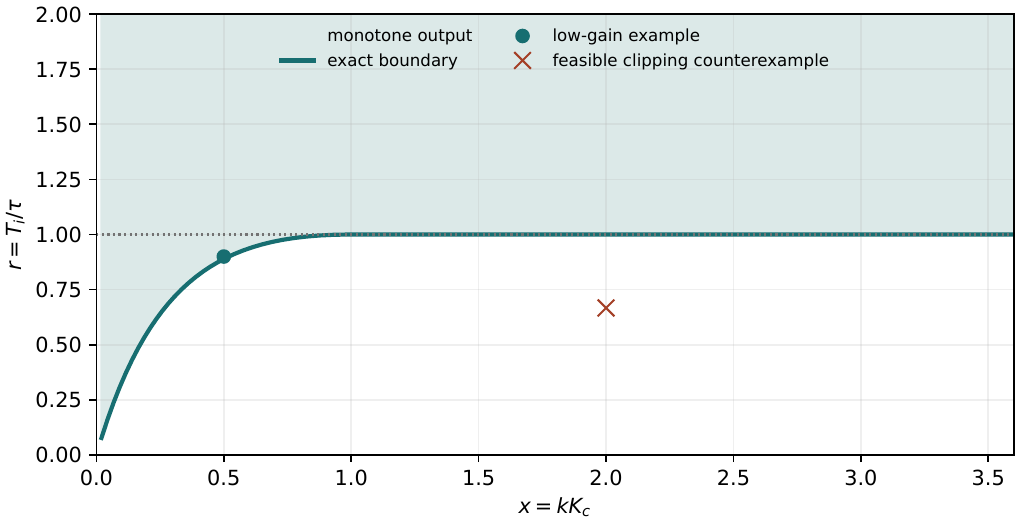}
 \caption{The complete monotone domain lies on and above $r=m(x)$.
 The low-gain example belongs to the domain although $r<1$.
 The amplitude-feasible clipping counterexample lies outside it}

 \label{fig:monotone_domain}
\end{figure}

\subsection{Why clipping still fails when overshoot is admitted}

Take $q=10,A=2$. The clipped cancellation controller $(x,r)=(2,1)$
gives $F_{10}=11/4$. The alternative
\begin{equation}
 (x,r)=(2,2/3),\qquad
 \Ecal=\frac29,\quad\Ucal=\frac7{18},\quad
 F_{10}=\frac{47}{18}<\frac{11}{4}
 \label{eq:clipping_counterexample}
\end{equation}
has command
\begin{equation}
 w=1+e^{-3\theta/2}
 \left[\cos\frac{\sqrt3\theta}{2}
       +\frac1{\sqrt3}\sin\frac{\sqrt3\theta}{2}\right].
 \label{eq:counterexample_command}
\end{equation}
Writing $\phi=\sqrt3\theta/2$ gives
$w'=-e^{-3\theta/2}(\cos\phi+\sqrt3\sin\phi)$.
Its first interior zero is at $\phi=5\pi/6$, where
\[
 w_{\min}=1-\frac1{\sqrt3}e^{-5\pi/(2\sqrt3)}
          \simeq0.9938037100.
\]
Subsequent extrema alternate about one and their distances from one
decrease by $e^{-\sqrt3\pi}$ at each half-cycle. Thus the maximum is
$w(0^+)=2$, the command is nonnegative, and the design is amplitude
feasible. Its error has complex poles and crosses zero, so it lies
outside $\mathcal C_{\rm mon}$. The improvement is exactly $5/36$,
or $5/99$ of the clipped objective. This establishes failure of clipping
over the full positive PI domain; it does not identify its constrained
global optimum.

\subsection{Amplitude and the rate after the initial jump}

On the efficient family $r=1,x\ge1$,
\begin{equation}
 y(t)=1-e^{-xt/\tau},\qquad
 u(t)=\frac1k[1+(x-1)e^{-xt/\tau}],
 \label{eq:front_physical_signals}
\end{equation}
and hence
\begin{equation}
 \norm{u}_\infty=\frac xk,\qquad
 \norm{u-u_\infty}_\infty=\frac{x-1}k,\qquad
 \sup_{t>0}|\dot u(t)|=\frac{x(x-1)}{k\tau}.
 \label{eq:front_peaks}
\end{equation}
If $B=k\tau\dot u_{\max}\ge0$, the part of this family satisfying
both bounds has
\begin{equation}
 1\le x\le X,\qquad
 X=\min\left\{A,\frac{1+\sqrt{1+4B}}2\right\}.
 \label{eq:front_two_caps}
\end{equation}
Clipping the two unconstrained selectors at $X$ is exact
\emph{on this family}. The amplitude theorem does not extend this
statement to every monotone PI controller when a rate bound is added.

The initial command changes from zero to $K_c$ instantaneously. Its
distributional derivative contains an impulse, which the regular rate
in \eqref{eq:front_peaks} does not count. A finite actuator slew limit
including that instant requires a ramped reference, a prefilter, or an
actuator state in the model. Likewise, if the predicted amplitude exceeds
the available range, saturation changes the dynamics and requires an
anti-windup analysis \citep{TarbouriechTurner2009}.

%% file: sections/08_model_mismatch.tex
\section{Exact cost of plant-parameter mismatch}
\label{sec:mismatch}

The cancellation optimum uses the plant time constant. To see what is lost
when that value is inaccurate, fix the physical weights $Q,R$ and design
the controller for $k_0,\tau_0>0$:
\[
 q_0=Qk_0^2/R,\qquad s_0=\sqrt{1+q_0},\qquad
 K_c=s_0/k_0,\quad T_i=\tau_0.
\]
Now let the actual plant remain first order and delay free, with
$k=\kappa k_0,\tau=\eta\tau_0$ and $\kappa,\eta>0$.
In its own normalized coordinates the unchanged controller has
\begin{equation}
 x=\kappa s_0,\qquad r=1/\eta,\qquad
 q=q_0\kappa^2,\qquad \sigma=\sqrt{1+q_0\kappa^2}.
 \label{eq:mismatch_coordinates}
\end{equation}
The centered command must use the actual equilibrium $1/k$. Centering
instead at $1/k_0$ would leave a nonzero limiting offset when
$\kappa\ne1$ and make the infinite-horizon effort integral diverge.

\begin{corollary}[Mismatch loss and its local expansion]
\label{cor:mismatch}
Every positive mismatch in \eqref{eq:mismatch_coordinates} preserves
internal stability. Relative to the optimal PI controller retuned to the
actual plant, the exact normalized loss is
\begin{equation}
 \boxed{L(\kappa,\eta)=
 \frac{(\sigma-\kappa s_0\eta)^2}
      {2\eta\kappa s_0(1+\kappa s_0)}
 +\frac{(\kappa s_0-\sigma)^2}{2(1+\kappa s_0)}.}
 \label{eq:mismatch_loss}
\end{equation}
For a unit step, $J_{\rm fixed}-J_{\rm retuned}=R\tau L/k^2$.
Writing $\epsilon_k=\kappa-1,\epsilon_\tau=\eta-1$,
\begin{equation}
 L=
 \frac{\epsilon_k^2}{2s_0^3}
 +\frac{\epsilon_k\epsilon_\tau}{s_0(1+s_0)}
 +\frac{s_0\epsilon_\tau^2}{2(1+s_0)}
 +O\!\left((|\epsilon_k|+|\epsilon_\tau|)^3\right).
 \label{eq:mismatch_taylor}
\end{equation}
For $\kappa s_0\ge1$, the response is nonovershooting exactly when
$\eta\le1$. For $\kappa s_0<1$, the condition is
$\eta\le(1+\kappa s_0)^2/(4\kappa s_0)$.
\end{corollary}
\begin{proof}
Positivity of $x,r$ gives stability by \eqref{eq:characteristic}.
Substitute \eqref{eq:mismatch_coordinates} into the two squares in
\eqref{eq:sos_certificate}. In the first term, multiply numerator and
denominator by $\eta^2$ to obtain \eqref{eq:mismatch_loss}.

For the expansion, put $s=s_0$. Since $q_0=s^2-1$,
\[
 \sigma=s+\frac{s^2-1}{s}\epsilon_k+O(\epsilon_k^2).
\]
The linear terms of the two squared numerators in the original
identity are
\[
 \sigma r-x=-\epsilon_k/s-s\epsilon_\tau+O(\|\epsilon\|^2),
 \qquad x-\sigma=\epsilon_k/s+O(\epsilon_k^2).
\]
Their denominators can be evaluated at $(x,r)=(s,1)$ for the
quadratic terms: each denominator correction contributes only at order
three. Therefore the second-order loss is
\[
 \frac{(\epsilon_k/s+s\epsilon_\tau)^2}{2s(1+s)}
 +\frac{(\epsilon_k/s)^2}{2(1+s)}.
\]
Collecting its three coefficients gives \eqref{eq:mismatch_taylor}.
The nonovershoot conditions follow by inserting $r=1/\eta$ into
\eqref{eq:nonovershoot_class}.
\end{proof}

The quadratic form in \eqref{eq:mismatch_taylor} is positive definite.
Its diagonal coefficients are positive and the determinant of its
symmetric coefficient matrix is $1/[4s_0(1+s_0)^2]>0$.
The exact loss is zero only when $\kappa=\eta=1$: the second square
forces $\kappa s_0=\sigma$, whose square gives $\kappa^2=1$; the first
then forces $\eta=1$.

A transparent case is exact gain but an inaccurate time constant.
Setting $\kappa=1$ gives
\begin{equation}
 L(1,\eta)=\frac{s_0(\eta-1)^2}{2\eta(1+s_0)}.
 \label{eq:time_mismatch_loss}
\end{equation}
For $q_0=8,s_0=3,\eta=1.05$, this is $1/1120$. The actual optimal
normalized cost is two, so the relative cost increase is only $1/2240$,
about $0.04464\%$. Nevertheless, the response overshoots: its slower
decay rate is $(4-\sqrt{3.4})/2>1$, with a negative error residue.
A second-order cost penalty can therefore coexist with loss of strict
nonovershoot for an arbitrarily small positive time-constant error.

The corollary evaluates a fixed nominal design. It does not solve a
minimax tuning problem over an uncertainty set, and its stability
conclusion does not cover added delay, neglected poles, actuator dynamics,
or a change in the sign of the plant gain.

%% file: sections/08_numerical_studies.tex
\section{Numerical examples and verification}
\label{sec:numerics}

The calculations begin with the MIACE shape bounds and their pure-delay
PI, filtered PID, and arbitrary-order cascade extensions, then put the delay-free PI tuning family into physical units
and examine constraints beyond
\cref{thm:amplitude}. The proofs of the global
optima and complete fronts do not depend on these numerical results.

\subsection{MIACE response-shape bounds}

The script \path{code/miace_shape_audit.py} checks
\cref{thm:miace_shape_gap,thm:miace_variation,prop:finite_miace_shape}
by direct physical-loop integration. The cases include optimal controllers,
settings on both sides of the monotone boundary, repeated poles, slow
responses, and prescribed non-PI commands. The latter are formed from
sums of decaying sinusoidal terms; their command is then applied to a
separately integrated first-order plant. Absolute areas, negative areas,
total variation, and backward motion are accumulated during integration.
No extremum grid is used to establish the theoretical inequalities.

\Cref{tab:miace_shape_examples} shows four PI examples at $\alpha=1$.
The optimum is monotone. Moving the integral time below the cancellation
value introduces overshoot, with a positive MIACE excess and a valid
upper bound on the accumulated backward motion. The tabulated peaks are
dense-grid numerical estimates; the analytical proof covers all times.
The two equal-cost monotone responses in \cref{fig:miace_shape_examples}
also illustrate why controller-independent optimality does not imply a
unique waveform.

\begin{table}[htbp]
\centering
\caption{MIACE shape checks at $\alpha=1$, with $m_1=1/\sqrt2$.
The quantities $P_y$ and $B_y$ are normalized by the step magnitude.}
\label{tab:miace_shape_examples}
\small
\input{data/miace_shape_table_rows.tex}
\end{table}

\input{data/miace_shape_verification.tex}

\subsection{Pure delay: attainment and its limits}
\label{sec:delay_numerical_checks}

\Cref{tab:delay_examples} compares the old gain with the corrected formula
at three delay values and $\alpha=1$. At $\delta=0.1$ both responses
are monotone, but the old gain no longer balances the two areas. At
$\delta=0.2$ the old gain produces about $0.070708\%$ overshoot of
the step, whereas \cref{thm:delayed_optimum} guarantees exactly zero
overshoot for the corrected setting. At $\delta=0.25$ the gain formula
is outside its attainment range. Its extrapolation $x=1.6,r=1$ has
$x\delta=0.4>1/\mathrm e$ and also overshoots by about
$0.070708\%$; the old gain $x=2$ overshoots by about $4.05196\%$.
The extrapolated setting is not asserted to solve the delayed PI
optimization. Its small positive cost excess is consistent with the
bound in \cref{thm:delayed_shape}.

\begin{table}[htbp]
\centering
\caption{Pure-delay examples with $r=1$ and $\alpha=1$. The corrected
formula is a proved optimum for the first two delays only.}
\label{tab:delay_examples}
\small
\input{data/delay_table_rows.tex}
\end{table}

The physical-loop calculation in \path{code/delay_audit.py} integrates
$y'=-y+w(\theta-\delta)$ and $z'=1-y$, with
$w=x(1-y+z/r)$. It uses the method of steps on intervals of length
$\delta$, so every delayed state comes from an already completed interval.
The DOP853 relative and absolute tolerances are $2\times10^{-12}$ and
$3\times10^{-14}$, and the maximum internal step is $\delta/12$.
No pole cancellation or rational delay approximation enters this solver.
Absolute and signed areas, negative areas, output variation, backward
motion, and the two quadratic energies are integrated independently
from its dense state solution. Quadrature is split at every delay boundary.

For $r=1$, expanding \eqref{eq:delayed_cancellation} in a right half-plane
and inverting each delayed power gives an independent exact expression:
\begin{equation}
 \widehat E(p)=\sum_{n=0}^\infty
 \frac{(-x)^n e^{-n\delta p}}{p^{n+1}},\qquad
 e(\theta)=\sum_{n=0}^{\lfloor\theta/\delta\rfloor}
 \frac{(-x)^n(\theta-n\delta)^n}{n!},\quad \delta>0.
 \label{eq:delayed_exponential_series}
\end{equation}
The term with $n=0$ is one. At each fixed time the time-domain sum is
finite. It is evaluated with 65 decimal digits and compared with the
physical-loop error at 81 points on $[0,10]$. The quadratic integrals
are checked separately against \eqref{eq:delayed_quadratic_energies}.

The audited settings include both sides of the attainment boundary,
two noncancellation controllers, the stability-preserving but oscillatory
case $x=2,\delta=0.5$, and the exact boundary at weights
$\alpha=0.5,1,\mathrm e,4$. At $\alpha=\mathrm e$ both restrictions
in \eqref{eq:delayed_attainment_range} are active. A further case uses
$\alpha=4$, $\delta=0.3$, and $x=1.25/1.3<1$ at $r=1$.
Its output is monotone because $x\delta<1/\mathrm e$, but its
initial centered command is negative. It therefore cannot attain
the lower bound even though it has no output overshoot. This is a
check of the distinction between output shape and the two equality
conditions, not a claim that this last setting is optimal.

\input{data/delay_verification.tex}

The numerical horizon for each case is
$\max\{40,48/x,128\delta\}$. The largest residuals and full-precision
gains are recorded in \path{data/delay_audit.json}.
Peak estimates combine a dense grid with local scalar refinement; they
do not provide rigorous interval bounds on the global peaks. Nor are the finite-horizon
integrals rigorous enclosures of arbitrary tails. The infinite-horizon
statements, including zero overshoot at the exact optimum, rest on
the preceding proofs.

\subsection{Filtered PID beyond the PI attainment range}
\label{sec:pid_numerical_checks}

\Cref{tab:pid_delay_examples} compares the exact PID constructions with
extrapolated PI settings at the same weights and delays. In the first
pair, the PID has $A_e=A_u=5/8$, while the PI at $x=1.6,r=1$ slightly
overshoots and has a positive MIACE excess. In the second pair, both
outputs are monotone, but the PI at $x=15/16,r=1$ starts with a negative
centered command. The larger cost of this PI therefore arises from its
command deficit rather than from output overshoot. These comparisons do
not claim that the extrapolated PI settings are optimal within PI.

\begin{table}[htbp]
 \centering
 \caption{Filtered PID and extrapolated PI. The delay is $\delta=1/4$
 for $\alpha=1$ and $\delta=1/3$ for $\alpha=4$. All areas and criteria
 are dimensionless. The PID values attain the exact bound.}
 \label{tab:pid_delay_examples}
 \small
 \input{data/pid_delay_table_rows.tex}
\end{table}

The script \path{code/pid_delay_audit.py} integrates the full physical loop,
with three states initialized at zero:
\begin{align}
 y'&=-y+w(\theta-\delta),& z'&=1-y,\\
 \zeta'&=(1-y-\zeta)/\rho,&
 w&=p_c(1-y)+i_cz+(d_c/\rho)(1-y-\zeta).
 \label{eq:pid_physical_states}
\end{align}
The stored parameter names \texttt{a,b,c,h} correspond respectively to
$p_c,i_c,d_c,\ell$ in the manuscript; $\rho$ and the normalized signals
are unchanged. The method of steps uses intervals of length $\delta$,
with the delayed command taken from an already completed interval.
The DOP853 tolerances are $2\times10^{-12}$ relative and
$3\times10^{-14}$ absolute, with maximum internal step
$\min\{\delta/12,\rho/4\}$. Neither pole cancellation nor a rational
approximation of the delay is used in this physical solver. Quadrature
is split at every delay boundary and separately accumulates signed and
absolute areas, negative areas, variation, and backward motion.

An independent calculation checks the error from its exact finite
transport series. For sufficiently large $\operatorname{Re}p$,
\begin{equation}
 \widehat E(p)=\sum_{n=0}^{\infty}
 \frac{(-i_c)^n e^{-n\delta p}}{p^{n+1}}\,H_F(p)^n.
 \label{eq:pid_transport_series}
\end{equation}
Write $H_\infty=\ell/\rho$ and
$H_F=H_\infty+(1-H_\infty)/(1+\rho p)$. Define, for $u\ge0$,
\begin{equation}
 J_{n,j}(u)=\frac{u^{n+j}}{\rho^j(n+j)!}
 {}_1F_1\left(j;n+j+1;-\frac u\rho\right).
 \label{eq:pid_series_basis}
\end{equation}
For $j=0$ this is $u^n/n!$. For $j\ge1$, convolution of
$u^n/n!$ with the impulse response
$u^{j-1}e^{-u/\rho}/[\rho^j(j-1)!]$ gives
\[
 \frac1{\rho^j n!(j-1)!}\int_0^u
 (u-s)^n s^{j-1}e^{-s/\rho}\dd s.
\]
Putting $s=uv$ and using the beta-integral representation of
${}_1F_1$ gives exactly \eqref{eq:pid_series_basis}. Expanding
$H_F^n$ by the binomial theorem and applying the time shift in
\eqref{eq:pid_transport_series} therefore gives
\begin{equation}
 e(\theta)=\sum_{n=0}^{\lfloor\theta/\delta\rfloor}(-i_c)^n
 \sum_{j=0}^n\binom nj H_\infty^{n-j}(1-H_\infty)^j
 J_{n,j}(\theta-n\delta),\qquad \delta>0.
 \label{eq:pid_finite_series}
\end{equation}
The term $n=j=0$ is one. Causality makes the sum finite at each fixed
finite time, so no time-domain truncation of an infinite series is used.

This formula is evaluated with 80 decimal digits at 26 points on
$[0,5]$ for each of the two central PID examples. Kummer's transformation
${}_1F_1(j;n+j+1;-x)=e^{-x}{}_1F_1(n+1;n+j+1;x)$ gives a
positive-term series for the special-function evaluation at $x\ge0$.
The resulting high-precision value is then compared with the independently
integrated physical error. This is a numerical cross-check of the exact
formula, not a rigorous bound on its floating-point evaluation error.

The eight cases also include the second two-rate PID with
$\lambda=5$, $\mu=7$, the explicit-family endpoint
$\alpha=4$, $\delta=1/\mathrm e$, and two additional PI references.
The exponential signs in both two-rate constructions have a separate
check using exact \texttt{Fraction} arithmetic. The degree-36 and degree-37
Taylor sums enclose the required exponentials. Rational interval operations
then verify filter positivity, $\ell>\rho$, the two rate restrictions,
the initial-command forcing bound, and the kernel norm. The interval
endpoints are retained as exact rational numbers in
\path{data/pid_delay_audit.json}; decimal endpoints are display values.

\input{data/pid_delay_verification.tex}

The horizon is 40 for every PID audit case. Peak estimates use 4001
samples on $[0,10]$ and do not provide rigorous bounds on the global peaks. Small final
residuals do not constitute rigorous enclosures of arbitrary tails.
Internal stability and infinite-horizon positivity in the cases covered by the theorems
rest on \cref{thm:pid_two_rate_optimum,thm:pid_explicit_family}.

\subsection{Arbitrary-order cascades with the exact delay}
\label{sec:n_numerics}

The script \path{code/order_n_audit.py} checks the construction in
\cref{thm:n_construction} in two independent state realizations.
The physical calculation keeps all $2N$ rational states. If
$(A_c,b_c,c_c,d_c)$ realizes $A/Q$, its equations are
\begin{align}
 z'&=A_cz+b_c(1-y),&
 w&=c_cz+d_c(1-y),\notag\\
 \sigma_1x_1'&=-x_1+w(\theta-\delta),&
 \sigma_jx_j'&=-x_j+x_{j-1}\quad(2\le j\le N),\qquad y=x_N.
 \label{eq:n_physical_audit}
\end{align}
All states and the command prehistory are zero. The method of steps
uses the completed preceding interval for every delayed command.
No cancellation or rational approximation of the delay is used in
this physical model.

The independent reduced calculation realizes $1/Q$ with $N$ states.
Its input is $e(\theta-\delta)$, with zero error prehistory, and its
output is $y$. This follows from the exact relation
$Y=e^{-\delta p}E/Q$. It therefore checks the cancellation-based
transfer calculation against the unreduced plant--controller loop.
The two outputs are compared at 201 points on $[0,20]$.
Separate real and complex transform evaluations check the exact
interpolation remainder and characteristic factorization.

\Cref{tab:n_audit} includes equal-lag orders $2,3,4,6$ at
$\delta=1/4$ and $\alpha=2N-1$. The two additional cases use
$\boldsymbol{\sigma}=(1,1/2,1/4)$ at $(\delta,\alpha)=(0.1,2)$
and $(1,3/4,1/2,1/4)$ at $(0.1,4)$. Their respective parity
inequalities are checked from the kernel moments. The controller
coefficients, margins, moment values, and full-precision results
are recorded in \path{data/order_n_audit.json}.

\begin{table}[htbp]
 \centering
 \caption{Arbitrary-order attainment checks. The first four plants
 have identical lags; the final two have the unequal normalized
 time constants specified in the text. The last column is the
 analytical lower bound $S m_\alpha$.}
 \label{tab:n_audit}
 \small\input{data/order_n_table_rows.tex}
\end{table}

The DOP853 tolerances are $2\times10^{-11}$ relative and
$5\times10^{-13}$ absolute. Signed and absolute areas are integrated
to $\theta=50$, with quadrature split at the delay boundaries.
On the six cases, the largest discrepancy between an integrated
area and its exact value is $3.48\times10^{-13}$. The two state
realizations differ by at most $3.45\times10^{-15}$ on the comparison
grid, and the largest exact-remainder evaluation residual is below
$5.81\times10^{-16}$. The output curves in
\cref{fig:n_responses} come from the physical model.

Finite-horizon quadrature and sampled signs are numerical cross-checks.
They do not enclose every tail or prove a sign on an uncountable time
set. The all-order assertion and the complete delay interval in
\cref{thm:n_all_orders} follow from the analytical estimates.
The separate counterexample in \cref{fig:n_counterexample} is verified
from its exact rational transfer functions and closed-form signals.

\subsection{A dimensional example}
\label{sec:dimensional_benchmark}

Take
\begin{equation}
 k=2,\quad \tau=5\,\mathrm{s},\quad Q=2,\quad R=1,
 \quad q=8,\quad\alpha=0.5.
 \label{eq:benchmark_parameters}
\end{equation}
The weights are matched by \eqref{eq:matched_weights}. Both criteria select
$x=3$, or $K_c=1.5$ and $T_i=5\,\mathrm{s}$. Suppose the available positive
command range ends at $u_{\max}=0.9$. Then $A=ku_{\max}=1.8$ and
\cref{thm:amplitude} selects $x=1.8$, $r=1$ in $\mathcal C_{\rm mon}$.

\begin{table}[t]
\centering
\caption{The dimensional example, with $T_i=5\,\mathrm{s}$ for all three
designs. $J=2E_2+U_2$; $\Mcal_{0.5}$ is dimensionless.}
\label{tab:dimensional_benchmark}
\begin{tabular}{@{}lrrrrrr@{}}
\toprule
Design & $x$ & $K_c$ & $E_2$ & $U_2$ & $J$ & $\Mcal_{0.5}$\\
\midrule
Minimum effort & 1.0 & 0.5 & 2.5000 & 0 & 5.0000 & 1.0000\\
Amplitude feasible & 1.8 & 0.9 & 1.3889 & 0.2222 & 3.0000 & 0.6383\\
Unconstrained optimum & 3.0 & 1.5 & 0.8333 & 0.8333 & 2.5000 & 0.5774\\
\bottomrule
\end{tabular}
\end{table}

The amplitude cap raises $J$ from $2.5$ to $3.0$ in this class, a $20\%$
increase. The feasible command starts at $0.9$, whereas the unconstrained
command starts at $1.5$. The associated post-jump rates are $0.144$ and
$0.6$. A separate bound of $0.15$ on this regular rate would give
$x\le(1+\sqrt7)/2=1.822876\ldots$ on the cancellation family, so the
amplitude cap remains the active restriction in this example.

\begin{figure}[t]
 \centering
 \includegraphics[width=0.92\linewidth]{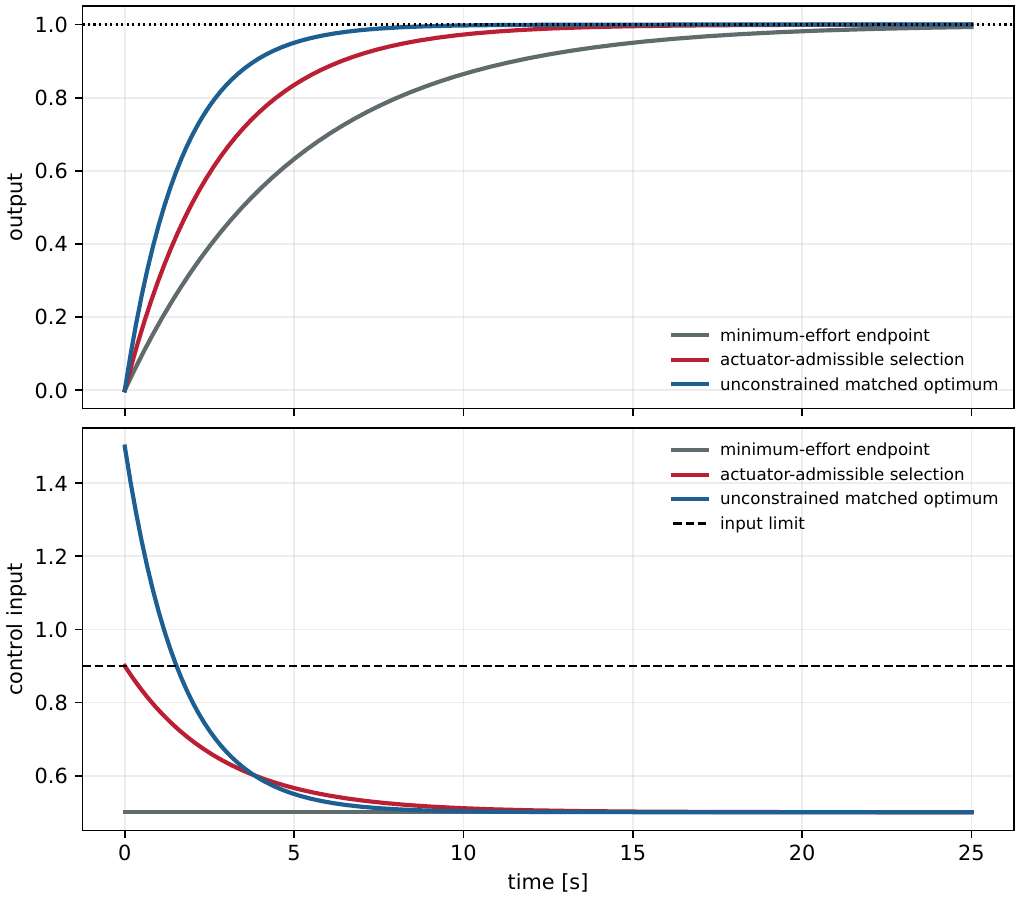}
 \caption{Output and command for the designs in
 \cref{tab:dimensional_benchmark}. The plotted command begins at $0^+$,
 after the ideal PI jump.}

 \label{fig:benchmark_responses}
\end{figure}

\subsection{Two-parameter constrained search}
\label{sec:bounded_search}

For the normalized plant $G(p)=1/(p+1)$, use the parallel form
\begin{equation}
 C(p)=p_c+\frac{i_c}{p},\qquad p_c=x,\quad i_c=x/r.
 \label{eq:parallel_normalized_pi}
\end{equation}
The numerical domain and command bounds are
\begin{equation}
 0\le p_c\le2,\qquad0.05\le i_c\le8,
 \qquad0\le w\le2,\qquad\sup_{\theta>0}|w'|\le1.
 \label{eq:numerical_box}
\end{equation}
The boundary $p_c=0$ includes pure integral control as a separate limiting
case of the parallel form; it is not part of the strictly positive
$(x,r)$ theorem domain. The rate condition applies after the initial jump.
In exact arithmetic, all extrema are obtained from the formulas described
in \cref{app:peak_oracle}; the numerical implementation uses floating-point
arithmetic.

We minimized $\beta\Ecal+\Ucal$, where $\beta$ is the same normalized
weight as $q$. Differential evolution used population multiplier $20$,
relative tolerance $10^{-10}$, and at most $700$ generations. SLSQP refinement
used objective tolerance $10^{-13}$ and at most $1000$ iterations.
Four separate seeds were used to check the five representative weights in
\cref{tab:constrained_search}. The supplementary archive lists the seed for
each of the other plotted weights. Feasibility is checked at the
full-precision gains, with maximum permitted violation $10^{-8}$;
rounded table entries are not substitutes for those gains.

\begin{table}[t]
\centering
\caption{Reproduced numerical candidates. Improvement is relative to the
best cancellation controller satisfying both bounds in
\eqref{eq:numerical_box}. ``Rate'' means the regular post-jump rate.}
\label{tab:constrained_search}
\input{data/constrained_table_rows.tex}
\end{table}

The cancellation family satisfies these two bounds only up to
$x=(1+\sqrt5)/2$. Once a bound becomes active, the reproduced candidates
move away from $p_c=i_c$. This is consistent with the exact counterexample
\eqref{eq:clipping_counterexample}; constraints can exclude the controller
that would dominate a given design in the unconstrained problem.

For $\alpha=1$, the unconstrained normalized MIACE optimum is $(2,2)$ with
$\Mcal_1^2=0.5$. The constrained calculation gives
\begin{equation}
 (p_c,i_c)\simeq(1.66484,1.77168),\qquad
 \Mcal_1^2\simeq0.51560940.
 \label{eq:constrained_miace_candidate}
\end{equation}
Its post-jump rate bound is active. Here differential evolution used
population multiplier $25$ and at most $900$ generations; the refinement
tolerances were unchanged. Four additional seeds reproduced the objective
to the tolerance reported below. Agreement between starts supports numerical
reproducibility; it does not prove global constrained optimality.

\begin{figure}[t]
 \centering
 \includegraphics[width=0.91\linewidth]{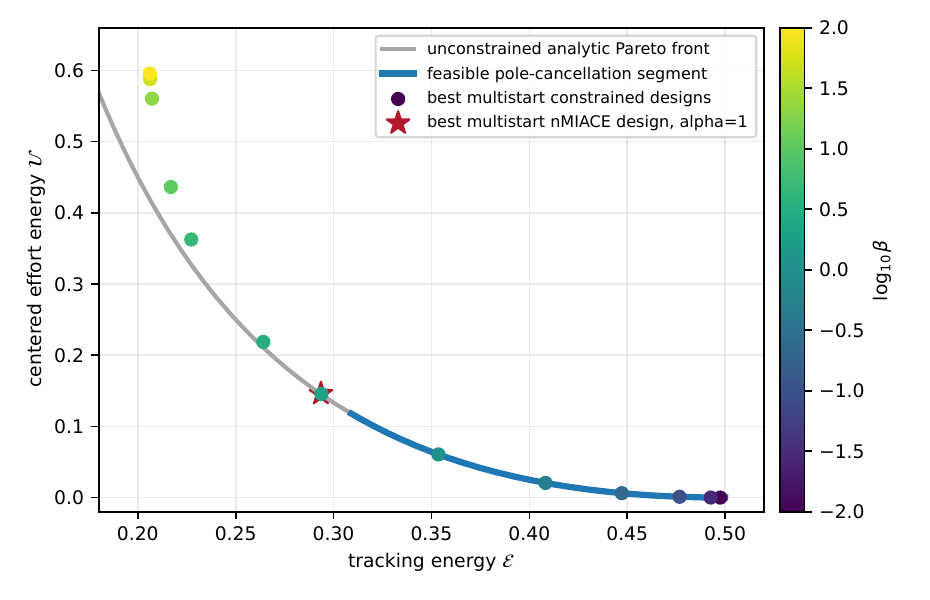}
 \caption{The analytic front, its feasible cancellation segment, and
 reproduced two-parameter constrained candidates. The colored points are
 numerical designs; a complete constrained Pareto front is not claimed.}

 \label{fig:constrained_pareto}
\end{figure}

\subsection{Checks independent of the tuning search}
\label{sec:numerical_audits}

The original transfer-function Gramian checks were reproduced on $256$ and
$4096$ parameter samples. A further $512$-case calculation used the physical
PI state equations in \cref{app:peak_oracle}. It checked the two energies,
the cross-energy balance, and the signal square independently of the
transfer-function realization. Thirty-six rational parameter combinations
also satisfied the two parameter identities exactly in rational arithmetic;
these are implementation checks in addition to the algebraic proofs.

For the absolute areas, $35$ direct integrations of the physical loop included
real, complex, repeated, and nearly repeated poles. Separate calculations
with step amplitudes $0.2$, $2.5$, and $-3$ checked
\eqref{eq:physical_nmiace_scaling}. The integration horizon was $45$ divided
by the slowest decay rate, with physical time scaling included. These are
floating-point cross-checks, not rigorous tail enclosures for arbitrary
controllers.

\begin{table}[t]
\centering
\caption{Maximum discrepancies in the reproduced checks. The first two
rows concern energies; the next two compare analytic formulas with direct
ODE calculations. The last rows compare objective values across four seeds.}
\label{tab:audit_summary}
\input{data/audit_table_rows.tex}
\end{table}

The $40$ command-extrema checks used an independent adaptive ODE solution
evaluated on a dense time grid. Finite grid spacing can miss an extremum,
while the ODE solver also has numerical error. The discrepancies in
\cref{tab:audit_summary} were below the declared $5\times10^{-7}$ threshold.
All optimizer runs passed their success and feasibility checks. The
supplement includes the full parameter values, convergence settings,
software versions, and JSON results, including the separate $20{,}000$-case
quadratic-identity check.

The extended monotonicity and mismatch results have a separate check.
Twenty-eight physical-loop integrations include both sides of the domain
boundary and repeated-pole cases. They also compare the negative-part
areas from direct integration with the absolute-area implementation.
Another 1200 controllers cover the complete monotone domain, including
its low-gain branch; their command peaks and costs are checked against
the clipping formulas. Finally, 512 physical-state Gramian evaluations
check the exact mismatch loss. These finite samples support the
implementation and do not establish the universal domain statements.
The energy and Pareto identities, the derivative used in clipping,
the mismatch identity, and its Hessian have zero symbolic residuals.

\begin{table}[htbp]
\centering
\caption{Checks of the complete-domain and mismatch extensions}
\label{tab:extension_checks}
\input{data/extension_table_rows.tex}
\end{table}

%% file: data/miace_shape_table_rows.tex
\begin{tabular}{@{}rrrrrr@{}}\toprule
$x$ & $r$ & $\Mcal_1$ & $P_y$ & $B_y$ & $(\Mcal_1/m_1-1)/2$\\\midrule
2.000000 & 1.000000 & 0.707107 & 0.000000 & 0.000000 & 0.000000\\
2.000000 & 0.980000 & 0.708084 & 0.000306 & 0.000306 & 0.000691\\
2.000000 & 0.800000 & 0.755047 & 0.016145 & 0.016146 & 0.033899\\
2.000000 & 0.200000 & 1.403592 & 0.236391 & 0.289700 & 0.492489\\
\bottomrule\end{tabular}

%% file: data/miace_shape_verification.tex
All three symbolic residuals vanish. The numerical checks cover 32 positive PI controllers, 16 non-PI commands, and 240 finite records. The largest scaled residual in the infinite-horizon gap identity is $9.46\times10^{-13}$, where scaling means division by $\max\{1,\Mcal_\alpha^2\}$. The largest absolute residual in $\mathcal V_y=1+2B_y$ is $1.26\times10^{-12}$. The direct PI integrals differ from the closed-form absolute areas by at most $9.71\times10^{-8}$ and from the closed-form output variation by $4.73\times10^{-10}$. The scaled finite-record identity residual is at most $2.00\times10^{-15}$. All checks passed their stated tolerances; raw cases, horizons, parameters, and limitations are in \path{data/miace_shape_audit.json}.

%% file: data/delay_table_rows.tex
\begin{tabular}{@{}llrrrr@{}}
\toprule
$\delta$ & Design & $x$ & $A_e$ & $A_u$ & $\Mcal_1$\\
\midrule
0.10 & Old gain & 2.000000 & 0.500000 & 0.600000 & 0.781025\\
0.10 & Optimum & 1.818182 & 0.550000 & 0.550000 & 0.777817\\
0.20 & Old gain & 2.000000 & 0.500708 & 0.700940 & 0.861409\\
0.20 & Optimum & 1.666667 & 0.600000 & 0.600000 & 0.848528\\
0.25 & Old gain & 2.000000 & 0.542173 & 0.809371 & 0.974183\\
0.25 & Extrapolated formula & 1.600000 & 0.625884 & 0.625832 & 0.885097\\
\bottomrule
\end{tabular}

%% file: data/delay_verification.tex
The delay audit covers $14$ physical-loop cases, including $12$ cancellation settings and $6$ settings satisfying the hypotheses of \cref{thm:delayed_optimum}. The largest discrepancy between the physical-loop error and the independent delayed-exponential series is $8.28\times10^{-16}$. The largest signed-area error is $5.92\times10^{-14}$, and the largest quadratic-energy error is $8.88\times10^{-16}$. The delay-gap and variation residuals are at most $1.29\times10^{-13}$ and $1.17\times10^{-13}$, respectively.

%% file: data/pid_delay_table_rows.tex
\begin{tabular}{@{}lrrrr@{}}
\toprule
Design & $A_e$ & $A_u$ & $\Mcal_\alpha$ & Lower bound\\
\midrule
PID, $\alpha=1$ & 0.62500000 & 0.62500000 & 0.88388348 & 0.88388348\\
PI, $\alpha=1$ & 0.62588446 & 0.62583214 & 0.88509730 & 0.88388348\\
PID, $\alpha=4$ & 1.06666667 & 0.26666667 & 1.19256959 & 1.19256959\\
PI, $\alpha=4$ & 1.06666667 & 0.27083333 & 1.19631959 & 1.19256959\\
\bottomrule
\end{tabular}

%% file: data/pid_delay_verification.tex
The eight physical-loop cases include four settings with proved attainment. Both two-rate constructions satisfy every sign condition checked using exact rational interval arithmetic. The largest difference between the physical error and the independent finite series is $3.44\cdot10^{-15}$. The signed-area discrepancy is at most $4.26\cdot10^{-14}$, the MIACE identity residual is at most $4.82\cdot10^{-14}$, and the variation identity residual is at most $2.33\cdot10^{-14}$. All numerical integrals are over $[0,40]$. These residuals check the implementation; the analytical proofs establish the signs and optimum on the infinite horizon.

%% file: data/order_n_table_rows.tex
\begin{tabular}{@{}rrrrrrr@{}}
\toprule
$N$ & $\delta$ & $\alpha$ & $A_e$ & $A_u$ & $\Mcal_\alpha$ & Bound\\
\midrule
2 & 0.25 & 3 & 1.6875000 & 0.5625000 & 1.9485572 & 1.9485572\\
3 & 0.25 & 5 & 2.7083333 & 0.5416667 & 2.9668305 & 2.9668305\\
4 & 0.25 & 7 & 3.7187500 & 0.5312500 & 3.9755110 & 3.9755110\\
6 & 0.25 & 11 & 5.7291667 & 0.5208333 & 5.9839194 & 5.9839194\\
3 & 0.10 & 2 & 1.2333333 & 0.6166667 & 1.5105187 & 1.5105187\\
4 & 0.10 & 4 & 2.0800000 & 0.5200000 & 2.3255107 & 2.3255107\\
\bottomrule
\end{tabular}

%% file: data/constrained_table_rows.tex
\begin{tabular}{@{}rrrrlr@{}}
\toprule
$\beta$ & $p_c$ & $i_c$ & $\beta\mathcal{E}+\mathcal{U}$ & Active bound & Improvement\\
\midrule
1 & 1.414214 & 1.414214 & 0.414214 & none & 0.000\%\\
2 & 1.664066 & 1.769115 & 0.733064 & rate & 0.408\%\\
5 & 1.966768 & 2.868176 & 1.499062 & rate & 9.864\%\\
10 & 2.000000 & 3.316625 & 2.605542 & amplitude & 18.785\%\\
50 & 1.987161 & 4.306353 & 10.900620 & amplitude & 29.985\%\\
\bottomrule\end{tabular}

%% file: data/audit_table_rows.tex
\begin{tabular}{@{}lr@{}}
\toprule
Check & Maximum discrepancy\\
\midrule
Transfer-function Gramian, 4096 cases & $1.01\times10^{-11}$\\
Physical-loop Gramian, 512 cases & $5.68\times10^{-14}$\\
Absolute areas versus physical-loop ODE & $8.90\times10^{-8}$\\
Command extrema versus ODE grid & $1.93\times10^{-8}$\\
Quadratic objective spread (largest of five weights) & $2.58\times10^{-13}$\\
Normalized MIACE squared, four seeds & $3.63\times10^{-14}$\\
\bottomrule\end{tabular}

%% file: data/extension_table_rows.tex
\begin{tabular}{@{}lr@{}}\toprule
Check & Maximum discrepancy\\\midrule
Negative-area gap, 28 physical-loop integrations & $2.07\times10^{-9}$\\
Command peak, 1200 monotone controllers & $0$\\
Mismatch loss, 512 physical-loop Gramians & $8.45\times10^{-15}$\\
\bottomrule\end{tabular}

%% file: sections/08b_nonlinear_balance.tex
\section{A nonlinear first-order extension through a storage balance}
\label{sec:nonlinear_balance}

The area argument does not require a transfer function. What it requires
is a balance that fixes the sum of the two signed areas. A nonlinear
first-order plant can have such a balance even when its speed changes
with its state. Consider, in fixed dimensionless time $\theta$,
\begin{equation}
 y'=a(y)(w-y),\qquad y(0)=0,\qquad
 a:\Rset\longrightarrow(0,\infty),
 \label{eq:nonlinear_plant}
\end{equation}
where $a$ is locally Lipschitz. The target is one, the equilibrium
command is $w_\infty=1$, and
$e=1-y$, $v=w-1$. This section has no transport delay. The factor
$1/a(y)$ can be read as a positive state-dependent storage coefficient
in the equation $a(y)^{-1}y'=w-y$; it need not be constant or monotone.

An admissible response has an ordinary locally integrable command,
a globally defined, locally absolutely continuous output satisfying
\eqref{eq:nonlinear_plant} almost everywhere, $y(\theta)\to1$, and
$e,v\in L^1(0,\infty)$. Retain exactly the same objective,
\begin{equation}
 A_e=\int_0^\infty|e|\dd\theta,\qquad
 A_u=\int_0^\infty|v|\dd\theta,\qquad
 \Mcal_\alpha=\sqrt{A_e^2+\alpha A_u^2},\qquad \alpha>0.
 \label{eq:nonlinear_objective}
\end{equation}
Define the storage coordinate and its change between the equilibria by
\begin{equation}
 \Phi_a(y)=\int_0^y\frac{ds}{a(s)},\qquad
 S_a=\Phi_a(1)=\int_0^1\frac{ds}{a(s)}>0.
 \label{eq:nonlinear_storage}
\end{equation}
The coordinate $\Phi_a$ is used to integrate the plant equation.
The error in MIACE remains $1-y$; neither the measured error nor the
time variable in \eqref{eq:nonlinear_objective} is replaced.

\begin{theorem}[Exact MIACE optimum for a nonlinear storage class]
\label{thm:nonlinear_miace}
For every admissible response of \eqref{eq:nonlinear_plant}, let
\[
 \Ncal=\int_0^\infty\max\{-e,0\}\dd\theta
       +\int_0^\infty\max\{-v,0\}\dd\theta,
 \qquad m_\alpha=\sqrt{\frac{\alpha}{1+\alpha}}.
\]
Then
\begin{gather}
 \int_0^\infty e\dd\theta+\int_0^\infty v\dd\theta=S_a,
 \qquad A_e+A_u=S_a+2\Ncal,
 \label{eq:nonlinear_balance}\\
 \boxed{\Mcal_\alpha^2-S_a^2m_\alpha^2
 =4m_\alpha^2\Ncal(S_a+\Ncal)
  +\frac{(A_e-\alpha A_u)^2}{1+\alpha}.}
 \label{eq:nonlinear_gap}
\end{gather}
The global minimum over this admissible class is
\begin{equation}
 \boxed{\min\Mcal_\alpha=S_a\sqrt{\frac{\alpha}{1+\alpha}}.}
 \label{eq:nonlinear_minimum}
\end{equation}
A response attains this minimum if and only if $e,v\ge0$ almost
everywhere and $A_e=\alpha A_u$. Every attaining output is
nondecreasing from zero to one.

One attaining causal feedback law, with equilibrium feedforward, is
\begin{equation}
 \boxed{w=1+\frac{1-y}{\alpha},\qquad v=\frac e\alpha.}
 \label{eq:nonlinear_feedback}
\end{equation}
Put $c_\alpha=1+1/\alpha$ and
$a_{\min}=\min_{0\le z\le1}a(z)$,
$a_{\max}=\max_{0\le z\le1}a(z)$. Its response satisfies
\begin{gather}
 e^{-c_\alpha a_{\max}\theta}
 \le e(\theta)\le e^{-c_\alpha a_{\min}\theta},
 \label{eq:nonlinear_decay}\\
 A_e=\frac{\alpha S_a}{1+\alpha},\qquad
 A_u=\frac{S_a}{1+\alpha},\qquad
 1\le w(\theta)\le1+\frac1\alpha,
 \quad\sup_{\theta\ge0}w(\theta)=1+\frac1\alpha.
 \label{eq:nonlinear_areas_peak}
\end{gather}
\end{theorem}
\begin{proof}
Positivity and continuity of $a$ make $\Phi_a$ continuously
differentiable. The chain rule for locally absolutely continuous
functions gives
\[
 \frac{d}{d\theta}\Phi_a(y(\theta))
 =\frac{y'(\theta)}{a(y(\theta))}
 =w(\theta)-y(\theta)=e(\theta)+v(\theta)
\]
almost everywhere. Integrating on $[0,T]$, using $y(0)=0$, and then
letting $T\to\infty$ proves the signed identity in
\eqref{eq:nonlinear_balance}. Absolute integrability justifies the
limit on the integral side, and continuity of $\Phi_a$ gives
$\Phi_a(y(T))\to\Phi_a(1)$.
The identity $|f|=f+2\max\{-f,0\}$ proves the absolute-area balance.

Apply the algebraic decomposition
\[
 A_e^2+\alpha A_u^2
 =\frac{\alpha}{1+\alpha}(A_e+A_u)^2
  +\frac{(A_e-\alpha A_u)^2}{1+\alpha}
\]
and substitute $A_e+A_u=S_a+2\Ncal$. This gives
\eqref{eq:nonlinear_gap}. Since $S_a>0$, the two terms on its
right vanish exactly when $\Ncal=0$ and $A_e=\alpha A_u$.
Vanishing negative-part integrals give the two almost-everywhere
signs. For an equality response,
$y'=a(y)(e+v)\ge0$ almost everywhere; local absolute continuity
makes $y$ nondecreasing. Its initial and limiting values then give
$0\le y\le1$ at every time.

It remains to show attainment by an admissible response.
Under \eqref{eq:nonlinear_feedback},
\begin{equation}
 y'=c_\alpha a(y)(1-y),\qquad
 e'=-c_\alpha a(1-e)e,
 \qquad e(0)=1.
 \label{eq:nonlinear_closed_loop}
\end{equation}
Local Lipschitz continuity gives local existence and uniqueness.
Starting at zero, $y$ increases while $y<1$ and cannot cross the
equilibrium $y=1$. Hence it stays in the compact interval $[0,1]$,
and the solution extends to every $\theta\ge0$. On this interval,
$0<a_{\min}\le a(y)\le a_{\max}<\infty$. Solving the scalar
equation for $e$ gives
\[
 e(\theta)=\exp\!\left[-c_\alpha
                   \int_0^\theta a(y(s))\dd s\right],
\]
which proves \eqref{eq:nonlinear_decay}, strict positivity at finite
times, and exponential convergence. Thus $e,v=e/\alpha$ belong
to $L^1$ and are nonnegative. The command has a finite initial
jump from its zero prehistory to $1+1/\alpha$ and then decreases
to one. It is an ordinary bounded function.

The pointwise ratio $v=e/\alpha$ gives $A_u=A_e/\alpha$.
Combining it with the balance yields the areas in
\eqref{eq:nonlinear_areas_peak}; the peak follows from $e(0)=1$
and $e$ decreasing to zero. The equality conditions are therefore
attained, establishing \eqref{eq:nonlinear_minimum} globally over
the declared admissible class.
\end{proof}

The feedback law contains the equilibrium feedforward term one.
It is not being identified with a positive PI controller acting only
on the tracking error. Its formula is independent of $a$, so the same
law attains the corresponding bound for every positive locally
Lipschitz $a$ in \eqref{eq:nonlinear_plant}. The minimum value and
the transient still depend on $a$ through $S_a$ and
\eqref{eq:nonlinear_closed_loop}. This independence concerns the
specified plant class; it does not cover an incorrect equilibrium
input, an added delay, or additional states. Nor does the theorem
assert uniqueness of the minimizing trajectory.

\subsection{A state-dependent speed with an exact area calculation}

Take
\begin{equation}
 a(y)=1+y^2,
 \qquad S_a=\arctan(1)=\frac\pi4.
 \label{eq:nonlinear_example}
\end{equation}
At $\alpha=1$, the attaining law is $w=2-y$, and
\[
 y'=2(1+y^2)(1-y),\qquad
 A_e=A_u=\frac\pi8,\qquad
 \Mcal_1=\frac{\pi}{4\sqrt2}.
\]
The areas have been evaluated without integrating a numerical
trajectory. For any $\alpha>0$, separation of variables also gives
an exact implicit expression for the output:
\begin{equation}
 c_\alpha\theta
 =-\frac12\log(1-y)
   +\frac14\log(1+y^2)+\frac12\arctan y,
 \qquad 0\le y<1.
 \label{eq:nonlinear_implicit}
\end{equation}
The derivative of the right side with respect to $y$ is
$[(1+y^2)(1-y)]^{-1}$, which verifies the separated equation
and the initial value.

For comparison, $a\equiv1$ under the same feedback has
$e=e^{-c_\alpha\theta}$ and $S_a=1$. At $\alpha=1$, its two
areas are $1/2$ rather than $\pi/8$, although both commands start
at two. The nonlinear plant moves faster as $y$ increases, and the
command returns to equilibrium sooner. In the area plane the
supporting line changes from $A_e+A_u=1$ to $A_e+A_u=\pi/4$;
the optimal ratio remains $A_e/A_u=\alpha$.
Figure~\ref{fig:nonlinear_miace} illustrates this comparison.

\begin{figure}[H]
 \centering
 \includegraphics[width=.96\linewidth]{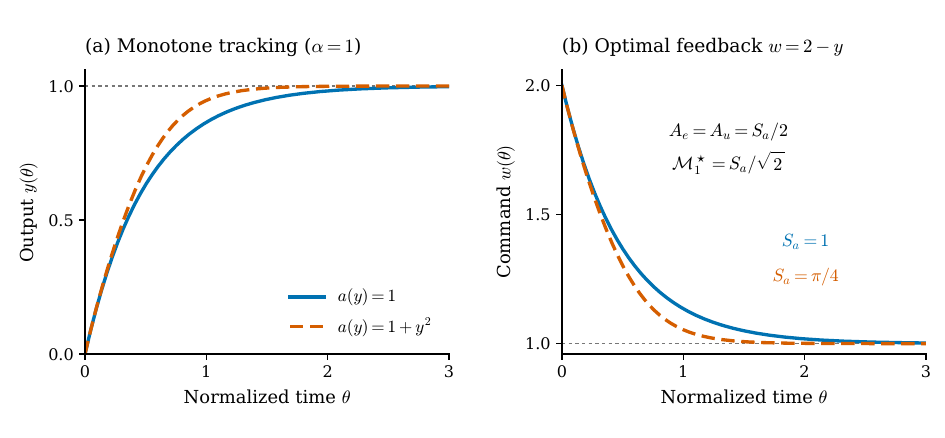}
 \caption{The exact nonlinear extension at $\alpha=1$. The plants
 $a(y)=1$ and $a(y)=1+y^2$ use the same feedback $w=2-y$.
 The panels show the monotone output $y$ and the command $w$.
 Both plants have $v=e$ and the same initial command peak, but
 different areas: $A_e=A_u=1/2$ for the linear lag and
 $A_e=A_u=\pi/8$ for the nonlinear plant.
 The displayed trajectories illustrate the exact
 balance; their sampled values are not used to establish the optimum.}

 \label{fig:nonlinear_miace}
\end{figure}

Here the nonlinear storage coordinate replaces the linear plant's
integrated state equation, while the MIACE criterion and its
negative-area penalty remain unchanged. The result extends the
first-order balance and attainment argument to
\eqref{eq:nonlinear_plant}. The delayed PI, filtered PID, and
positive-cascade theorems retain their original hypotheses; no
claim about arbitrary nonlinear plants with delay follows from
this calculation.

%% file: sections/09_discussion.tex
\section{Interpretation and scope}
\label{sec:discussion}

For the delay-free plant, the MIACE value gives explicit bounds on overshoot and total backward motion. The exact
identity \eqref{eq:miace_shape_gap} separates the cost of negative areas
from the imbalance between tracking and effort. Through $y'=e+v$, the
same areas bound every reversal of the output. Thus the criterion controls
both overshoot and motion away from the target before it is reached.
The inequality holds for any admissible command of the first-order plant;
PI is needed to identify its unique two-parameter minimizer, not to establish
the inequality. Total variation is an existing response-shape measure
\citep{HubaEtAl2021}; the calculation here connects it to the attained
MIACE value and retains the endpoint correction for finite records.

For $L=0$, the controller $x=r=1$ provides a physical reference for both fronts.
It applies the command needed at equilibrium and lets the plant respond
with its own time constant. Raising $x$ along $r=1$ replaces that time
constant by $\tau/x$. The price is an initial command $x/k$ followed by
an exponential return to $1/k$. The squared-energy front charges heavily
for a large excursion, while the absolute-area front counts its area.
In both cases the nominal efficient family is the same, although the
weights selecting a member differ.

The second parameter, $r$, moves the controller zero relative to the
plant pole. Without constraints, this change cannot improve both
objectives. For the quadratic pair,
\eqref{eq:same_tracking_controller} constructs a cancellation controller
with the same tracking energy and lower effort whenever $r\ne1$.
For the absolute pair, \eqref{eq:negative_area_gap} gives a physical
measure of the departure from the efficient line: the excess is twice
the total area of output overshoot and command deficit below equilibrium.
These comparisons concern every positive PI controller, including
oscillatory and repeated-pole cases.

An amplitude constraint preserves the clipped solution when overshoot
is forbidden. The complete domain in \cref{thm:monotone} explains
why: its additional $x<1$ branch is dominated by the feasible
equilibrium-command controller. Once overshoot is allowed, the exact
counterexample improves on clipping by using a faster integral action.
The numerical search with a post-jump-rate constraint explores a still
different admissible set. Its agreement across starts establishes
reproducibility of the reported candidates, rather than a missing proof
of global constrained optimality.

The mismatch calculation separates stability, cost, and response shape.
All positive gains and time constants in the same first-order plant class
preserve stability under the fixed positive PI controller. Cost loss is
given exactly by \eqref{eq:mismatch_loss} and begins at second order
near the nominal model. Monotonicity has a one-sided boundary instead:
when the normalized proportional gain remains at least one, any actual
time constant larger than the nominal integral time produces some
overshoot. An engineering requirement on overshoot must therefore be
checked even when the calculated cost penalty is small. The MIACE shape
bound is consistent with this distinction: a positive optimality gap permits
some reverse motion, but places an explicit limit on its total size when
the criterion and the lower bound refer to the same actual model.

With a pure delay, the same argument starts after the waiting interval.
The extra error area is exactly $\delta=L/\tau$, and the bound on output
travel uses $A_e+A_u-\delta$. Thus the delay changes both the reference
cost and the way a relative cost excess is converted into an overshoot
bound. It does not remove the penalty for negative areas. What becomes
more restrictive is attainment: the positive PI family reaches the lower
bound exactly on the range in \cref{thm:delayed_optimum}. Within that
range the gain is reduced by $1/(1+\delta)$, the integral time remains
$\tau$, and the output is monotone. The command's initial linear rise
compensates for its lower starting value, leaving the same peak as in
the delay-free optimum.

The endpoint of this range is not a general boundary between good and
bad delayed tuning. Beyond it, the analytical lower bound remains valid,
but no positive ideal PI controller can attain it. A controller may still
have a monotone output while its centered command changes sign. Determining
the actual delayed PI optimum in that region requires a further analysis;
neither the delay-free front nor the cancellation formula supplies it.
The delayed quadratic comparison reinforces the distinction: its old gain
is not even stationary after a nonzero delay is introduced.

The filtered PID constructions show that the PI attainment boundary is
a restriction of that controller family. At $\alpha=1$, $\delta=1/4$,
a genuine filtered PID reaches the lower bound with a monotone output,
although no positive PI reaches it. For $\alpha>\mathrm e$, the explicit
family in \cref{thm:pid_explicit_family} supplies a further interval above
$1/\alpha$. The derivative and its filter reshape the command during the
waiting interval; their benefit is established by the sign of $v$ as
well as the output shape. At $\alpha=4$, $\delta=1/3$, both the PID and
the extrapolated PI can have monotone outputs, but only the constructed
PID keeps the command above its equilibrium value and balances the areas.

These results also separate uniqueness of a controller from uniqueness
of the optimal objective pair. The PI equality analysis fixes both PI
parameters. For filtered PID it fixes the integral gain and a relation
between the remaining coefficients, while a continuum of settings can
satisfy the sign conditions. Their two areas and MIACE values coincide,
but their waveforms need not. A secondary requirement on command amplitude,
noise sensitivity, or robustness could distinguish them; no such secondary
optimum follows from the area criterion alone.

The filter realization matters. The present proof uses a derivative of
the full tracking error with an independently chosen filter time. It does
not carry over without calculation to a derivative of the measured output,
a fixed filter ratio, sampled implementation, or a saturated actuator.
The necessary cancellation also concerns the nominal first-order pole.
The delay-free mismatch calculation does not establish robustness
of these delayed PID designs. Positivity and exponential stability are
proved here for the specified exact plant and known delay.

The cascade result separates three questions that coincide for the
first-order plant. The low-frequency moment determines the area bound;
the cascade states make equality imply monotonicity; the interpolation
and positivity conditions establish that a controller can actually
attain the bound. None of these steps can replace either of the others.
The strict necessary condition $\alpha>\mu+\delta-1$ shows that equal
weights are already unsuitable for bound attainment in two or more
identical lags. This does not prove overshoot of the actual optimum
when the bound is unattainable. It leaves that different optimization
problem open.

For arbitrary positive lag constants, the existence interval depends
on those constants and on the weight. For identical lags, the explicit
choice $\alpha=2N-1$ ensures attainment at every order throughout the common delay interval
$0\le L/\tau\le1/4$. This is a sufficient interval, not the largest
one. The construction has minimal realization order exactly $N$ in
the stated families; it does not show that no controller of lower
order could reach the same objective pair. Such order limitations
belong to a separate question, also addressed for rational monotonic
tracking in \citet{Taghavian2026}.

The resulting output has the regularity supplied by $N$ plant lags:
its first $N-1$ derivatives are continuous through the end of the
waiting interval. This gives a precise meaning to a smooth transient.
It does not impose a physical slew limit on the command, whose
initial finite jump is allowed by the proper controller. The command
remains above equilibrium but need not decrease monotonically.
The nominal plant factors cancelled by the controller are stable;
the proof retains them in the full characteristic determinant.
Uncertain cancellation and delay require a separate robustness analysis.

The distinction between exact and near attainment also matters.
Every attaining positive-cascade response is monotone, but the
quantitative first-order reversal bound extends to higher order here
only when the additional balance operator $Z=pH/(1-H)$ has a
nonnegative impulse measure. Even three identical lags fail that
operator condition. Conversely, the exact third-order minimum-phase
counterexample shows that nonnegative error and command do not
generally eliminate backward motion. The absence of overshoot,
monotonicity, and small MIACE excess are related properties under
stated dynamics, rather than interchangeable descriptions of a trace.

The scope of the quadratic optimality result is also precise. Its
signal-level completion of squares is classical LQR, not a separate
general optimal-control principle. The explicit PI parameter identities,
the complete efficient controller sets, their absolute-area equality
conditions, and the full nonovershooting amplitude problem supply the
controller-specific content. The mismatch corollary is a consequence of
those identities. The quadratic stationary formula remains attributed to the dissertation.
MIACE is traced to the author's 2003 diploma work and PIDBul2003, with
the 2011 dissertation providing the later explicit finite-horizon formula.
This chronology is distinct from a claim that the present global and
shape theorems had already been established in the early software.

The internal-model extension resolves feasibility in a larger controller
class. Its exact condition $\alpha>S-1$ agrees with the previous
necessary inequality, but its realization contains delay blocks.
The rational constructions still answer the separate finite-order
question. The exponential moment also separates the command's area
from its peak: MIACE equality leaves many waveforms possible, while
the sharp amplitude bound selects a unique rectangular excess command
at the least peak. This command has finite jumps, so the cap theorem
does not impose a slew limit or describe a saturated implementation.
The delay-error estimate gives a sufficient stability condition and an explicit cost bound,
based on conservative convolution estimates. It does not preserve exact
attainment or nominal signal signs under perturbation.

The nonlinear result in \cref{thm:nonlinear_miace} identifies an
integrated storage balance as another source of the same area geometry.
The transformation is used to integrate the plant equation; the error
inside MIACE remains $1-y$. The attaining feedback does not depend
on the positive function $a(y)$, although the attained value and the
time history do. This gives a common optimal law within that nonlinear
class, with unit equilibrium gain and zero delay. It does not establish
optimality under an altered equilibrium relation, added delay, actuator
saturation, or arbitrary nonlinear dynamics. Nor does an optimal pair
of areas imply a unique minimizing trajectory.

The assertion that every positive PI setting is stable concerns the
delay-free plant only. The pure-delay extension uses an exactly known
delay and does not establish a robustness margin for delay uncertainty.
The rational positive-cascade extension is also nominal; the internal-model
delay-error result is stated separately in \cref{thm:complete_robust}.
Neither a hard physical slew limit
through the initial jump nor actuator saturation is represented by the
stated PI and filtered PID equations. The analysis does not establish a disturbance-
rejection optimum, anti-windup design, or hardware performance.
Within its stated domain, it supplies an exact benchmark: a numerical
tuning method can be checked against both the optimum and the gap
of every returned controller.

%% file: sections/10_conclusion.tex
\section{Conclusion}
\label{sec:conclusion}

MIACE admits an exact connection between its value, the signs of its
two signals, and the shape of the resulting response. For a stable
unit-gain plant with pure delay, the signed tracking-error and centered
command areas sum to $S=\mu+\delta$, where $\mu=-H'(0)$. When $S>0$,
the lower bound is $S\sqrt{\alpha/(1+\alpha)}$. Its squared excess
separates into a penalty for negative areas and a penalty for an
incorrect ratio of the two absolute areas. Equality requires both
signals to be nonnegative and fixes their areas separately. This
statement is independent of the chosen controller family.

For the first-order plant, the physical balance also bounds total
output travel, accumulated backward motion, and peak overshoot.
Without delay, a relative MIACE excess of at most $\eta$ permits at
most $\eta/2$ of the step in either of the latter two quantities.
Every response attaining the bound is monotone. The finite-record
version retains the terminal error and describes only the observed
interval. A known delay changes the signed-area intercept and the
conversion from excess cost to possible reverse motion.

In the positive ideal PI family, equality identifies the unique
attaining controller. Without delay, $T_i=\tau$ and
$kK_c=1+1/\alpha$; with delay, the proportional gain is divided by
$1+L/\tau$ and attainment requires the exact range derived in
\cref{thm:delayed_optimum}. Filtered PID extends attainment beyond
that range. Its equality conditions fix the integral gain and force
a controller structure, while nonnegative convolution representations
provide sufficient conditions for stability and nonnegativity. The attaining
PID need not be unique, and the constructions do not characterize
the entire remaining delay domain.

For a positive cascade of arbitrary order, a separate state argument
proves that every attaining response is monotone. Attainment by an
ordinary command requires the strict inequality
$\alpha>\mu+\delta-1$. An exact finite-support remainder of the delay
then constructs a proper rational controller of the same order as
the plant. Its convolution resolvent proves internal exponential
stability and nonnegativity of both signals. Every fixed cascade
admits this construction at a suitable weight for all sufficiently
small delays. For $N$ identical lags, the explicit choice
$\alpha=2N-1$ works for every integer $N\ge2$ throughout
$0\le L/\tau\le1/4$, attaining
$(N+L/\tau)\sqrt{(2N-1)/(2N)}$.

The output is nondecreasing and has $N-1$ continuous derivatives;
the finite initial command jump is retained. The controller order
and the stated delay interval are achieved values, not proofs of
minimal necessary order or of the largest possible interval.
The additional positivity condition for a quantitative near-minimum
bound is stated separately. An exact minimum-phase counterexample
shows why equality alone does not ensure monotonicity outside the
positive-cascade class.

When internal delays are permitted, the exponential-moment condition
is both necessary and sufficient for a nonnegative centered command
to produce a monotone cascade output. The exact attainment condition
becomes $\alpha>S-1$, and a stable internal-model realization constructs
an attaining controller throughout this range. With a cap $1\le w\le M$,
the available command area is exactly
$(M-1)\log[M/(M-1)]$. A rectangular excess command attains its boundary
and gives the least possible peak among all MIACE-equality responses.
The capped sign-preserving optimum is explicit; global optimality
outside that class is asserted only when the universal bound is reached.
A sufficient small-gain estimate also bounds MIACE degradation under
a small error in the delay, while leaving shape preservation unclaimed.

The quadratic comparison selects the same efficient delay-free PI
family with a different weight mapping. Its classical LQR bound,
explicit parameter gap, complete front, actuator result, and mismatch
loss provide complementary checks. They retain their stated
first-order and delay-free qualifications.

For the nonlinear first-order class $y'=a(y)(w-y)$, with positive locally
Lipschitz $a$, a storage balance gives the intercept
$S_a=\int_0^1 a(s)^{-1}\dd s$ and the same
exact MIACE excess identity. The feedback $w=1+(1-y)/\alpha$ attains
the global bound with a finite command jump and a monotone output.
This construction requires no knowledge of the state-dependent response
rate, while the optimum value depends on that rate. The result retains
the original two absolute areas and is proved for zero delay.

The original criterion and its PIDBul2003 implementation arose in the
author's 2003 diploma work. The present analysis concerns the exact
infinite-horizon bound, its equality conditions, and constructions
that attain it. Monotonic tracking, external positivity, interpolation,
and convolution resolvents have their own established literature,
which is explicitly acknowledged. The contribution is the connection
of these tools to MIACE attainment for the specified dynamics.
General parameter uncertainty, nonlinear saturation, sampling, and
disturbance-rejection optimality remain outside the analysis. The new
delay-error result treats a fixed controller and unchanged lag
constants and gain; it is not a robust-optimality theorem.
A natural next step is to extend MIACE to broader classes of nonlinear, time-varying and multivariable systems. The central question is when a storage balance yields sharp bounds on the criterion and allows its minimizers to be characterized in terms of tracking and control effort. Pure delays, actuator constraints and model uncertainty introduce further questions of attainability and robustness. Establishing precise links between absolute signal areas, settling time and oscillatory behaviour would also help turn the original “minimum inertia” motivation into a more general theory of controller design.

%% file: declarations.tex
\section*{Acknowledgements and declarations}
\noindent\textbf{Funding.} This research received no external funding.
\noindent\textbf{Competing interests.} The author has no relevant financial
or non-financial interests to disclose.

\noindent\textbf{Author contributions.} Latchezar P. Tomov: conceptualization,
methodology, formal analysis, investigation, software, visualization,
writing, and validation.

\noindent\textbf{Use of generative AI.} OpenAI ChatGPT and Codex were used
during August--October 2026 for drafting and revision, assistance with
mathematical proofs and lean 4 formalization and Python code, computational checks, reference
searches, and language editing. The derivations and computational checks are
presented explicitly in the manuscript and accompanying material. The paper is based on original optimization criteria, developed by the author in 2003, on his hypothesis for the behaviour of certain systems optimized with that criteria and some partial results for quadratic optimization from 2011. The usage of AI is to verify and extend the ideas of the author. The author
is responsible for the submitted content and its scholarly claims.

\section*{Data and software availability}
The reproduction code, generated numerical data, verification scripts and
figure sources are provided in the ancillary archive
\texttt{MIACE\_Reproducibility.zip}.
The reported theoretical results require no external dataset.
The \href{https://github.com/luchesart/MIACE/blob/monotone-tracking/MIACE_Proofs_Companion.pdf}{proofs companion}
provides background for reading the proofs \citep{Tomov2026Companion}.
The full \href{https://github.com/luchesart/MIACE/tree/monotone-tracking}{Lean project and theorem--proof map}
are available in the same repository \citep{Tomov2026Lean}.

%% file: sections/appendices.tex
\appendix
\section{The second-order energy formula}
\label{app:h2}

A realization of $H(p)=(dp+e_0)/(ap^2+bp+c)$ is
\[
 A=\begin{bmatrix}0&1\\-c/a&-b/a\end{bmatrix},\quad
 B=\begin{bmatrix}0\\1/a\end{bmatrix},\quad
 C=\begin{bmatrix}e_0&d\end{bmatrix}.
\]
The controllability Gramian $P=\int_0^\infty e^{At}BB^\top e^{A^\top t}\dd t$
solves $AP+PA^\top+BB^\top=0$ because $A$ is Hurwitz. Writing its entries
as $P_{11},P_{12},P_{22}$, the three scalar equations are
\[
 2P_{12}=0,\qquad P_{22}-\frac caP_{11}-\frac baP_{12}=0,\qquad
 -\frac{2c}aP_{12}-\frac{2b}aP_{22}+\frac1{a^2}=0.
\]
They give $P_{12}=0$, $P_{22}=1/(2ab)$, and $P_{11}=1/(2bc)$.
Therefore $CPC^\top=e_0^2/(2bc)+d^2/(2ab)$, proving
\eqref{eq:h2_lemma}, including repeated poles.

\section{Equality in the two absolute-area bounds}
\label{app:miace_uniqueness}

Suppose $0<z=r/x<1$ and both signed-area bounds are attained.
Continuity then gives $e\ge0$ and $v\ge0$. Since $v(0^+)=x-1$,
we must have $x\ge1$. The real decay rates, when they exist, solve
\[
 f(\lambda)=\lambda^2-(1+x)\lambda+x/r=0.
\]

If $r>1$, then $f(1)=x(1/r-1)<0$, so
$0<\lambda_1<1<\lambda_2$. The positive residues in
\eqref{eq:positive_residues} apply. Since $v=-e'-e$, its slower exponential
has coefficient $a_1(\lambda_1-1)<0$. That term eventually dominates,
contradicting $v\ge0$.

If $r<1$ and the roots are complex, $e$ is a decaying nonzero sinusoid and
changes sign. If the roots are distinct and real, $f(1)>0$ and
$\lambda_1+\lambda_2=1+x\ge2$ imply $1<\lambda_1<\lambda_2$.
The slower error residue is
$(1-\lambda_1)/(\lambda_2-\lambda_1)<0$, so $e$ is eventually negative.
For a repeated root $\lambda=(1+x)/2$, the initial conditions give
\[
 e(\theta)=e^{-\lambda\theta}\left[1+\frac{1-x}{2}\theta\right].
\]
Here $r=4x/(1+x)^2<1$ and $x\ge1$ imply $x>1$; the bracket eventually
becomes negative. Thus $r<1$ is also impossible.

Consequently $r=1$, and $e=e^{-x\theta}$,
$v=(x-1)e^{-x\theta}$. The prescribed signed area then fixes $x=1/z$.
This proves uniqueness at every interior point of the $L^1$ front, not only
at one selected weight.

\section{Extrema, absolute areas, and numerical checks}
\label{app:peak_oracle}

In parallel coordinates, $e$, $w-1$, and $w'$ satisfy
$f''+(1+p_c)f'+i_cf=0$. For two distinct real poles, each is a sum of two
exponentials and has at most one interior stationary point. A repeated pole
gives $e^{\rho\theta}(a+b\theta)$. For complex poles, each is a damped
sinusoid: successive stationary values alternate in sign and their
magnitudes decrease geometrically. Evaluating the initial and limiting
values and the first extrema of both signs therefore gives the global
extrema in exact arithmetic. The implementation evaluates these formulas
in floating-point arithmetic and checks them separately against an ODE
solver; it does not provide rigorous interval bounds.

The candidate times can be written explicitly. For real distinct decay
rates, let $f=c_1e^{-\lambda_1\theta}+c_2e^{-\lambda_2\theta}$,
$0<\lambda_1<\lambda_2$. An interior stationary point exists only if
$-\lambda_2c_2/(\lambda_1c_1)>0$ and
\[
 \theta_*=\frac{\log[-\lambda_2c_2/(\lambda_1c_1)]}
                  {\lambda_2-\lambda_1}>0.
\]
Zero coefficients are handled as single exponentials. A possible zero of
$f$ is obtained by replacing that ratio by $-c_2/c_1$. For
$f=e^{-\lambda\theta}(a+b\theta)$, its only possible interior stationary
point is $\theta_*=(b-\lambda a)/(\lambda b)$ when $b\ne0$ and that
value is positive. Its possible zero is $-a/b$.

For $f=e^{-d\theta}(a\cos\omega\theta+b\sin\omega\theta)$, write
$f'=e^{-d\theta}(C\cos\omega\theta+D\sin\omega\theta)$, with
$C=-da+\omega b$, $D=-db-\omega a$. If $(C,D)\ne(0,0)$ and
$\delta=\operatorname{atan2}(D,C)$, all stationary times are
$\theta_j=(\delta+\pi/2+j\pi)/\omega$; retain nonnegative times.
Successive values change sign and shrink by $e^{-d\pi/\omega}$.
Thus the first extrema of both signs, the initial value, and the limit
zero suffice. Applying the construction to $v=w-1$ and $v'=w'$
gives the command and regular-rate bounds.

For the absolute integrals, the real-pole cases are integrated on the
intervals separated by their possible zero. In the complex-pole case,
$e^{-d\theta}(a\cos\omega\theta+b\sin\omega\theta)$ has zeros separated
by $\pi/\omega$. Integrate the first partial lobe, integrate the next full
lobe, and sum the remaining absolute lobe areas as a geometric series with
ratio $e^{-d\pi/\omega}$. This evaluates the infinite tail rather than
truncating it at a simulation horizon.
For direct implementation, an antiderivative is
\[
 \frac{e^{-d\theta}}{d^2+\omega^2}
 \bigl[(-da-\omega b)\cos\omega\theta
       +(\omega a-db)\sin\omega\theta\bigr].
\]
At a zero crossing, integrate each lobe with its own sign before summing
absolute areas. Repeated or nearly repeated real poles need their limiting
form to avoid subtracting large residues; the ODE checks cover this case.

The separate state-space check uses the centered integral contribution
$h=(x/r)\int_0^\theta e\dd\eta-1$. Its equations are
\[
 \begin{bmatrix}e\\h\end{bmatrix}'=
 \begin{bmatrix}-(1+x)&-1\\x/r&0\end{bmatrix}
 \begin{bmatrix}e\\h\end{bmatrix},\qquad
 \begin{bmatrix}e(0)\\h(0)\end{bmatrix}=
 \begin{bmatrix}1\\-1\end{bmatrix},\qquad v=xe+h.
\]
Observability Gramians in this realization check the physical loop without
reusing the transfer-function realization of \cref{app:h2}. Additional
ODE quadrature checks cover absolute areas, nearly repeated poles, and
non-unit reference steps. A counterexample check evaluates
\eqref{eq:clipping_counterexample}. The script
\texttt{python code/reproduce\_package.py}, run from the source-project root,
rebuilds the figures, numerical designs, and all checks. Seeds, tolerances,
and full-precision results are supplied in the accompanying supplementary
archive.